\documentclass[a4paper,twoside]{article}
\usepackage{a4}
\usepackage{amssymb}
\usepackage{amsmath}
\usepackage{upref}
\usepackage[english]{babel}
\usepackage{bbm}
\usepackage[active]{srcltx}
\usepackage[dvips,colorlinks,citecolor=blue,linkcolor=blue]{hyperref}
\usepackage[dvipsnames]{color}
\makeatletter
\def\sectionmark#1{} 
\def\subsectionmark#1{}
\newcommand{\sectnr}{\ifnum \c@secnumdepth >\z@
                 \thesection.\hskip 1em\relax \fi}
\def\@evenhead{\footnotesize\rm\thepage\hfil\leftmark\hfil}
\def\@oddhead{\footnotesize\rm\hfil\rightmark\hfil\thepage}
\def\tableofcontents{\section*{Contents} 
 \@starttoc{toc}}
\makeatother
\makeatletter
\def\@biblabel#1{#1.}
\makeatother
\makeatletter
\let\Thebibliography=\thebibliography
\renewcommand{\thebibliography}[1]{\def\@mkboth##1##2{}\Thebibliography{#1}
\addcontentsline{toc}{section}{References}
\frenchspacing 
\setlength{\@topsep}{0pt}
\setlength{\itemsep}{0pt}%
\setlength{\parskip}{0pt plus 2pt}%
}
\makeatother
\makeatletter
\def\mdots@{\mathinner.\nonscript\!.%
 \ifx\next,.\else\ifx\next;.\else\ifx\next..\else
 \nonscript\!\mathinner.\fi\fi\fi}
\let\ldots\mdots@
\let\cdots\mdots@
\let\dotso\mdots@
\let\dotsb\mdots@
\let\dotsm\mdots@
\let\dotsc\mdots@
\def\vdots{\vbox{\baselineskip2.8\p@ \lineskiplimit\z@
    \kern6\p@\hbox{.}\hbox{.}\hbox{.}\kern3\p@}}
\def\ddots{\mathinner{\mkern1mu\raise8.6\p@\vbox{\kern7\p@\hbox{.}}%
    \raise5.8\p@\hbox{.}\raise3\p@\hbox{.}\mkern1mu}}
\makeatother
\makeatletter
\let\Enumerate=\enumerate
\renewcommand{\enumerate}{\Enumerate%
\setlength{\@topsep}{0pt}
\setlength{\itemsep}{0pt}%
\setlength{\parskip}{0pt plus 1pt}%
\renewcommand{\theenumi}{\textup{(\alph{enumi})}}%
\renewcommand{\labelenumi}{\theenumi}%
}
\let\endEnumerate=\endenumerate
\renewcommand{\endenumerate}{\endEnumerate\unskip}
\makeatother
\makeatletter
\def\@seccntformat#1{\csname the#1\endcsname.\quad}
\makeatother
\newcommand{\authortitle}[2]{\author{#1}\title{#2}\markboth{#1}{#2}}
\newcommand{\auth}[2]{{#1, #2.}}
\newcommand{\art}[6]{{\sc #1, \rm #2, \it #3 \bf #4 \rm (#5), \mbox{#6}.}}
\newcommand{\artin}[3]{{\sc #1, \rm #2, in #3.}}
\newcommand{\artprep}[3]{{\sc #1, \rm #2, #3.}}
\newcommand{\book}[3]{{\sc #1, \it #2, \rm #3.}}
\newcommand{\AND}{{\rm and }}
\RequirePackage{amsthm}
\newtheoremstyle{descriptive}%
  {\topsep}   
  {\topsep}   
  {\rmfamily} 
  {}          
  {\bfseries} 
  {.}         
  { }         
  {}          
\newtheoremstyle{propositional}%
  {\topsep}   
  {\topsep}   
  {\itshape}  
  {}          
  {\bfseries} 
  {.}         
  { }         
  {}          
\newtheoremstyle{remarkstyle}%
  {\topsep}   
  {\topsep}   
  {\rmfamily}  
  {}          
  {\itshape} 
  {.}         
  { }         
  {}          
\theoremstyle{propositional}
\newtheorem{thm}{Theorem}[section]
\newtheorem{prop}[thm]{Proposition}
\newtheorem{lem}[thm]{Lemma}
\newtheorem{cor}[thm]{Corollary}
\theoremstyle{descriptive}
\newtheorem{deff}[thm]{Definition}
\newtheorem{example}[thm]{Example}
\newtheorem{remark}[thm]{Remark}
\makeatletter
\renewenvironment{proof}[1][\proofname]{\par
  \pushQED{\qed}%
  \normalfont 
  \trivlist
  \item[\hskip\labelsep
        \itshape
    #1\@addpunct{.}]\ignorespaces
}{%
  \popQED\endtrivlist\@endpefalse
}
\makeatother
\newdimen\extrawidth
\def\iintlim#1#2{\setbox0\hbox{$\scriptstyle#1$}%
        \setbox1\hbox{$\scriptstyle#2$}%
        \extrawidth=\wd1 \advance\extrawidth-\wd0
        \ifdim\extrawidth<0pt \extrawidth=0pt\fi%
        \int_{#1\kern\extrawidth \kern .5em}^{#2\kern -\wd1} \kern -.5em%
}
\newcommand{\cprime}{{\mathsurround0pt$'$}}
\newcommand{\setm}{\setminus}
\newcommand{\simge}{\gtrsim}
\newcommand{\simle}{\lesssim}
\newcommand{\longhookrightarrow}{\lhook\joinrel\longrightarrow}
\def\vint{\mathop{\mathchoice%
          {\setbox0\hbox{$\displaystyle\intop$}\kern 0.22\wd0%
           \vcenter{\hrule width 0.6\wd0}\kern -0.82\wd0}%
          {\setbox0\hbox{$\textstyle\intop$}\kern 0.2\wd0%
           \vcenter{\hrule width 0.6\wd0}\kern -0.8\wd0}%
          {\setbox0\hbox{$\scriptstyle\intop$}\kern 0.2\wd0%
           \vcenter{\hrule width 0.6\wd0}\kern -0.8\wd0}%
          {\setbox0\hbox{$\scriptscriptstyle\intop$}\kern 0.2\wd0%
           \vcenter{\hrule width 0.6\wd0}\kern -0.8\wd0}}%
          \mathopen{}\int}
\DeclareMathOperator{\diam}{diam}
\DeclareMathOperator{\dist}{dist}
\DeclareMathOperator{\Tr}{Tr}
\DeclareMathOperator{\Ext}{Ext}
\newcommand{\bdry}{\partial}
\newcommand{\bdy}{\bdry}
\newcommand{\loc}{_{\rm loc}}
\newcommand{\tB}{\widetilde{B}}
\newcommand{\bB}{B_{\itoverline{X}}}
\newcommand{\mx}{\mu} 
\DeclareMathOperator*{\proj}{proj}
\newcommand{\pip}{\varphi}
{\catcode`p =12 \catcode`t =12 \gdef\eeaa#1pt{#1}}      
\def\accentadjtext#1{\setbox0\hbox{$#1$}\kern   
                \expandafter\eeaa\the\fontdimen1\textfont1 \ht0 }
\def\accentadjscript#1{\setbox0\hbox{$#1$}\kern 
                \expandafter\eeaa\the\fontdimen1\scriptfont1 \ht0 }
\def\accentadjscriptscript#1{\setbox0\hbox{$#1$}\kern   
                \expandafter\eeaa\the\fontdimen1\scriptscriptfont1 \ht0 }
\def\accentadjtextback#1{\setbox0\hbox{$#1$}\kern       
                -\expandafter\eeaa\the\fontdimen1\textfont1 \ht0 }
\def\accentadjscriptback#1{\setbox0\hbox{$#1$}\kern     
                -\expandafter\eeaa\the\fontdimen1\scriptfont1 \ht0 }
\def\accentadjscriptscriptback#1{\setbox0\hbox{$#1$}\kern 
                -\expandafter\eeaa\the\fontdimen1\scriptscriptfont1 \ht0 }
\def\itoverline#1{{\mathsurround0pt\mathchoice
        {\rlap{$\accentadjtext{\displaystyle #1}
                \accentadjtext{\vrule height1.593pt}
                \overline{\phantom{\displaystyle #1}
                \accentadjtextback{\displaystyle #1}}$}{#1}}
        {\rlap{$\accentadjtext{\textstyle #1}
                \accentadjtext{\vrule height1.593pt}
                \overline{\phantom{\textstyle #1}
                \accentadjtextback{\textstyle #1}}$}{#1}}
        {\rlap{$\accentadjscript{\scriptstyle #1}
                \accentadjscript{\vrule height1.593pt}
                \overline{\phantom{\scriptstyle #1}
                \accentadjscriptback{\scriptstyle #1}}$}{#1}}
        {\rlap{$\accentadjscriptscript{\scriptscriptstyle #1}
                \accentadjscriptscript{\vrule height1.593pt}
                \overline{\phantom{\scriptscriptstyle #1}
                \accentadjscriptscriptback{\scriptscriptstyle #1}}$}{#1}}}}
\newcommand{\al}{\alpha}
\newcommand{\alp}{\alpha}
\newcommand{\be}{\beta}
\newcommand{\bw}{\beta}
\newcommand{\tw}{\varepsilon} 
\newcommand{\ga}{\gamma}
\newcommand{\dmu}{d\mx}
\newcommand{\eps}{\varepsilon}
\newcommand{\la}{\lambda}
\newcommand{\La}{\Lambda}
\newcommand{\om}{\omega}
\newcommand{\Om}{\Omega}
\renewcommand{\phi}{\varphi}
\newcommand{\p}{{$p\mspace{1mu}$}}   
\newcommand{\R}{\mathbb{R}}
\newcommand{\N}{\mathbb{N}}
\newcommand{\Z}{\mathbb{Z}}
\newcommand{\xh}{\hat{x}}
\newcommand{\yh}{\hat{y}}
\newcommand{\dbdX}{d_{X}}
\newcommand{\dbdY}{d_{Y}}
\newcommand{\Np}{N^{1,p}}
\newcommand{\gut}{g_{\tilde{u}}}
\newcommand{\ut}{\tilde{u}}
\newcommand{\Bt}{\widetilde{B}}
\newcommand{\ft}{\tilde{f}}
\newcommand{\xhat}{\hat{x}}
\newcommand{\yhat}{\hat{y}}
\newcommand{\zhat}{\hat{z}}
\newcommand{\s}{\sigma}
\newcommand{\ka}{\kappa}
\newcommand{\z}{\zeta}
\newcommand{\Bni}{{B_{n,i}}}
\newcommand{\Bzr}{{B(\z,r)}}
\newcommand{\Bzn}{{B(\z,r_n)}}
\newcommand{\Bxin}{{B(\xi,r_n)}}
\newcommand{\Bppal}{{B^\theta_{p,p}}}

\numberwithin{equation}{section}
\newenvironment{ack}{\medskip{\it Acknowledgement.}}{}

\begin{document}

\authortitle{Anders Bj\"orn, Jana Bj\"orn, James T. Gill,
	and Nageswari Shanmugalingam}
{Geometric analysis on Cantor sets and trees}
\author{
Anders Bj\"orn \\
\it\small Department of Mathematics, Link\"opings universitet, \\
\it\small SE-581 83 Link\"oping, Sweden\/{\rm ;}
\it \small anders.bjorn@liu.se
\\
\\
Jana Bj\"orn \\
\it\small Department of Mathematics, Link\"opings universitet, \\
\it\small SE-581 83 Link\"oping, Sweden\/{\rm ;}
\it \small jana.bjorn@liu.se
\\
\\
James T. Gill 
\\
\it \small Department of Mathematics and Computer Science, Saint Louis University, \\
\it \small 220 N. Grand Blvd, St. Louis, MO 63103, U.S.A.\/{\rm ;}
\it \small jgill5@slu.edu
\\
\\
Nageswari Shanmugalingam 
\\
\it \small  Department of Mathematical Sciences, University of Cincinnati, \\
\it \small  P.O.\ Box 210025, Cincinnati, OH 45221-0025, U.S.A.\/{\rm ;}
\it \small shanmun@uc.edu
}

\date{}
\maketitle

\noindent{\small
{\bf Abstract.} 
Using uniformization, Cantor type sets can be regarded as boundaries
of rooted trees.
In this setting, we show that the trace of a first-order Sobolev space 
on the boundary of a regular rooted tree 
is exactly a Besov space with an explicit smoothness exponent.
Further, we study quasisymmetries between the boundaries of two trees, 
and show that they have rough quasiisometric extensions to the trees. 
Conversely, we show that
every rough quasiisometry between two trees
extends as a quasisymmetry between their
boundaries. 
In both directions we give sharp estimates for the involved constants.
We use this to obtain quasisymmetric 
invariance of certain Besov spaces 
of functions on Cantor type sets.
}

\bigskip
\noindent
{\small \emph{Key words and phrases}: 
Besov space, 
Cantor set, 
doubling measure,
extension, 
metric space,
Newtonian space,
Poincar\'e inequality,
quasisymmetry, 
rough quasiisometry, 
Sobolev space,
totally disconnected,
trace, 
tree, 
ultrametric,
uniformly perfect,
visual boundary.
}

\medskip
\noindent
{\small Mathematics Subject Classification (2010): 
Primary: 30L10;
Secondary:, 30L05, 31E05, 46E35, 51M10.
}

\section{Introduction} 

Much of the recent development of analysis in metric measure spaces has
tended to focus on two types
of metric measure spaces: those that are highly connected (whose measures are doubling
and support a Poincar\'e inequality, see for example~\cite{BBbook}, \cite{BBS}, \cite{BBS4},
\cite{BS}, \cite{BP00}, \cite{BP03}, \cite{GKS}, \cite{HaKo}, \cite{Hei}, 
\cite{HeiKo}, \cite{Kort}, \cite{Sh-rev}, and the references
therein), and those that are fractals with a minimal connectedness
property (the so-called post-critically finite fractals such as the Sierpi\'nski gasket, see for example~\cite{DSV}, \cite{HPS}, \cite{Str}, \cite{SW},
and the references therein). Totally disconnected
sets such as Cantor sets tend to fall outside of both these categories. 
Papers such as Bellissard--Pearson~\cite{BePe} and Kigami~\cite{Ki} 
have studied analysis on Cantor type metric spaces, but
only from the point of view of linear theory ($p=2$). 
The goal of this paper is to study such totally disconnected sets from the point of
view of nonlinear analysis on metric measure spaces,
with emphasis on function spaces and on quasisymmetric mappings 
between such sets.

The setting considered here is that of uniformly perfect totally 
disconnected metric measure spaces, including various
types of Cantor sets. As explained in Semmes~\cite{S11b},~\cite{S11a} such spaces 
are, up to biLipschitz equivalence, 
also obtained as ultrametric spaces which are boundaries 
of rooted regular trees equipped with a weighted metric 
(called uniformization metric in~\cite{BHK}).
This point of view is similar to the uniformization process considered by 
Bonk--Heinonen--Koskela~\cite{BHK},
and to the hyperbolic buildings, 
obtained by pasting together hyperbolic regions in a combinatorial way,
studied by e.g.\ Bourdon--Pajot~\cite{BP00},~\cite{BP03}.
However, while the Gromov boundaries in~\cite{BHK} as well as
boundaries of hyperbolic buildings considered in~\cite{BP00}
and \cite{BP03} are highly connected, the boundary Cantor sets 
considered in this paper are totally disconnected. 
This in particular means that the Besov spaces considered below are nontrivial 
for all smoothness exponents $\theta>0$, in contrast to e.g.\ 
Th\'eor\`eme~0.3 in~\cite{BP03}. 
Note also that, because of the essentially one-dimensional structure of the trees,
our setting does not fall under the scope of spaces with $Q$-bounded
geometry considered with $Q>1$ in \cite[Section~9]{BHK}.

Cantor sets embedded in Euclidean spaces support a fractional 
Sobolev space theory based on Besov spaces. 
Indeed, Besov functions on such sets  are traces  
of the classical Sobolev functions on the ambient Euclidean spaces,
see Jonsson--Wallin~\cite{JW80},~\cite{JW84}.
See also Danielli--Garofalo--Nhieu~\cite{DaGaNh01},~\cite{DaGaNh07}
for such results on ambient Carnot--Carath\'eodory spaces.
Similar extension and trace theorems on more general subsets of
Euclidean spaces, obtained by means of 
Haj\l asz--Sobolev type spaces
on metric spaces, can be found in Haj\l asz--Martio~\cite{HM}.
For further discussion of Sobolev functions on Euclidean domains
and their extension and trace theorems
we refer the reader to Maz{\cprime}ya~\cite{MazBook}.

Thus the potential theory on such
Cantor sets is linked to the classical potential theory on the ambient Euclidean space. 
In the first part of this paper we obtain similar trace and extension theorems for Sobolev
and Besov spaces on regular trees and their Cantor type boundaries.
In particular, we show that the Besov space $\Bppal$ on the boundary  
is exactly the trace of the Newton--Sobolev space $\Np$ on the associated 
regular tree.  Here the smoothness exponent of the Besov space is
\[
\theta = 1- \frac{\be/\eps-Q}{p},
\]
where $Q$ is the Hausdorff dimension of the Cantor type boundary
and $\be/\eps$ is a ``dimension'' determined by 
the uniformization metric and a weighted  measure 
on the tree, see Propositions~\ref{prop-trace} and~\ref{prop-ext} 
and Theorem~\ref{thm-trace-sharp}.  In the setting considered here, we 
necessarily have $\be/\eps>Q$ as stipulated in~\eqref{Kvs-beta}, and so we 
have $\theta<1$. This is in contrast to Bourdon--Pajot~\cite{BP03}, where, 
when $p<Q$, one needs $\theta=Q/p>1$.

The trace theorem we obtain in this paper corresponds exactly to the exponents 
in Jonsson--Wallin~\cite{JW80},~\cite{JW84}. For trees, our result
extends and complements the general
trace result for Besov spaces on metric spaces 
in Gogatishvili--Koskela--Shan\-mu\-ga\-lin\-gam~\cite[Theorem~6.5]{GKS}.
As an application of our trace result, for sufficiently large $p$
we obtain embeddings of Besov spaces
on Cantor sets into spaces of H\"older continuous functions, 
see Proposition~\ref{prop-emb-Holder}.
Along the way we also show that trees with bounded degree,
equipped with a weighted metric and measure 
(called a uniformization metric in~\cite{BHK})  are doubling and support a 
$1$-Poincar\'e inequality, see Sections~\ref{sect-doubl} and~\ref{Poinc}.

In Bourdon--Pajot~\cite[Th\'eor\`eme~0.1]{BP03}
certain Besov spaces with the smoothness
exponent $\theta=Q/p$ were identified with cohomologies of conformal gauges.
As a special case of the above trace theorem, we
obtain a variant of this result in our setting of totally disconnected
Cantor type boundaries, 
see the comment following   Proposition~\ref{prop-trace}. 

In the Euclidean setting it is now well known 
that quasiconformal mappings preserve the classical Sobolev spaces
$W^{1,n}(\mathbb{R}^n)$, 
see the discussion in Heinonen--Kilpel\"ainen--Martio~\cite{HKM}. 
On totally disconnected spaces, quasiconformal mappings are not so useful,
because of the lack of nonconstant curves. 
Instead, one should consider \emph{quasisymmetries}, 
i.e.\ mappings satisfying
\begin{equation*}  
   \frac{d(f(x),f(y))}{d(f(x),f(z))} \le \eta\biggl(\frac{d(x,y)}{d(x,z)}\biggr)
\end{equation*}
for all $x,y,z$ with $x\ne z$, where $\eta:[0,\infty)\to[0,\infty)$ is a homeomorphism.
In the setting of quasisymmetric mappings between Ahlfors regular metric spaces, 
Koskela--Yang--Zhou~\cite[Theorem~5.1]{KYZ} recently obtained 
an invariance result for Besov spaces with the smoothness
exponent $\theta=Q/p$, where $Q$ is the Ahlfors-regularity dimension.
Quasisymmetries also turn out to be the natural maps
for studying boundaries of hyperbolic buildings and Gromov hyperbolic spaces, 
see for example~\cite{BP03}, \cite{FM}, \cite{Gr}, \cite{HSX}, \cite{Kap}, \cite{SX}, and
the references therein.  Since trees are the quintessential
models of Gromov hyperbolic spaces (see for example 
Bridson--Haefliger~\cite{BH}), these maps are natural for us as well. 

Hence in the second part of this paper we consider quasisymmetries 
between the boundaries of (not necessarily regular) trees, 
and show that they can be extended to \emph{rough quasiisometries}
(also called quasiisometries in the literature)
between the corresponding trees. A mapping $F$ between two
trees is a \emph{rough quasiisometry} if there are positive constants 
$L_1$, $L_2$ and $\Lambda$ such that for all points $x$ and $y$ 
in the domain tree,
\begin{equation*}   
     L_1|x-y|-\Lambda\le |F(x)-F(y)|\le L_2 |x-y|+\Lambda,
\end{equation*}
where $|\cdot-\cdot|$ denotes the unweighted geodesic distance on the 
tree, and the density condition that for each $z \in Y$ there is 
an $x \in X$ so that $|F(x) -z|< \Lambda$ holds.
Conversely, we show that every  rough quasiisometry between two trees induces
a quasisymmetry between the respective boundaries. These results appear in
Sections~\ref{sect-ext-qs} and~\ref{sect-ext-rough}, and 
extend a theorem by Gromov~\cite{Gr} to our setting of trees with 
totally disconnected  boundaries. The parameters
$\eta$, $L_1$ and $L_2$ associated with the obtained mappings 
are optimal and match each other, see Theorems~\ref{thm-qs2roughqiso} 
and~\ref{rough-qs}, and the comments following them.

The above extension result for quasisymmetries, together with our
trace result for Besov spaces, is in turn used to show that certain Besov spaces on
uniformly perfect ultrametric spaces are preserved by quasisymmetric mappings, 
see Theorem~\ref{thm-Besov-inv}. For example, we show that
every quasisymmetric $(\al_1,\al_2)$-power map as in~\eqref{eq-al1-al2-eta}
between two Cantor type spaces of Hausdorff dimensions $Q_X$ and $Q_Y$
induces the following embeddings between their Besov spaces,
\begin{align*}
B_{p,p}^{Q_Y/p+\tau/\al_1} &\longhookrightarrow B_{p,p}^{Q_X/p+\tau}
\longhookrightarrow B_{p,p}^{Q_Y/p+\tau/\al_2}, \\
B_{p,p}^{Q_Y/p-\tau/\al_2} &\longhookrightarrow B_{p,p}^{Q_X/p-\tau} 
\longhookrightarrow B_{p,p}^{Q_Y/p-\tau/\al_1},
\end{align*}
with $\tau\ge 0$, see Remark~\ref{rem-KYZ}.
This extends (in the setting of such spaces) the above mentioned Theorem~5.1 
in Koskela--Yang--Zhou~\cite{KYZ} beyond the case $\theta=Q/p$ considered there.
Thus, potential theory on uniformly perfect ultrametric spaces 
is associated with the theory of quasisymmetric mappings between them. 
We also direct interested readers to Hambly--Kumagai~\cite{HaKu}
for a discussion linking rough quasiisometries (called rough 
isometries in~\cite{HaKu}) to potential theory on graphs that arise 
as approximations of finitely ramified fractals. 

\begin{ack}
This research began when the first three authors visited 
the University of Cincinnati in 2010, and was continued
during the visits of the fourth author
to the University of Washington, Seattle, and to  the Link\"oping University in 2011. 
The authors wish to thank these institutions for their kind hospitality.
We also wish to thank David Minda for pointing out the reference Jeffers~\cite{Jeff}.

A.~B. and J.~B. were supported by the Swedish Research Council.
A.~B. was also a Fulbright scholar during his
visit to the University of Cincinnati, supported by the Swedish
Fulbright Commission, while J.~B. was a Visiting Taft Fellow
during her visit to the University of Cincinnati, supported by the Charles Phelps 
Taft Research Center at the University of Cincinnati. J.~G. was supported by the 
University of Washington at the beginning of this work and is supported by 
NSF Mathematical Sciences Postdoctoral Research Fellowship DMS-1004721. 
N.~S. was partly supported by the Taft Research Center,
the Simons Foundation grant \#200474 and the NSF grant DMS-1200915.
\end{ack}

\section{Notation and preliminaries}
\label{sect-prelim}

A \emph{graph} $G$ is a pair $(V,E)$, where $V$ is a set of vertices and
$E \subset V \times V$ is a set of edges. We are interested in undirected
graphs and consider $(x,y)$ and $(y,x)$ to be the same edge for $x,y \in V$.
Two vertices $x,y \in V$ are neighbors if $(x,y) \in E$. The degree of a vertex 
equals the number of neighbors it has. We will be interested in infinite graphs, 
but all vertices will be required to have finite degree.

The graph structure gives rise to a natural well-known connectivity structure.
A \emph{tree} is a connected graph without cycles, or equivalently
a graph such that for any pair of vertices $x,y \in V$  there
is a unique path of distinct edges connecting $x$ to $y$.
A graph (or tree) is made into a \emph{metric graph} by considering
each edge as a geodesic of length one.

We will only be interested in rooted trees.  A \emph{rooted tree} $X$ is a tree with 
a distinguished vertex called the \emph{root}, which we will denote by $0$.
In Section~\ref{sect-ext-qs} and later, when we deal with more than one tree, we
denote the root of a tree $X$ by $0_X$.

For $x\in X$, let $|x|$ be the distance from the root $0$ to $x$,
that is, the number of edges in the geodesic from $0$ to $x$. 
The geodesic connecting two vertices $x,y \in X$ is denoted
by $[x,y]$, and its length (the number of edges it contains)  is denoted
$|x-y|$; note that if $x$ and $y$ are descendants of two different children of $0$, then 
$|x-y|=|x|+|y|$. We write $x<y$ if a vertex $y$ is a descendant of a vertex $x$ (that is,
$|x|<|y|$ and $x$ lies in the geodesic connecting $0$ to $y$),
and more generally $x\le y$ if the geodesic from $0$ to $y$ passes
through $x$; in this case $|x-y|=|y|-|x|$.

The neighbors of a vertex $x \in X$ are of two types:
A \emph{parent} $y$ of $x$ is the neighbor which is closer to the root,
and all other neighbors are \emph{children} of $x$.
Each vertex has exactly one parent, except for the root itself which has none.
We will mostly consider rooted trees such that each vertex other than the root
has degree at least $3$, while the root $0$ is expected to have degree at least $2$.

The most familiar rooted trees are \emph{binary trees} in which each vertex has
exactly two children. More generally a \emph{$K$-ary} tree is a rooted tree such that
each vertex has exactly $K$ children. Note that a $K$-ary tree is almost regular: all vertices 
but the root have degree $K+1$, whereas the root has degree $K$. In this paper we say that
a tree is \emph{regular} if it is $K$-ary for some integer $K\ge 1$.

As is customary, we say that $A \simle B$ and equivalently
$B \simge A$, if there is a constant $C>0$
(independent of the variables that $A$ and $B$ are functions of)
such that $A\le C B$. We also write $A\simeq B$ if $A \simle B \simle A$.

Let $\tw>0$ be fixed from now on. We introduce a \emph{uniformizing} metric 
(in the sense of Bonk--Heinonen--Koskela~\cite{BHK}) on $X$ by 
\begin{equation}\label{unif-metric-trees}
     d_X(x,y) = \int_{[x,y]} e^{-\tw|z|} \,d|z|.
\end{equation} 
Here $d|z|$ stands for the measure which gives each edge Lebesgue measure 1, as we 
consider each edge to be an isometric copy of the open unit interval and the vertices are the 
points which close this interval. In this metric, $\diam X = 2/\tw$ if the root has 
at least two children and every vertex has at least one child.  Though this metric defines a 
weighted metric on both the vertices and the edges (seen as copies of an open interval on the 
real line), we will typically only discuss the distance between vertices.  

Throughout the paper we assume that $1 \le p < \infty$.

\section{Doubling condition on trees}
\label{sect-doubl}

\emph{In this section we assume that $X$ is a rooted tree such that
each vertex has at least one and at most $K$ children.}  The proofs of the results in this section, however, are substantially simpler if we assume $K$-regularity of the tree.  In Remark~\ref{rem-nonreg} we show how to remove the regularity assumption so that the results hold under the above
generality.

\medskip

The aim of this section is to show that the weighted measure 
\begin{equation}\label{tree-measure}
     d\mx(x) = e^{-\bw|x|}\, d |x|
\end{equation}
is doubling on $X$ (when equipped with the uniformizing metric $d_X$), where 
\begin{equation}\label{Kvs-beta}
    \bw>\log K
\end{equation} 
is fixed from now on. (If $\bw \le \log K$, then $\mu(X)=\infty$ for the regular $K$-ary
tree by~\eqref{eq-mu=infty} below,  and as $X$ is bounded, $\mu$ would not be doubling. 
This case is therefore not of interest to us. For nonregular trees we might have 
$\mu(X)<\infty$ even if  $\be\le\log K$, but we do not consider this case here.) 

We shall estimate the measure of balls in $X$ and show that it is doubling.
Let $B(x,r) = \{y\in X: d_X(x,y)<r\}$ denote an (open) ball in $X$ 
with respect to the metric $d_X$.
Also let $F(x,r) = \{y\in X: y\ge x \text{ and } d_X(x,y) < r\}$ be the 
downward directed ``half ball''. Note that $X = B(0,1/\tw) = F(0,1/\tw)$ and that 
$F(x,\infty)=F(x,e^{-\tw|x|}/\tw)$. We need to consider several cases depending 
on whether the radius $r$ is small compared with $|x|$ or not. 

The following comparisons and estimates for ``half balls'' will be useful.
We first state a simple algebraic lemma which will simplify our calculations.

\begin{lem}    \label{lem-1-1-t}
Let $\sigma>0$ and $t\in [0,1]$. Then
\[
    \min\{1,\sigma\} t \le 1 - (1-t)^\sigma \le \max\{1,\sigma\} t.
\]
\end{lem}

\begin{proof}
Let $f(\tau)=\tau^\sigma$.  If $\sigma\ge1$, then $f$ is convex and for $0\le\tau\le1$,
\[
    1-\tau \le f(1)-f(\tau) \le f'(1) (1-\tau).
\]
Letting $\tau=1-t$ gives the conclusion. The case $\sigma\le1$ is treated similarly.
\end{proof}

\begin{lem}  \label{lem-comp-B-G}
For every $x\in X$ and $r>0$ we have
\[
   F(x,r) \subset B(x,r) \subset F(z,2r),
\]
where $z\le x$ and 
\begin{equation}
   |z|=\max\biggl\{|x|-\frac{1}{\tw} \log(1+\tw re^{\tw|x|}),0 \biggr\}.
\label{eq-def-z}
\end{equation}
\end{lem}

\begin{proof}
The first inclusion is clear and true for all $r$. As for the second inclusion, 
note first that if $r\le(1-e^{-\tw|x|})/\tw$ then
\begin{equation}
    |z|=|x|-\frac{1}{\tw} \log(1+\tw re^{\tw|x|})
\label{eq-def-z-2}
\end{equation}
and 
\[
   d_X(x,z) = \int_{|z|}^{|x|} e^{-\tw t}\,dt 
        = \frac{1}{\tw} e^{-\tw |x|} (e^{-\tw(|z|-|x|)} - 1)  = r.
\]
At the same time, if $r\ge(1-e^{-\tw|x|})/\tw$, then 
\[
   d_X(x,z)=d_X(x,0) = \int_{0}^{|x|} e^{-\tw t}\,dt 
        = \frac{1}{\tw} (1-e^{-\tw|x|}) \le r.
\]
Hence, for all $r>0$ and all $y\in B(x,r)$, we have
\[
   d_X(y,z) \le d_X(x,y) + d_X(x,z) <  2r.
\]
Clearly, $y\ge z$ for such $y$, which completes the proof.
\end{proof}

\begin{lem}  \label{lem-est-F(z,r)-small-r}
For $z\in X$ and\/ $0<r\le e^{-\tw|z|}/\tw$,
\[
    \mu(F(z,r)) \simeq e^{(\tw-\be)|z|} r.
\]
\end{lem}

Note that the upper estimate in Lemma~\ref{lem-est-F(z,r)-small-r} holds
even if some vertices in $X$ have no children, i.e.\ if we allow finite branches.

\begin{proof}
Let $\rho>0$ be such that 
\begin{equation}
    \int_{|z|}^{|z|+\rho} e^{-\tw t}\,dt = \frac{1}{\tw} e^{-\tw|z|} (1-e^{-\tw\rho}) = r.
\label{eq-def-rho}
\end{equation}
Note that for each $|z|\le t\le|z|+\rho$, the number of  points 
$y\in F(z,r)$ with $|y|=t$ is approximately $K^{t-|z|}$. Hence
\begin{align}   \label{eq-mu=infty}
\mu(F(z,r)) &\simeq \int_{|z|}^{|z|+\rho} K^{t-|z|} e^{-\bw t}\,dt 
   = \frac{K^{-|z|}}{\be-\log K} e^{(\log K-\bw)|z|} (1-e^{(\log K-\bw)\rho}) \nonumber\\
     &= \frac{e^{-\be|z|}}{\be-\log K} (1-(1-\tw re^{\tw|z|})^{(\bw-\log K)/\tw}).
\end{align}
Lemma~\ref{lem-1-1-t} with $t=\tw re^{\tw|z|}$ implies that
\[
   \mx(F(z,r)) \simeq e^{-\be|z|} \tw re^{\tw|z|} \simeq e^{(\tw-\be)|z|} r.\qedhere
\]
\end{proof}

\begin{cor}  \label{cor-est-B(x,r)-small-r}
If\/ $0<r\le e^{-\tw|x|}/\tw$, then $\mu(B(x,r))\simeq e^{(\tw-\be)|x|} r$.
\end{cor}

\begin{proof}
Let $z$ be as in Lemma~\ref{lem-comp-B-G}.
If $z=0$ then $B(x,r)\subset F(0,r+\rho)$, where
\[
\rho=d_X(0,x)=\int_0^{|x|} e^{-\tw t}\,dt = \frac{1}{\tw} (1-e^{-\tw|x|}) \le r
\]
and $r+\rho\le 1/\tw = e^{-\tw|z|}/\tw$. For $z>0$ we have
\[
2r \le \frac{e^{-\tw|x|}(1+\tw re^{\tw|x|})}{\tw} = \frac{e^{-\tw|z|}}{\tw}.
\]
Since in both cases $1\le e^{|x|-|z|} \le (1+\tw re^{\tw|x|})^{1/\tw} \simeq 1$, 
the result now follows by applying Lemma~\ref{lem-est-F(z,r)-small-r}
to $F(x,r)$ and $F(z,2r)$ (or $F(0,r+\rho)$ for $z=0$).
\end{proof}

\begin{lem} \label{lem-est-B(x,r)-mellan-r}
Let $x\in X$ and 
\begin{equation}
\frac{e^{-\tw|x|}}{\tw} \le r \le \frac{1}{\tw} (1-e^{-\tw|x|}).
\label{eq-r-ge-1-e/1+e}
\end{equation}
Then
\(
   \mx(B(x,r)) \simeq r^{\bw/\tw}.
\)
\end{lem}

\begin{proof}
By Lemma~\ref{lem-comp-B-G}
and the second inequality in \eqref{eq-r-ge-1-e/1+e},  we have 
$B(x,r)\subset F(z,\infty)=F(z,e^{-\tw|z|}/\tw)$, 
where $0\le z\le x$ is given by \eqref{eq-def-z-2}.
Lemma~\ref{lem-est-F(z,r)-small-r} then yields
\begin{align}
\mx(F(z,\infty)) &\simle e^{(\tw-\be)|z|} e^{-\tw|z|} \simeq e^{-\bw|z|}.
\label{eq-est-G-e-be-z}
\end{align}
Now, the first inequality in \eqref{eq-r-ge-1-e/1+e} implies that
$1+\tw re^{\tw|x|} \le 2\tw re^{\tw|x|}$.
It then follows from \eqref{eq-def-z-2} that 
\[
e^{-\bw|z|} = e^{-\bw|x|} (1+\tw re^{\tw|x|})^{\bw/\tw}
\le (2\tw r)^{\bw/\tw}.
\]
Inserting this into~\eqref{eq-est-G-e-be-z} finishes the proof of 
the upper bound.

As for the lower bound we have, using \eqref{eq-def-z-2} again, that
\[
\mx(B(x,r)) \ge \int_{|z|}^{|x|} e^{-\bw t}\,dt 
= \frac{e^{-\bw|x|}}{\bw} ((1+\tw re^{\tw|x|})^{\bw/\tw} - 1).
\]
The function $f(t)=((1+t)^{\bw/\tw}-1)/t^{\bw/\tw}$ is monotone 
and $\lim_{t\to\infty} f(t)=1$.
As $\tw re^{\tw|x|}\ge1$, this yields that
$f(\tw re^{\tw|x|})\ge \min\{1,f(1)\}\simeq1$. Hence
\[
   \mx(B(x,r)) \simge  e^{-\bw|x|} (\tw re^{\tw|x|})^{\bw/\tw} \simeq r^{\bw/\tw}.\qedhere
\]
\end{proof}

\begin{lem}   \label{lem-est-B(x,r)-large-r}
Let $x\in X$ and $d_X(x,0)=(1-e^{-\tw|x|})/\tw \le r \le 2\diam X$. 
Then
\(
\mx(B(x,r)) \simeq r.
\) 
In particular, if $x=0$ then this estimate holds for all $r\ge0$.
\end{lem}

\begin{proof}
We have $0\in \itoverline{B(x,r)}$ by assumption, 
and hence $B(x,r)\subset F(0,2r)$.
It then follows from Lemma~\ref{lem-est-F(z,r)-small-r} that
\[
  \mx(B(x,r)) \le \mx(F(0,2r)) \simle r.
\]
As for the lower bound, consider first the case $r<1/\tw$.
As $0\in \itoverline{B(x,r)}$,
letting $\rho=-\log(1-\tw r)/\tw$, implies
\[ 
\mx(B(x,r)) \ge  \int_0^\rho e^{-\bw t} \,dt = \frac{1-e^{-\bw\rho}}{\bw} 
     =  \frac{1}{\bw} (1- (1-\tw r)^{\bw/\tw}).
\] 
Lemma~\ref{lem-1-1-t} then yields
\(
\mx(B(x,r)) \simge r.
\)
If $1/\tw\le r\le 2\diam X \le  4/\tw$, 
then
\[
\mx(B(x,r)) \ge \mx\bigl(B\bigl(x,\tfrac{1}{5}r\bigr)\bigr) \simeq r.\qedhere
\]
\end{proof}

The following proposition follows from Corollary~\ref{cor-est-B(x,r)-small-r}
together with Lemmas~\ref{lem-est-B(x,r)-mellan-r}
and~\ref{lem-est-B(x,r)-large-r}.
Note that if $\bw=\tw$, then $X$ becomes Ahlfors $1$-regular, that is,
$\mu(B(x,r))\simeq r$ for all $x\in X$ and $0 < r \le 2 \diam X$.

\begin{prop}   \label{prop-est-B(x,r)}
Let $x\in X$, $0<r\le 2\diam X$ and $R_0=e^{-\tw|x|}/\tw$.
If\/ $|x|\le(\log2)/\tw$ then $\mx(B(x,r)) \simeq r$.
If\/ $|x|\ge(\log2)/\tw$ then 
\[
\mx(B(x,r)) \simeq \begin{cases}
 e^{(\tw-\bw)|x|} r & \text{for } r\le R_0 \\ 
 r^{\bw/\tw}       & \text{for } r\ge R_0.  
          \end{cases}
\]
\end{prop}

\begin{proof}
If $|x|\le(\log2)/\tw$ then $e^{(\tw-\bw)|x|}\simeq1$ and the result follows
directly from Corollary~\ref{cor-est-B(x,r)-small-r} and 
Lemma~\ref{lem-est-B(x,r)-large-r}. 
(Note that $e^{-\tw|x|}\ge\tfrac12$ in this case.)

If $|x|\ge(\log2)/\tw$ and $r<(1-e^{-\tw|x|})/\tw$ then the estimate
follows directly from Corollary~\ref{cor-est-B(x,r)-small-r} and
Lemma~\ref{lem-est-B(x,r)-mellan-r}.
For $r\ge(1-e^{-\tw|x|})/\tw \ge 1/2\eps$ we have by 
Lemma~\ref{lem-est-B(x,r)-large-r} that
$\mx(B(x,r))\simeq r \simeq 1 \simeq r^{\be/\tw}$.
\end{proof}

\begin{cor}  \label{cor-dim-s}
The following dimension condition holds for all balls $B(x,r)$ and $B(x',r')$
with $x'\in B(x,r)$ and\/ $0<r'\le r$,
\begin{equation}
\frac{\mx(B(x',r'))}{\mx(B(x,r))} \simge \Bigl( \frac{r'}{r} \Bigr)^s,
\label{eq-dim-cond}
\end{equation}
where $s=\max\{1,\bw/\tw\}$ is the best possible.
\end{cor}

\begin{proof}
Assume first that $x'=x$. 
The general case $x'\in B(x,r)$ will be taken care of later.
Assume also that $r\le2\diam X$.

If $|x|\le(\log2)/\tw$, then 
$\mx(B(x,r))\simeq r$ for all $0\le r\le 2\diam X$,
by Proposition~\ref{prop-est-B(x,r)}, and \eqref{eq-dim-cond} follows.
Also, if $|x|\ge(\log2)/\tw$ and both $r'$ and $r$ belong 
to the same interval in Proposition~\ref{prop-est-B(x,r)}, 
then \eqref{eq-dim-cond} follows directly
from Proposition~\ref{prop-est-B(x,r)}.

Let us therefore assume that $|x|\ge (\log2)/\tw$
and $r'\le R_0\le r$. Then
\[
\frac{\mx(B(x,r'))}{\mx(B(x,r))} \simeq \frac{e^{(\tw-\bw)|x|} r'}{r^{\bw/\tw}}
\simeq \begin{cases}
{\displaystyle \biggl( \frac{R_0}{r'} \biggr)^{\bw/\tw-1} 
       \Bigl(\frac{r'}{r} \Bigr)^{\bw/\tw}},
& \text{if } \bw\ge\tw, \\
{\displaystyle \biggl( \frac{R_0}{r} \biggr)^{\bw/\tw-1} \Bigl(\frac{r'}{r} \Bigr)},
& \text{if } \bw\le\tw.
\end{cases}
\]
Since $R_0/r'\ge1\ge R_0/r$, \eqref{eq-dim-cond} follows in this case as well.
This also shows that \eqref{eq-dim-cond} cannot hold for any
$s < \max\{1,\be/\eps\}$.

Now, let $x'\in B(x,r)$ and $0<r'\le r\le \diam X$.
Then $B(x,r)\subset B(x',2r)$ and hence by the above

\[
  \frac{\mx(B(x',r'))}{\mx(B(x,r))} 
   \ge \frac{\mx(B(x',r'))}{\mx(B(x',2r))} \simge \Bigl( \frac{r'}{2r} \Bigr)^s.
\]
Finally, if $r\ge \diam X$, then $B(x,r)=X=B(x',\diam X)$ and thus
\[
  \frac{\mx(B(x',r'))}{\mx(B(x,r))}  = \frac{\mx(B(x',r'))}{\mx(B(x',\diam X))} \simge \Bigl( \frac{r'}{\diam X} \Bigr)^s
       \ge \Bigl( \frac{r'}{r} \Bigr)^s.\qedhere
\]
\end{proof}

This immediately gives the following corollary.

\begin{cor}
The measure $\mx$ is doubling, i.e.\ $\mu (B(x,2r)) \simle \mu (B(x,r))$.
\end{cor}

\begin{remark}  \label{rem-nonreg}
Assume that $X$ is nonregular but such that each vertex has at least one 
and at most $K$ children for some $K\ge1$, and 
let $\be>\log K$.
By adding edges to $X$ we obtain a regular $K$-ary tree
$X_K$ containing $X$, equipped with the same metric and measure as $X$. 
For every $x\in X$ and $r>0$ we then have
\[
B(x,r) \subset B_{X_K}(x,r) := \{y\in X_K: d_X(x,y)<r\}.
\]
At the same time, an infinite geodesic from 0 in $X$, passing through $x$,
can be regarded as a regular $1$-ary tree $X_1$ contained in $X$. Hence
\[
B(x,r) \supset B_{X_1}(x,r) := \{y\in X_1: d_X(x,y)<r\}.
\]
Since $X_1$ and $X_K$ are regular $1$-ary and $K$-ary trees, all results in 
this section apply to them and we have by Proposition~\ref{prop-est-B(x,r)},
\[
\mu(B_{X_1}(x,r)) \le \mu(B(x,r)) \le \mu(B_{X_K}(x,r)) \simeq \mu(B_{X_1}(x,r)).
\]
This implies that the results in this section hold also for nonregular $X$
with bounds on the number of children.
In particular, the measure $\mu$ given by~\eqref{tree-measure} is doubling.
\end{remark}

\section{Poincar\'e inequality on trees}\label{Poinc}

\emph{In this section we assume that $X$ is a rooted tree such that
each vertex has at most $K$ children
{\rm(}in particular,  finite branches are allowed\/{\rm)}.}

\medskip

We will show that the measure and metric, 
given in Section~\ref{sect-doubl}, together support a $1$-Poincar\'e inequality for
functions and their upper gradients. Let $u \in L^1\loc (X)$.  
We say that a Borel function $g:X \to [0,\infty]$ is an \emph{upper gradient} of $u$ if
\begin{equation} \label{eq-ug-cond}
      |u(z) - u(y)| \leq \int_\gamma g \,d_X s 
\end{equation}
whenever $z,y \in X$ and $\gamma$ is the geodesic from $z$ to $y$, 
where $d_X s$ denotes the arc length 
measure with respect to the metric $d_X$.
In the current literature on metric measure spaces,
the inequality~\eqref{eq-ug-cond} is required to hold for all rectifiable curves with end points 
$z$ and $y$. However, in the setting of a tree 
any such curve contains
the geodesic connecting $z$ to $y$,
and it is therefore equivalent to define it as above on the tree.

The notion of upper gradients is due to Heinonen and 
 Koskela~\cite{HeiKo}; we refer interested readers to
Haj\l asz~\cite{Haj1} and Bj\"orn--Bj\"orn~\cite{BBbook} for detailed 
 discussions on upper gradients.

The \emph{Newtonian space} $\Np(X)$, 
$1  \le p < \infty$, 
is defined as the collection of functions for which the norm
\[
   \|u\|_{\Np(X)}
    := \biggl( \int_X u^p \,d\mu + \inf_g \int_X g^p \,d\mu\biggl)^{1/p} <\infty,
\]
where the infimum is taken over all upper gradients of $u$.
If $u \in \Np(X)$, then it has a minimal \p-weak upper gradient
$g_u$, which in our case is an upper gradient
(since the empty set is the only curve family with zero \p-modulus
on a tree).
The minimal upper gradient $g_u$ is unique up to measure zero, 
and is minimal 
in the sense that if $g \in L^p(X)$ is any upper gradient of $u$
then $g_u \le g$ a.e. 
We refer the interested reader 
to Shanmugalingam~\cite[Corollary~3.7]{Sh-harm} ($p>1$)
and Haj\l asz~\cite[Theorem~7.16]{Haj1} ($p \ge1$)
for proofs of the existence and uniqueness of
such a minimal upper gradient.

Our aim in this section is to establish the following $1$-Poincar\'e inequality,
\[ 
\vint_B |u(y) - u_B| \,d\mu(y) \leq C r \vint_B g \,d\mu , 
\]
where $u$ is integrable, $g$ is an arbitrary upper gradient of $u$,
$r$ is the radius of $B$ and
\[
  u_E :=\vint_E u \,\dmu := \frac{1}{\mu(E)} \int_{E} u\, d\mu \]
whenever  $E$ is measurable and $0<\mu(E)<\infty$.

\begin{lem}    \label{lem-est-int-u(y)-u(z)}
Let $B=B(x,r)\subset X$ be a ball and let $z$ be as in Lemma~\ref{lem-comp-B-G}.
Then for every $u:B\to\R$ and every upper gradient $g$ of $u$ in $B$,
\[
\int_B |u(y)-u(z)|\,d\mx(y) 
\le \int_B g(w) e^{(\bw-\tw)|w|} \mx(\{y\in B: y\ge w\}) \,d\mx(w).
\]
\end{lem}

\begin{proof}
As $d_X s= e^{-\tw|x|}\,d|x| = e^{(\bw-\tw)|x|}\,d\mu(x)$, we obtain using the 
Fubini theorem that  
\begin{align*}
\int_B |u(y)-u(z)|\,d\mx(y)
    & \le \int_B \int_{[z,y]} g \, d_X s \, d\mu(y) \\
    & = \int_B \int_{[z,y]} g(w) e^{(\bw-\tw)|w|}\,d\mu(w)\, d\mu(y) \\
    & = \int_B g(w) e^{(\bw-\tw)|w|} \mx(\{y\in B: y\ge w\}) \,d\mx(w).
    \qedhere
\end{align*}
\end{proof}

\begin{thm} \label{thm-1-PI}
The space $X=(X,d,\mx)$ supports a\/ $1$-Poincar\'e inequality. 
\end{thm}

For simplicity, the proof below  will
assume that $X$ is regular $K$-ary; to generalize the proof to nonregular
trees one only needs to note the comment after 
Lemma~\ref{lem-est-F(z,r)-small-r}.

\begin{proof}
Let $B=B(x,r)$ be a ball.
We shall use the estimate from Lemma~\ref{lem-est-int-u(y)-u(z)}.
Let us first estimate $|w|$ for $w\in B$.
For $|w|\ge|x|$, we must have
\[
r > \int_{|x|}^{|w|} e^{-\tw t}\,dt 
= \frac{1}{\tw} e^{-\tw|x|} (1 - e^{\tw(|x|-|w|)}),
\]
which yields
\begin{equation}
|w| < |x| - \frac{1}{\tw} \log(1-\tw re^{\tw|x|}).
\label{eq-est-largest-xi}
\end{equation}
For $|w|<|x|$, this is trivially true.
Now, we distinguish two cases.

(i) Assume first that $r\le e^{-\tw|x|}/3\tw$, 
i.e.\  $\tw re^{\tw|x|} \le \tfrac13$.
A simple calculation using \eqref{eq-est-largest-xi} shows that
for all $w\in B$,
\begin{align*}
e^{-\tw|w|} &> e^{-\tw|x|} (1- \tw re^{\tw|x|})
\ge \frac{2e^{-\tw|x|}}{3},
\end{align*}
and hence 
\[
2r \le \frac{2e^{-\tw|x|}}{3\tw} < \frac{e^{-\tw|w|}}{\tw}.
\]
Lemma~\ref{lem-est-F(z,r)-small-r}  then implies that
\[
    \mx(\{y\in B:y\ge w\}) \le \mx(F(w,2r)) \simle e^{(\tw-\bw)|w|} r.
\]
Inserting this into Lemma~\ref{lem-est-int-u(y)-u(z)} yields
\begin{equation}
   \int_B |u(y)-u(z)|\,d\mx(y) \simle r \int_B g(w) \,d\mx(w). \label{eq-PI-with-u(z)}
\end{equation}

(ii) Assume  instead that $r\ge e^{-\tw|x|}/3\tw$.  Then
\[
\mx(\{y\in B:y\ge w\}) \le \mx(F(w,e^{-\tw|w|}/\tw)) \simle e^{-\bw|w|},
\]
by Lemma~\ref{lem-est-F(z,r)-small-r}. 
Inserting this into Lemma~\ref{lem-est-int-u(y)-u(z)} yields
\begin{align}
\int_B |u(y)-u(z)|\,d\mx(y) 
&\simle \int_B g(w) e^{-\tw|w|} \,d\mx(w)
\le  e^{-\tw|z|} \int_B g(w) \,d\mx(w).
\label{eq-est-PI-with-e-z}
\end{align}
From~\eqref{eq-def-z} and the above choice of $r$ we have
\begin{equation*}
e^{-\tw|z|} \le e^{-\tw|x|} (1 + \tw re^{\tw|x|}) 
= e^{-\tw|x|} + \tw r \le 4\tw r.
\end{equation*}
Together with \eqref{eq-est-PI-with-e-z}, this proves \eqref{eq-PI-with-u(z)}
also for $r \ge e^{-\tw|x|}/3\tw$.

To finish the proof, observe that $u(z)$ in \eqref{eq-PI-with-u(z)}
can by replaced by $u_B$ as follows
\[
\vint_B |u-u_B|\,d\mx
\le \vint_B |u(y)-u(z)|\,d\mx(y) + |u(z)-u_B|
\le 2\vint_B |u(y)-u(z)|\,d\mx(y).
\]
\end{proof}

\begin{cor} \label{cor-pp-PI}
The space $X=(X,d,\mx)$ supports a\/ $(p,p)$-Poincar\'e inequality,
i.e.
\[ 
     \vint_B |u - u_B|^p \,d\mu \leq C r \vint_B g^p \,d\mu.
\]
\end{cor}

That this follows 
from the $1$-Poincar\'e inequality established in 
Theorem~\ref{thm-1-PI}
is well known and a simple proof can be given along the lines of 
pp.\ 11--12 in Heinonen--Kilpel\"ainen--Martio~\cite{HeKiMa}
(alternatively one can e.g.\ appeal to Theorem~5.1 in 
Haj\l asz--Koskela~\cite{HaKo}).

\section{Hausdorff dimension of 
\texorpdfstring{$\partial X$}{boundary of X} and Besov functions on 
\texorpdfstring{$\partial X$}{boundary of X}}
\label{sect-Traces}

\emph{In Sections~\ref{sect-Traces} and~\ref{sect-Extensions}
 we assume that $X$ is a regular $K$-ary tree, $K \ge 2$.}

\medskip

In this section we construct the boundary of the  regular $K$-ary tree 
and show that it is Ahlfors regular with regularity exponent
depending solely on $K$ and on the metric density exponent $\eps$ of
the tree. We then study a family of Besov spaces $B^\theta_{p,p}(\partial X)$ of 
functions on the boundary $\partial X$ of the tree, and prove that continuous 
functions are dense in these Besov spaces.

By the discussion in  the previous two sections, we know that $X$  is a metric space 
equipped with a doubling measure supporting a $1$-Poincar\'e inequality. 

A tree is the quintessential Gromov hyperbolic space, and hence we can consider
the visual boundary of the tree as in Bridson--Haefliger~\cite{BH}. 
The discussion in Bonk--Heinonen--Koskela~\cite{BHK} tells us that this
visual boundary is the same as the metric boundary of the tree equipped with the uniformizing metric
$d_X$ given in~\eqref{unif-metric-trees}. The focus of this section is to describe and study the properties
of this boundary.

We define the boundary of a tree $X$, denoted $\partial X$, by completing $X$ with 
respect to the metric $d_X$. An equivalent construction of $\partial X$ is as follows.
An element $\zeta$ in $\partial X$ is identified with an infinite geodesic  in $X$ starting 
at the root $0$. If we denote the geodesic by concatenation of vertices, then
\[ 
     \zeta = 0 x_1 x_2 x_3 \ldots,
\]
where $x_i$ is a vertex in $X$ at a
distance $i$ from the root, and $x_{i+1}$ is 
a child of $x_i$. Given two points $\zeta, \xi \in \partial X$,
the distance between them is the length (with respect to the
metric $d_X$) of the infinite geodesic $[\zeta,\xi]$ between them.
If this infinite geodesic is $k$ edges  from the root $0$ 
(counting each edge as having unit length)  then by~\eqref{unif-metric-trees},
\begin{equation}\label{bdy-metric}
   d_{X}(\zeta,\xi)=2
   \int_k^\infty e^{-\eps t} \, dt = \frac{2}{\eps} e^{-\eps k}.
\end{equation}
Following Bridson--Haefliger~\cite{BH}, the restriction of
$d_X$ to $\bdy X$ is called the \emph{visual metric} on $\bdy X$.

The metric $d_X$ is thus defined on $\itoverline{X}$ and we will 
consider balls with respect to this metric in $X$, $\itoverline{X}$ and $\bdry X$.
To avoid confusion, points in $X$ are denoted by Latin letters such as
$x$, $y$, $z$ and $w$, while for points in $\bdry X$ we use Greek letters
such as $\z$, $\xi$, $\chi$ and $\om$.

Balls in $X$ will thus be denoted $B(x,r)$, while $B(\z,r)$ stands for a ball in $\bdry X$.
Since $X$ and $\bdry X$ are disjoint, this should not cause any confusion.
For balls in $\itoverline{X}$ we write $\bB(x,r)$ and $\bB(\z,r)$,
depending on whether the center lies in $X$ or $\bdry X$.

Recall that a metric space $(Z,d_Z)$ is an \emph{ultrametric space} 
if for each triple of points $x,y,z\in Z$ we have $d_Z(x,z)\le \max\{d_Z(x,y),d_Z(y,z)\}$.

\begin{lem}\label{lem-ultra-metric-center}
The metric space\/ $(\partial X, d_X)$ 
is an ultrametric space, and consequently, whenever
$\zeta\in\partial X$, $r>0$, and $\xi\in B(\zeta,r)$, we have $B(\zeta,r)=B(\xi,r)$.
\end{lem}

\begin{proof}
Let $\zeta,\xi,\eta\in\partial X$.
Let $k$ be the number of edges in the shortest curve connecting $0$ to the infinite geodesic
$[\zeta,\xi]$, $k_1$ be the number of edges in the shortest curve connecting $0$ to $[\zeta,\eta]$,
and $k_2$ be the number of edges in the shortest curve connecting $0$ to $[\xi,\eta]$. Then
$k\ge \min\{k_1,k_2\}$, and so
\[
  e^{-\eps k}\le \max\{ e^{-\eps k_1}, e^{-\eps k_2}\},
\]
from which, together with \eqref{bdy-metric}, the ultra metric property follows. The latter part of the lemma
is a direct consequence of the ultrametric property of $\partial X$.
\end{proof}

\begin{lem} \label{lem-dimension}
$\bdy X$ is an Ahlfors $Q$-regular space with Hausdorff dimension 
\[ 
    Q=\frac{\log K}{\eps}.
\] 
\end{lem}

\begin{proof} 
We equip $\bdy X$ with the natural probability measure $\nu$
as in Falconer~\cite{Fal} by distributing the unit mass uniformly on $\partial X$. Let  
$x \in \bdy X$ and $0<r \le e^\eps \diam \bdy X =  e^\eps\diam X = 2e^\eps/\eps$.
Let $k \in \Z$ be such that
\[
      \frac{2}{\eps} e^{-\eps k} < r  \le  \frac{2}{\eps} e^{-\eps (k-1)}.
\]
Then $\bdy X$ is the union of $K^k$ disjoint open balls of radius $r$,
each of which has, by definition, $\nu$-measure $K^{-k}$.
Since any point of a ball can be used as a center we see that
$\nu(B(\z,r))=K^{-k}$ for every $\z\in\bdry X$,
where
\begin{equation} \label{eq-level}
    k=k(r):=\biggl\lfloor 1+\frac{1}{\eps}\log \frac{2}{\eps r}\biggr \rfloor
\end{equation}
and so 
\begin{equation} \label{eq-Q}
   \nu(B(\z,r))  \simeq r^Q.
\end{equation}
Since $\nu(\bdy X)=1$, we see that
\eqref{eq-Q} also holds even when $e^\eps \diam X\le r \le  2 \diam X$
(but with different implicit constants).
It also follows that
\[
    \nu(A) \simeq  \mathcal{H}^Q(A)
    \quad
    \text{for all measurable sets $A \subset \bdy X$},
\]
 where
$\mathcal{H}^Q$ denotes the $Q$-dimensional Hausdorff measure.
Thus $\nu$ is the normalized $Q$-dimensional Hausdorff measure
on $\bdy X$.
\end{proof}

\begin{example}  \label{ex-fractal}
The boundary $\bdry X$ can be (up to a biLipschitz mapping) identified
with a totally disconnected regular fractal set defined by $K$ similarities, 
each with contraction ratio $e^{-\eps}$.

For example, $K=2$ and $\eps=\log3$ gives the usual ternary Cantor dust, while
$K=4$ and $\eps=\log4$ gives the $1$-dimensional Garnett--Ivanov set, 
which was the first example 
of a set  in the plane with positive length but zero analytic capacity, 
see Garnett~\cite{garn} and Ivanov~\cite[footnote on p.\ 346]{ivan}.

Letting $K=3$ and $\eps=\log3$ leads to the following $1$-dimensional totally
disconnected ``Sierpi\'nski dust'': Split an equilateral triangle into 9
smaller congruent equilateral triangles and pick the three which contain
the vertices of the original one. Repeat this construction for each of
the chosen smaller triangles.

If we instead let $K=3$ and $\eps=\log2$ then the resulting fractal will
have dimension $Q=(\log3)/(\log2)$. This dimension is the 
same as that of
the Sierpi\'nski gasket, 
but this Cantor set will be totally disconnected, i.e.\ the three subgenerations
in the usual Sierpi\'nski gasket have to be considered as having positive 
distance from each other. 
In fact, this   fractal is just a snowflaked version of the above 
``Sierpi\'nski dust'' (with a new metric $|x-y|^{1/Q}$).
\end{example}

We now wish to find the connection between certain function spaces on $X$ and $\partial X$.   
Namely in $X$ we consider the Newtonian space $\Np(X)$, as
defined in Section~\ref{Poinc}.
On the boundary, $\partial X$, we consider another space of functions, the Besov space.
Let $f:\partial X \to \mathbb{R}$.  
Let $\nu$ denote the normalized $Q$-dimensional Hausdorff measure on $\partial X$.  
For $t >0$ and $p \ge 1$ we set
\[ 
    E_p (f,t) := \biggl( \int_{\partial X} \vint_{B(\zeta,t)} |f(\zeta)-f(\xi)|^p 
      \,d\nu(\xi) \,d\nu(\zeta) \biggr)^{1/p}, 
\]
and for $\theta > 0$ and $q \ge 1$,
\begin{equation}  \label{def-Bppal-norm} 
  \| f \|_{B^{\theta}_{p,q}(\partial X)} 
  := \biggl( \int_0^\infty \biggl(\frac{E_p(f,t)}{t^\theta}\biggr)^{q} \frac{dt}{t} \biggr)^{1/q}. 
 \end{equation}
The Besov space $B^{\theta}_{p,q}(\partial X)$ consists of all 
$f\in L^p(\bdy X)$
for which this seminorm is finite.
In this paper we only deal with the Besov spaces for which $q=p$, 
that is, the  spaces $B^\theta_{p,p}(\partial X)$.
The expression 
\[
     \|f\|_{\tB^\theta_{p,p}(\partial X)}:=  \|f\|_{L^p(\partial X)}  + \|f\|_{B^\theta_{p,p}(\partial X)}
\]
is a norm on $B^{\theta}_{p,p}(\partial X)$.

The following lemma shows that the Besov seminorm~\eqref{def-Bppal-norm} can 
equivalently be calculated as an infinite sum.
We shall also see that in bounded spaces (as here) the integral in the
definition of the Besov seminorm can be taken over a finite interval.
We formulate these results for our situation, but they hold in any metric
space, provided that the measure on it is doubling.

\begin{lem}  \label{lem-norm-equiv-sum}
Let\/ $0<\s<1$ and $t_n=C\s^n$, $n\in\Z$. Then
\[
\|f\|_{\Bppal(\bdry X)}^{p} \simeq \sum_{n=-\infty}^\infty \biggl( 
           \frac{E_p(f,t_n)}{t_n^\theta} \biggr)^p.
\]
Furthermore, 
\begin{equation}  \label{eq-equiv-norm}
 \|f\|_{\Bppal(\bdry X)}^p\simeq
    \int_{\bdy X}\int_{\bdy X}\frac{|f(\zeta)-f(\xi)|^p}{d_X(\zeta,\xi)^{\theta p}\nu(B(\zeta,d_X(\zeta,\xi)))}\,
       d\nu(\xi)\, d\nu(\zeta).
\end{equation}
\end{lem}

\begin{proof}
The doubling property for $\nu$ implies that for $t_{n+1}\le t\le t_{n}$,
\[
    E_p(f,t_{n+1}) \simle E_p(f,t) \simle  E_p(f,t_{n}).
\]
Hence
\begin{align*}
\int_{t_{n+1}}^{t_{n}} \biggl(\frac{E_p(f,t)}{t^\theta} \biggr)^p \frac{dt}{t} 
&\simle  \biggl(\frac{E_p(f,t_n)}{t_n^\theta} \biggr)^p 
 \int_{t_{n+1}}^{t_{n}} \frac{dt}{t} \simeq \biggl( \frac{E_p(f,t_n)}{t_n^\theta} \biggr)^p.
\end{align*}
The lower bound in terms of $E_p(f,t_{n+1})$ is obtained similarly and summing up
over all $n\in\Z$ completes the proof of the first part. The second
part follows directly from the computations in
Gogatishvili--Koskela--Shanmugalingam~\cite[p.~226]{GKS}
or by combining the Ahlfors regularity 
of $\nu$ with~\eqref{eq-Fubini-v-n} and~\eqref{eq-est-int-rn-infty} below (with $r_n$ replaced by 0).
\end{proof}

\begin{remark}   \label{rem-finite-int}
Note that if $\diam \partial X < t_0$,
the sum in Lemma~\ref{lem-norm-equiv-sum}
can equivalently be taken over $n\ge0$.
Consequently, the integral in \eqref{def-Bppal-norm} can be taken over
the finite interval $(0,2\diam \partial X)$.
\end{remark}

It follows that 
\begin{equation} \label{Besov-subset}
    B_{p,p}^{\theta_2}(\bdy X) \subset B_{p,p}^{\theta_1}(\bdy X)
    \quad \text{if }0 < \theta_1 < \theta_2.
\end{equation}
The following example shows that the converse inclusion is false.
This also directly yields that $B_{p,p}^{\theta}(\bdy X)$ is nontrivial for all $\theta>0$.
Note also that~\eqref{eq-equiv-norm} and the H\"older inequality yield
that for $1\le q<p$ and $0<\tau<\theta$, 
\begin{equation*} 
\|f\|_{B^\tau_{q,q}(\bdry X)} \le \|f\|_{\Bppal(\bdry X)}
\biggl( \int_{\bdry X}\int_{\bdry X} 
\frac{d_X(\z,\xi)^{(\theta-\tau)pq/(p-q)} \,d\nu(\xi)\,d\nu(\z)}{\nu(B(\z,d_X(\z,\xi)))} 
\biggr)^{1/q-1/p},
\end{equation*}
where the last integral converges since $\tau<\theta$.
Thus
\begin{equation} \label{Besov-subset-2}
    \Bppal(\bdry X)\subset B^\tau_{q,q}(\bdry X)
    \quad \text{whenever } 1\le q\le p \text{ and } 0<\tau<\theta.
\end{equation}

\begin{example} \label{ex-power}
Let $\alp >-Q/p$ with $\alp \ne 0$.  Fix 
$\zeta_0 \in \bdy X$ and
set $f(\xi)=d_X(\xi,\zeta_0)^\alp$.
Further, let $0 < t < e^\eps \diam\partial X$ and  
let $k(t)$ be given by \eqref{eq-level}.

Let us first estimate $E_p(f,t)$ for $\alp>0$.
Let $\zeta \in B(\zeta_0,t)$
and $d=d_X(\zeta,\zeta_0)$.
Note that if $d_X(\xi,\zeta_0)= d$,
then $f(\zeta)=f(\xi)$.
Thus 
\begin{align*}
     \int_{B(\zeta_0,t)} |f(\xi)-f(\zeta)|^p \, d\nu(\xi)
     &\le  \int_{B(\zeta_0,d)} f(\zeta)^p \, d\nu(\xi) 
      + \int_{B(\zeta_0,t) \setm B(\zeta_0, e^{\eps} d)} f(\xi)^p 
         \, d\nu(\xi) \\
     & \le \int_{B(\zeta_0,t) \setm  B(\zeta_0,  d)} f(\xi)^p \, d\nu(\xi),
\end{align*}
since $\nu(B(\zeta_0,d))=\nu(B(\zeta_0, e^{\eps} d) \setm  B(\zeta_0,d))/(K-1)$
and $f(\z)= f(\xi)$ for
$\xi\in B(\zeta_0,e^{\eps} d)\setm B(\zeta_0,d)$.
Hence, by summing over the shells around $\zeta_0$ we obtain that
\begin{align*}
     \vint_{B(\zeta_0,t)} |f(\xi)-f(\zeta)|^p \, d\nu(\xi)
     &\le \vint_{B(\zeta_0,t)} f(\xi)^p \, d\nu(\xi) \\
     &\simeq K^{k(t)}\sum_{k=k(t)+1}^\infty e^{-\eps k \alp p} K^{-k}
     \simeq  e^{-\eps k(t) \alp p}
     \simeq  t^{\alp p}.
\end{align*}
Here we used the fact that for each integer $k>k(t)$ there are $K-1$
balls in $\bdry X$ of radius $r_k=2e^{-\eps k}/\eps$ and at distance $r_{k}$
from $\z_0$.

On the other hand, if $\zeta \in B(\zeta_0,t) \setm B(\zeta_0,e^{-\eps}t)$, 
then 
\begin{align*}
     \vint_{B(\zeta_0,t)} |f(\xi)-f(\zeta)|^p \, d\nu(\xi)
     &\simge \vint_{B(\zeta_0,e^{-\eps}t) } |f(\xi)-f(\zeta)|^p \, d\nu(\xi) \\
     &\ge   |(te^{-2\eps})^\alp-(te^{-\eps})^\alp|^p 
     \simeq t ^{\alp p} .
\end{align*}
Since $B(\zeta_0,t)=B(\zeta,t)$, 
it thus follows
that
\begin{equation} \label{eq-Ex-AA}
     \int_{B(\zeta_0,t)} \vint_{B(\zeta,t)} |f(\xi)-f(\zeta)|^p \, d\nu(\xi)
     \, d\nu(\zeta)
     \simeq K^{-k(t)} t^{\alp p}
     \simeq t^Q t^{\alp p}
\end{equation}
if $\alp >0$.

Let us similarly estimate for $-Q/p < \alp < 0$.
If $\zeta \in B(\zeta_0,e^{-\eps k}t) \setm B(\zeta_0,e^{-\eps (k+1)}t)$,
where $k$ is a nonnegative integer, then after again
summing over shells we get that
\begin{align*}
     \vint_{B(\zeta_0,t)} |f(\xi)-f(\zeta)|^p \, d\nu(\xi)
     & \simeq \sum_{j=0}^\infty  K^{-j} 
     | e^{-\eps (j+k(t)) \alp} - e^{-\eps (k+k(t)) \alp}|^p \\
     & \simeq e^{-\eps k(t) \alp p} 
        \biggl( \sum_{j=0}^{k-1}  K^{-j} e^{-\eps k \alp p}
          + \sum_{j=k+1}^{\infty}  K^{-j} e^{-\eps j \alp p} \biggr)\\
     & \simeq e^{-\eps k(t) \alp p} 
        ( e^{-\eps k \alp p} + e^{(-\eps  \alp p- \log K)k}) \\
     & \simeq t^{\alp p} e^{-\eps k \alp p}. 
\end{align*}
Thus
\begin{align*}
     \int_{B(\zeta_0,t)} \vint_{B(\zeta,t)} |f(\xi)-f(\zeta)|^p \, d\nu(\xi)
     \, d\nu(\zeta)
     & \simeq t^{\alp p} K^{-k(t)} \sum_{k=0}^\infty K^{-k}  e^{-\eps k \alp p} \\
     & \simeq t^{\alp p} K^{-k(t)} 
     \simeq t^Q t^{\alp p}
\end{align*}
yielding the estimate  \eqref{eq-Ex-AA} also for $\alp <0$.

For 
$\zeta \in \bdy X \setm B(\zeta_0,t)$ we instead see that
$f$ is constant within $B(\zeta,t)$, and thus
\[
     \int_{\bdy X \setm B(\zeta_0,t)}  \vint_{B(\zeta,t)} |f(\zeta)-f(\xi)|^p \, d\nu(\xi)
     \, d\nu(\zeta)
     =0.
\]
Hence $E_p(f,t)^p \simeq t^{Q+\alp p}$ and 
\[
  \| f \|_{B^{\theta}_{p,p}(\partial X)}^p 
      \simeq \int_0^{2e^{\eps}/\eps}  \frac{t^{Q+\alp p}}{t^{\theta p}}  \frac{dt}{t} 
      < \infty
\]
if and only if $\theta < \alp + Q/p$.

Thus if $0 < \theta_1  < \theta_2 $ we can choose 
$\alp \ne 0$ such that $ \theta_1 -Q/p    < \alp < \theta_2 -Q/p$ and obtain  that 
$ f\in B^{\theta_1}_{p,p}(\partial X) \setm  B^{\theta_2}_{p,p}(\partial X)$.
\end{example}

The functional analytic approach to Besov spaces in the classical Euclidean setting, 
using interpolation as in Bennett--Sharpley~\cite{BS1}, as well as the approach to 
Besov spaces using atomic decompositions as
in Triebel~\cite{Tr}, immediately yield the density of continuous functions in the 
classical Besov spaces. Our definition of Besov spaces,
equivalent to that of the interpolation approach of~\cite{BS1} (see
Gogatishvili--Koskela--Shanmugalingam~\cite{GKS}) under the assumption
that the underlying metric space has a doubling measure supporting a \p-Poincar\'e 
inequality, does not on its own imply the density of continuous functions in the 
corresponding Besov space. Note that  Cantor type
sets do not support any Poincar\'e inequality for function-upper gradient pairs. 
However, we will next show that because of the ultrametric structure of the Cantor sets, 
continuous functions  are indeed dense in the Besov space.

\begin{prop}\label{prop-dense-cont}
The set of all Lipschitz continuous functions in $B_{p,p}^\theta(\partial X)$ 
is dense in $B_{p,p}^\theta(\partial X)$.
\end{prop}

\begin{proof}
Let $u\in B_{p,p}^\theta(\partial X)$. We will approximate $u$
by continuous functions on $\partial X$ as follows.
For $n\ge1$, let $\Bni$, $i=1,2,\ldots,K^n$, be the $K^n$ balls of radius
$r_n=2e^{(1-n)\tw}/\tw$, whose union is $\bdry X$.
Note that all these balls have the same measure  
$\nu(\Bni)=K^{-n}\simeq r_n^Q$.
For each $n$, $i$ and $\xi\in\Bni$ let
\[
u_n(\xi) = \vint_\Bxin u\,d\nu = \vint_\Bni u\,d\nu.
\]
The functions $u_n$ are piecewise constant and 
Lipschitz continuous on
$\bdry X$, since $\{\Bni\}_{i=1}^{K^n}$ form  
a pairwise disjoint clopen cover of $\partial X$.
Let $v_n=u_n-u$. We shall show that $\|v_n\|_{\Bppal(\bdry X)}\to0$ 
as $n\to\infty$ (i.e.\ $r_n\to0$).
Note first that $u_n(\z)=u_n(\xi)$ whenever $d_X(\z,\xi)< r_n$.
Hence $|v_n(\z)-v_n(\xi)|=|u(\z)-u(\xi)|$ for such $\z$ and $\xi$
and consequently,
\begin{align*}
I(r_n)
&:= \iintlim{0}{r_n} \int_{\bdry X} \vint_\Bzr \frac{|v_n(\z)-v_n(\xi)|^p}
{r^{\theta p}} \,d\nu(\xi) \,d\nu(\z) \frac{dr}{r} \\ 
&\quad = \iintlim{0}{r_n} \int_{\bdry X} \vint_\Bzr \frac{|u(\z)-u(\xi)|^p}
{r^{\theta p}} \,d\nu(\xi) \,d\nu(\z) \frac{dr}{r} \to 0,
\end{align*}
as $r_n\to0$, by the finiteness of $\|u\|_{\Bppal(\bdry X)}$.
Next, by the Fubini theorem and the fact that $\nu(\Bzr)\simeq r^Q$
we see that
\begin{align}  \label{eq-Fubini-v-n}
II(r_n)
&:= \iintlim{r_n}{\infty}
\int_{\bdry X} \vint_\Bzr \frac{|v_n(\z)-v_n(\xi)|^p}
           {r^{\theta p}}  \,d\nu(\xi) \,d\nu(\z) \frac{dr}{r} \nonumber\\ 
&\simeq \int_{\bdry X} \int_{\bdry X} |v_n(\z)-v_n(\xi)|^p
\int_{r_n}^\infty \frac{\mathbbm{1}_\Bzr(\xi)}{r^{\theta p+Q}} 
       \frac{dr}{r} \,d\nu(\xi) \,d\nu(\z).
\end{align}
Since $\mathbbm{1}_\Bzr(\xi)\ne0$ if and only if $r>d_X(\z,\xi)$,
the last integral becomes
\[
\int_{r_n}^{\infty}
\frac{\mathbbm{1}_\Bzr(\xi)}{r^{\theta p+Q}} \frac{dr}{r} 
= \frac{1}{\theta p+Q} \max\{r_n,d_X(\z,\xi)\}^{-\theta p-Q}
\simeq \frac{1}{(r_n+d_X(\z,\xi))^{\theta p+Q}}.
\]
Inserting this into \eqref{eq-Fubini-v-n} shows that 
\begin{equation}  \label{eq-est-int-rn-infty} 
II(r_n) 
\simeq \int_{\bdry X} \int_{\bdry X} \frac{|v_n(\z)-v_n(\xi)|^p}
  {(r_n+d_X(\z,\xi))^{\theta p+Q}} \,d\nu(\xi) \,d\nu(\z).
\end{equation}
Recall that $v_n=u_n-u$.
To estimate the last integral we first note that by the H\"older and
triangle inequalities,
\begin{align}  
|v_n(\z)-v_n(\xi)|^p 
& = \biggl| \vint_\Bzn (u(\chi)-u(\z))\,d\nu(\chi) 
   - \vint_\Bxin (u(\om)-u(\xi))\,d\nu(\om) \biggr|^p \nonumber \\
& \le 2^{p-1} \biggl( \vint_\Bzn |u(\chi)-u(\z)|^p\,d\nu(\chi) 
   + \vint_\Bxin |u(\om)-u(\xi)|^p\,d\nu(\om) \biggr). \nonumber
\end{align}
Next, note that the roles of $\z$ and $\xi$ above and 
in \eqref{eq-est-int-rn-infty}
are symmetric, so that
\[
II(r_n) \simle \int_{\bdry X} \int_{\bdry X} \vint_\Bzn \frac{|u(\chi)-u(\z)|^p}
  {(r_n+d_X(\z,\xi))^{\theta p+Q}} \,d\nu(\chi) \,d\nu(\xi) \,d\nu(\z).
\]
We next split the middle integral (with respect to $\xi$) 
into integrals over $\Bzn$ and $\bdry X\setm\Bzn$. 
The integral over $\Bzn$ is estimated as 
\begin{align} \label{eq-est-int-Bzn}
&\int_{\bdry X} \int_\Bzn \vint_\Bzn \frac{|u(\chi)-u(\z)|^p}
  {(r_n+d_X(\z,\xi))^{\theta p+Q}} \,d\nu(\chi) \,d\nu(\xi) \,d\nu(\z)
\nonumber\\
&\kern 10em \le \int_{\bdry X} \vint_\Bzn \frac{|u(\chi)-u(\z)|^p}
  {r_n^{\theta p}} \,d\nu(\chi) \,d\nu(\z),
\end{align}
while the  integral over $\bdry X\setm\Bzn$ is split into 
integrals over the $K^n-1$ balls $\Bni$ with radii $r_n$.
Note that $\nu(\Bni)=\nu(\Bzn)=K^{-n}\simeq r_n^Q$
and for each $j=0,\ldots,n-1$ there are $(K-1)K^{j}$ such balls $\Bni$
which have distance $e^{j\tw} r_n$ to $\Bzn$,
i.e.\ $d_X(\z,\xi)= e^{j\tw} r_n$ for $\xi\in\Bni$.
As $\xi$ only appears in the denominator $(r_n+d_X(\z,\xi))^{\theta p+Q}$,
summing up over all these balls $\Bni$ gives
\begin{align} \label{eq-est-int-bdryX-Bzn}
&\int_{\bdry X} \int_{\bdry X\setm\Bzn} \vint_\Bzn \frac{|u(\chi)-u(\z)|^p}
  {(r_n+d_X(\z,\xi))^{\theta p+Q}} \,d\nu(\chi) \,d\nu(\xi) \,d\nu(\z)
\nonumber\\
&\quad\le \int_{\bdry X} \vint_\Bzn |u(\chi)-u(\z)|^p
  \sum_{j=0}^{n-1} \frac{(K-1)K^{j}\nu(\Bzn)}{(e^{j\tw} r_n)^{\theta p+Q}} 
              \,d\nu(\chi) \,d\nu(\z).
\end{align}
Since $K/e^{\tw(\theta p+Q)}=e^{-\tw\theta p}<1$, the sum is majorized by
$C\nu(\Bzn)/r_n^{\theta p+Q} \simeq r_n^{-\theta p}$.
Inserting this into \eqref{eq-est-int-bdryX-Bzn} together with
\eqref{eq-est-int-Bzn} shows that 
\[
II(r_n) \simle \int_{\bdry X} \vint_\Bzn  \frac{|u(\chi)-u(\z)|^p}
     {r_n^{\theta p}} \,d\nu(\chi) \,d\nu(\z)=\biggl( \frac{E_p(u,r_n)}{r_n^\theta} \biggr)^p.
\]
From Lemma~\ref{lem-norm-equiv-sum} and the finiteness 
of $\|u\|_{\Bppal(\bdry X)}$ we now easily conclude that
$II(r_n)\to0$ as $r_n\to0$,
and hence $\|u_n-u\|_{\Bppal(\bdry X)}\to0$ as $n\to\infty$.
Since 
\begin{align}  \label{eq-est-Lp-vn}
\|u_n-u\|^p_{L^p(\bdry X)} &= \int_{\bdry X} |u_n(\z)-u(\z)|^p\,d\nu(\z) \\
&\le \int_{\bdry X} \vint_\Bzn  |u(\chi)-u(\z)|^p \,d\nu(\chi) \,d\nu(\z),
\nonumber
\end{align}
we also see that $\|u_n-u\|^p_{L^p(\bdry X)}\to0$ as $r_n\to0$.
\end{proof}

\section{Traces of Newtonian functions on regular trees}
\label{sect-traceBS}\label{sect-Extensions}

\emph{Recall that in  Sections~\ref{sect-Traces} and~\ref{sect-Extensions}
 we assume that $X$ is a regular $K$-ary tree, $K \ge 2$.}

\medskip

In this section, we consider conditions
under which Newtonian functions on $X$ 
have a Besov trace on $\partial X$, and
when we may extend Besov functions on $\partial X$ to Newtonian functions on $X$.
We shall see that the obtained results are sharp and in 
Theorem~\ref{thm-trace-sharp} we give a precise trace result.

\begin{prop}\label{prop-trace}
Let $X$ be a regular $K$-ary tree with the metric $d_X$ defined 
by the exponential weight as in~\eqref{unif-metric-trees} with $\eps>0$ and the
measure $\mu$ defined by the exponential weight with $\beta>\log K$,
and let $p\ge1$.  Then for every $\theta$ satisfying
\begin{equation} \label{eq-thm-trace1}
    0 <  \theta \le 1- \frac{\be -\log K}{p\tw},
\end{equation}
there is a bounded linear trace operator\/ $\Tr:N^{1,p}(X)\to\Bppal(\bdry X)$ 
such that for $f\in N^{1,p}(X)$,
\begin{align*}
    \|{\Tr f}\|_{L^p(\bdry X)} &\le |f(0)| + C \|g_f\|_{L^p(X)} \quad \text{and} 
            \quad \|{\Tr f} \|_{B^\theta_{p,p}(\bdry X)}\simle \|g_f\|_{L^p(X)}.
\end{align*}
In particular,  $\|{\Tr f} \|_{\tB^\theta_{p,p}(\bdry X)} \simle \|f\|_{\Np(X)}$.
Furthermore, for Lipschitz functions $f: X \to \R$ 
we have that\/ $\Tr f=f\vert_{\partial X}$, where
the continuous extension of $f$ to $\itoverline{X}$ is also denoted by $f$.
\end{prop}

Examples~\ref{ex-trace-1} and~\ref{ex-trace2} below
show that Proposition~\ref{prop-trace} is sharp.

When $p>1$, $0 <  \theta <1- (\be -\log K)/p\tw$ and $1 \le \beta/\eps < Q+1$,
the above result can be deduced from Theorem~6.5 in 
Gogatishvili--Koskela--Shanmugalingam~\cite{GKS}, which deals with
general metric spaces supporting a Poincar\'e inequality, and their Ahlfors regular subsets. 
The borderline case $\theta =1- (\be -\log K)/p\tw$
can however not be obtained from~\cite{GKS}.
To apply Theorem~6.5 from~\cite{GKS}  note first that the
parameter $\gamma$ is, in our setting, given by
$\ga=Q=(\log K )/\eps$,
the Hausdorff dimension of $\bdry X$.
We then find $q<p$ such that $p<sq/(s-Q)(Q+1)$, which is only possible
if $s <Q+1$.
Furthermore, let $\alp$ be such that $\theta < \alp < 1-(s-Q)/p$.
If we choose $q$ above sufficiently close to $p$ we can then find $\la< 1-\alp$ satisfying
the remaining requirements in Theorem~6.5 in \cite{GKS}.
After noting that $p^* > p$,
our result (in the case mentioned above) follows from \eqref{Besov-subset-2}.

For $\beta < \eps$, \cite{GKS} gives less sharp exponents then ours,
and for $p=1$ and $Q+1 \le s:=\max\{1,\be/\eps\}< Q+p$ 
one cannot obtain any embedding from~\cite{GKS}. 
Our range $\beta/\eps < Q+p$ is an improvement on~\cite{GKS}
in our setting, and is sharp by Example~\ref{ex-trace-1}.

Note that in the classical Euclidean setting, 
if a set $K\subset\mathbb{R}^n$ is Ahlfors $d$-regular 
and compact, then the trace 
space (on $K$) of the classical Sobolev space $W^{1,p}(\R^n)$ is the Besov space
$B^\tau_{p,p}(K)$ with $\tau=1-(n-d)/p$.
This result is due to Jonsson--Wallin~\cite{JW80},~\cite{JW84} 
and meshes well with our trace result, since the dimension of
the (ambient) tree is $s=\max\{1,\be/\eps\}$
and the dimension of the boundary is $(\log K)/\eps$.

In Bourdon--Pajot~\cite[Th\'eor\`eme~0.1]{BP03}, the Besov space 
$B^{Q/p}_{p,p}$ on a uniformly perfect Ahlfors $Q$-regular metric space $Z$ is
identified with the $l^p$-cohomology of the conformal gauge of $Z$.
Their arguments utilize Theorem~9.8 from Bonk--Heinonen--Koskela~\cite{BHK}
and thus apply provided that $Z$ is a continuum. In our setting, $\bdry X$ is a totally 
disconnected set. Its cohomology corresponds to the class of functions $f:V\to\R$,
defined on the set $V$ of vertices of $X$, whose differential $df\in l^p(E)$
(where $E$ is the set of edges of $X$). Comparing $df$ with~\eqref{eq-def-gu-linear} 
shows that for every edge $[x,y]\in E$, we have $|df|\simeq e^{\eps|x|}g_f$.
Hence $df\in l^p(E)$ if and only if $g_f\in L^p(X,\mu)$, where $\mu$ is
given by~\eqref{tree-measure} with $\be=p\eps$.
Note that with this choice of parameters we have 
$\theta = Q/p = 1 - (\be/\eps -Q)/p$. Thus, in our setting of regular trees, 
a result corresponding to Th\'eor\`eme~0.1 in~\cite{BP03} follows from 
Theorem~\ref{thm-trace-sharp}.

A direct consequence of Proposition~\ref{prop-trace} and 
Example~\ref{ex-trace-1} below is that for $p>1$
there is a bounded linear trace operator $\Tr:N^{1,p}(X)\to L^p(\bdry X)$ 
if and only if $p>(\be-\log K)/\tw$.
For $p=1$ the same is true except that we
do not know if there is a bounded linear trace operator
if $p=1=(\be-\log K)/\tw$.

\begin{example} \label{ex-trace-1}
Let $f$ be the continuous function on $X$ 
given by $f(x)=\log (|x|+1)$. Note that
the function $g(x)=e^{\eps |x|}/(|x|+1)$ 
is an upper gradient of $f$ on $X$ with respect to
the metric $d_X$, and so
\[
    \int_X g_f^p\,d\mu 
    \simeq \sum_{k=0}^\infty \frac{e^{p\eps k}}{(k+1)^p} K^k e^{-\be k}
    = \sum_{k=0}^\infty \frac{e^{(p\eps-\be + \log K)k}}{(k+1)^p}
\]
which is finite if and only if either $p < ( \be -\log K)/\eps$
or $p = ( \be -\log K)/\eps>1$.
In this case we also see that $f \in L^p(X)$, by the
$(p,p)$-Poincar\'e inequality in Corollary~\ref{cor-pp-PI},
and thus $f \in \Np(X)$.

On the other hand, $f(x) \to \infty$ as $x \to \bdy X$.
So the only reasonable trace $f$ can leave at $\bdy X$ is the function
which is $\infty$ everywhere, which does not belong to any Besov 
or Lebesgue space. The requirement
\[
      p > \frac{\be -\log K}{\eps}
\]
is therefore necessary to be able to have any trace 
result, with the possible exception of a trace result
also when $p = 1=( \be -\log K)/\eps$.
\end{example}

\begin{example} \label{ex-trace2}
In this example we shall show that the range
in \eqref{eq-thm-trace1} is sharp. Assume that 
\[
     \theta > 1- \frac{\be -\log K}{p\tw} .
\]
Then we can find $\ga$ such
that $\max\{\eps(1-\theta),0\} < \ga < \min\{\eps, (\be -\log K)/p\}$.

Let $f$ be the continuous function on $X$ 
defined as follows:
We set $f(0)=0$, and for each vertex $x$ we choose exactly one child 
$c(x)$ of $x$ and set (recursively) $f(c(x))=f(x)+ e^{(\gamma-\eps)|x|}$,
while for all other children $y$  of $x$ we set $f(y)=f(x)$. 
Finally we require $f$ to be  linear
on each of the edges (with respect to the $d_X$ metric),  i.e.,
\[
    f(t)=f(x)+(f(w)-f(x))\frac{d_X(t,x)}{d_X(w,x)}
    \quad  \text{for } t \in [x,w],
\]
for any child $w$ of the vertex $x$.

It follows that the function $g(t)=\eps e^{\ga |x|}/(1-e^{-\eps})$ 
when $t\in[x,c(x)]$ for some vertex $x$, and
$g(t)=0$ for all other values of $t$, is an upper gradient of $f$.
Hence, in a similar way to Example~\ref{ex-trace-1}
we see that $f \in \Np(X)$ (since $p\ga <  \be -\log K$) and that $f$ is a 
bounded continuous function (because $\ga<\eps$)
which has a continuous extension to $\bdy X$ (also denoted by $f$).

For $r < e^{\eps} \diam X$, let $k(r)$ be as in~\eqref{eq-level}. 
Let $B_i$ be any of the $K^{k(r)}$ pairwise disjoint balls with radius $r$,
whose union is $\bdry X$, and let
$z$ be the largest common ancestor of all the points in $B_i$.
Note that $|z|=k(r)$. We define a bijection $H$ between the subtree 
$\{x \in X : x \ge z\}$ and $X$ inductively as follows: 
$H(z)=0$ and for each vertex $x\ge z$, the children of $x$ are mapped to 
the children of $H(x)$ in such a way that  $H(c(x))=c(H(x))$.
Note that this bijection extends to a bijection between $B_i$
and $\bdy X$, and that for $\xi\in B_i$,
\begin{equation}  \label{eq-scaling}
d\nu(\xi) = K^{-k(r)} \,d\nu(H(\xi)).
\end{equation}
Moreover, if $y$ is a child of some $x\ge z$, then 
\[
   f(y)-f(x) = e^{(\gamma-\eps)k(r)} (f(H(x))-f(H(y))).
\]
Thus, for $\z,\xi \in B_i$, we obtain by continuity that
\[
f(\z)-f(\xi) = e^{(\gamma-\eps)k(r)} (f(H(\z))-f(H(\xi))).
\]
Together with \eqref{eq-scaling} this yields
\[
\vint_{B_i} \vint_{B_i} |f(\zeta)-f(\xi)|^p \,d\nu(\xi) \,d\nu(\zeta) 
       =  e ^{p(\ga -\eps)k(r)} \vint_{\bdy X} \vint_{\bdy X} |f(\zeta)-f(\xi)|^p 
                    \,d\nu(\xi) \,d\nu(\zeta),
\]
where the last double integral is clearly positive, finite and independent 
of $r$ and $i$.
Multiplying the last identity by $\nu(B_i)= K^{-k(r)}$
and summing over $i=1,\ldots,K^{k(r)}$ gives, since $B(\z,r)=B_i$ for 
$\z\in B_i$,
\[
    E_p (f,r)^p = \sum_{i=1}^{K^{k(r)}} \int_{B_i} \vint_{B_i} |f(\zeta)-f(\xi)|^p 
        \,d\nu(\xi) \,d\nu(\zeta) 
         \simeq  e ^{p(\ga -\eps)k(r)} \simeq r^{p(\eps-\ga)/\eps}.
\]
Hence, by Remark~\ref{rem-finite-int},
\[
    \|f|_{\bdy X}\|_{B^\theta_{p,p}(\partial X)}^p 
     \simeq   \int_0^{2\diam X} 
   \frac{t^{p(\eps-\ga)/\eps}}{t^{\theta p}} \frac{dt}{t} = \infty
\]
because $ \theta \ge  1-\ga/\eps$.
Thus $f|_{\bdy \Om} \notin B^\theta_{p,p}(\bdy X)$.
\end{example}

\begin{proof} [Proof of Proposition~\ref{prop-trace}]
Let $f\in\Np(X)$. We shall first show that for $\nu$-a.e.\  $\z\in\bdry X$, the limit
\begin{equation}  \label{eq-def-trace}
   \ft(\z)= \lim_{[0,\z)\ni x\to\z} f(x),
\end{equation}
taken along the geodesic ray $[0,\z)$, 
exists and defines the trace $\Tr f:=\ft:\bdry X\to\R$, with norm estimates. (Note that 
if $f$ is Lipschitz, then $\ft(\z)= \lim_{x\to\z} f(x)$  for every $\z\in\bdry X$.)

To this end, let $\z\in\bdry X$ be arbitrary and
let $x_j=x_j(\z)$ be the ancestor of $\z$ with $|x_j|=j$.
Set $r_j=2e^{-j\eps}/\eps$.
Recall that on the edge $[x_j,x_{j+1}]$ we have
$ds \simeq e^{(\be-\eps)j}\,d\mu \simeq r_j^{1-\be/\eps}\,d\mu$.
Fix $n\in\N$ and let $m\ge n$ be arbitrary. Then
\begin{align*} 
|f(x_m)-f(x_n)| &\le \sum_{j=n}^{m-1} |f(x_{j+1})-f(x_j)| \\
    &\le \sum_{j=n}^\infty \int_{[x_j,x_{j+1}]} g_f \,ds
    \simeq \sum_{j=n}^\infty r_j^{1-\be/\eps} \int_{[x_j,x_{j+1}]} g_f \,d\mu. \nonumber
\end{align*}
We shall now show that, independently of $m$, this tends to zero
as $n\to\infty$, for $\nu$-a.e.\ $\z\in\bdry X$. 
Thus, the sequence $\{f(x_n)\}_{n=0}^\infty$ is a Cauchy sequence, and has
a limit as $n\to\infty$, for $\nu$-a.e.\ $\z\in\bdry X$.

Choose $0<\ka<\theta p$ and insert $r_j^{\ka/p} r_j^{-\ka/p}$ into the above sum.
If $p>1$, then H\"older's inequality applied first to the integral and then to the 
sum, together with the estimate $\mu([x_j,x_{j+1}])\simeq r_j^{\be/\eps}$,
implies that 
\begin{align*}   
\sum_{j=n}^\infty r_j^{1-\be/\eps} \int_{[x_j,x_{j+1}]} g_f \,d\mu 
&\le \sum_{j=n}^\infty r_j^{\ka/p} r_j^{1-\be/\eps p-\ka/p} \biggl( \int_{[x_j,x_{j+1}]} g_f^p 
           \,d\mu \biggr)^{1/p} \nonumber \\
&\simle  r_n^{\ka/p} \biggl( \sum_{j=n}^\infty r_j^{p-\be/\eps-\ka} 
     \int_{[x_j,x_{j+1}]} g_f^p\,d\mu \biggr)^{1/p},
\end{align*}
since $r_j=r_n e^{(n-j)\tw}$, which gives the convergent sum 
$\sum_{j=n}^\infty r_j^{\ka/(p-1)} \simeq r_n^{\ka/(p-1)}$.
(For $p=1$ the estimate is simpler and H\"older's inequality is not needed.)
It follows that
\begin{equation}   \label{eq-f-xm-f-xn}
|f(x_m(\z))-f(x_n(\z))|^p 
\simle r_n^{\ka} \sum_{j=n}^\infty r_j^{p-\be/\eps-\ka} 
     \int_{[x_j,x_{j+1}]} g_f^p\,d\mu.
\end{equation}
Integrating over all $\z\in\bdry X$ we obtain by means of Fubini's theorem,
\begin{align*}  
&\int_{\bdry X} |f(x_m(\z))-f(x_n(\z))|^p \,d\nu(\z) \\
&\quad \quad \quad \simle r_n^{\ka} \int_{\bdy X} \sum_{j=n}^\infty r_j^{p-\be/\eps-\ka} 
        \int_{[x_j,x_{j+1}]} g_f^p\,d\mu \,d\nu(\zeta) \\
&\quad \quad \quad = r_n^{\ka} \int_{X} g_f(x)^p 
        \int_{\bdy X} \sum_{j=n}^\infty r_j^{p-\be/\eps-\ka} \mathbbm{1}_{[x_j,x_{j+1}]}(x)
        \,d\nu(\zeta)\,d\mu(x).
\end{align*}
Note that $\mathbbm{1}_{[x_j,x_{j+1}]}(x)$ is nonzero only if 
$j\le|x|\le j+1$ and $x<\z$. Thus, the last estimate can be written as
\[
   \int_{\bdry X} |f(x_m(\z))-f(x_n(\z))|^p \,d\nu(\z) 
       \simle r_n^{\ka}  \int_{X} g_f(x)^p r_{j(x)}^{p-\be/\eps-\kappa}
         \nu(E(x))\, d\mu(x),
\]
where 
$E(x)=\{\z \in \bdy X : \z > x\}$ and 
$j(x)$ is the largest integer such that $j(x)\le|x|$.
Since $\nu(E(x))\simle r_{j(x)}^Q$  
and $p-\be/\eps-\kappa+Q >0$, we obtain that
\begin{align*}
\int_{\bdry X} |f(x_m(\z))-f(x_n(\z))|^p \,d\nu(\z)  
 &\simle r_n^{\ka}  \int_{X} g_f(x)^p r_{j(x)}^{p-\be/\eps-\kappa+Q}
                \, d\mu(x) \\
&\simle r_n^{\ka}  \int_{X} g_f^p \, d\mu \to 0, \quad \text{as }n\to\infty.
\end{align*}
Hence, the sequence $\{f(x_n)\}_{n=0}^\infty$ is a Cauchy sequence for $\nu$-a.e.\ 
$\z\in\bdry X$, and has a limit as $n\to\infty$.
This also shows (by letting $n=0$ and $f(x_m(\z))\to\ft(\z)$) that
\[
\int_{\bdry X} |\ft(\z)-f(0)|^p \,d\nu(\z) \simle \int_X g_f^p \,d\mu,
\]
and thus
$\|\ft\|_{L^p(\bdry X)} \le |f(0)| + C \|g_f\|_{L^p(X)}$.

To estimate $\|\ft\|_{\Bppal(\bdry X)}$, we let
$m\to\infty$ in~\eqref{eq-f-xm-f-xn} to obtain for
$\nu$-a.e.\ $\z\in\bdry X$,
\begin{equation*}   
|\ft(\z)-f(x_n)|^p 
\simle r_n^{\ka} \sum_{j=n}^\infty r_j^{p-\be/\eps-\ka} 
     \int_{[x_j,x_{j+1}]} g_f^p\,d\mu.
\end{equation*}
A similar estimate holds for $\nu$-a.e.\ $\xi\in \bdry X$ such that $d_X(\z,\xi)=r_n$,
with the $x_j$'s replaced by the ancestors $y_j$ of $\xi$.
Note that $x_n=y_n$ for such $\xi$. It follows that
\begin{equation*}   
|\ft(\xi)-\ft(\z)|^p \simle
   r_n^{\ka} \sum_{j=n}^\infty r_j^{p-\be/\eps-\ka} \biggl( 
        \int_{[x_j,x_{j+1}]} g_f^p\,d\mu + \int_{[y_j,y_{j+1}]} g_f^p\,d\mu \biggr),
\end{equation*}
where $n=n(\z,\xi)\approx -\log(\eps d_X(\z,\xi)/2)$  
denotes the level of the largest common ancestor of 
$\z$ and $\xi$. 
Inserting this into \eqref{eq-equiv-norm} yields
\begin{align*}  
\|\ft\|^p_{\Bppal(\bdry X)}
&\simle \int_{\bdy X}\int_{\bdy X} \frac{r_n^{\ka}}{d_X(\zeta,\xi)^{\theta p +Q}} \\
&\quad  \times \sum_{j=n}^\infty r_j^{p-\be/\eps-\ka} \biggl( 
        \int_{[x_j,x_{j+1}]} g_f^p\,d\mu + \int_{[y_j,y_{j+1}]} g_f^p\,d\mu \biggr)
        \,d\nu(\xi)\,d\nu(\zeta),
\end{align*}
where again $n=n(\z,\xi)$ depends on $\z$ and $\xi$.
Because the roles of $\z$ and $\xi$ in the last 
formula are symmetric, it suffices to estimate the expression with
the integral over $[x_j,x_{j+1}]$.
By writing $\bdry X= \bigcup_{n=0}^\infty A_n$, where
$A_n=\{\xi\in\bdy X: d_X(\xi,\z)=r_n\}$, 
together with the estimate $\nu(A_n)\simle r_n^Q$, this leads to
\begin{align*}   
\|\ft\|^p_{\Bppal(\bdry X)} 
&\simle  \int_{\bdy X} \sum_{n=0}^\infty r_n^{-\theta p-Q+\ka} \int_{A_n} \sum_{j=n}^\infty 
  r_j^{p-\be/\eps-\ka} \int_{[x_j,x_{j+1}]} g_f^p \,d\mu \, d\nu(\xi) \,d\nu(\z) 
\nonumber \\
&\simle \int_{\bdy X} \sum_{n=0}^\infty r_n^{-\theta p+\ka} \sum_{j=n}^\infty 
  r_j^{p-\be/\eps-\ka} \int_X g_f(x)^p \mathbbm{1}_{[x_j,x_{j+1}]}(x) \,d\mu(x)
  \,d\nu(\z), 
\end{align*}
where the edge $[x_j,x_{j+1}]$ belongs to the geodesic ray connecting the root
$0$ to $\z$.  Note that $\mathbbm{1}_{[x_j,x_{j+1}]}(x)$ is nonzero only if 
$n\le j\le|x|\le j+1$ and $x<\z$. Thus, the last estimate can be written as
\begin{align*}
&\|\ft\|^p_{\Bppal(\bdry X)} \\
 &\quad \simle \int_{\bdy X}\sum_{n=0}^\infty r_n^{-\theta p+\kappa}
     \int_X g_f(x)^p r_{j(x)}^{p-\be/\eps-\kappa} \mathbbm{1}_{\{z\in X:\z>z\}}(x)
        \mathbbm{1}_{\{z\in X:|z|\ge n\}}(x) \, d\mu(x)\, d\nu(\z)\\
  &\quad \simle\int_{\bdy X}\int_X 
   g_f(x)^p r_{j(x)}^{p-\be/\eps-\kappa} \mathbbm{1}_{\{z\in X:\z>z\}}(x)
    \sum_{n=0}^{j(x)}r_n^{-\theta p+\kappa}\, d\mu(x)\, d\nu(\z),
\end{align*}
where $j(x)$ is the largest integer such that $j(x)\le|x|$. Using the Fubini theorem we then get
\begin{equation*}  
\|\ft\|^p_{\Bppal(\bdry X)} \simle
\int_X g_f(x)^p  r_{j(x)}^{p-\be/\eps-\ka} 
     \nu(E(x))
    \sum_{n=0}^{j(x)} r_n^{-\theta p+\ka} \,d\mu(x).
\end{equation*}
The last sum is comparable to $r_{j(x)}^{-\theta p+\ka}$,
by the choice $\ka<\theta p$.
Since  $\nu(E(x))\simle r_{j(x)}^Q$  
and $p-\be/\eps-\theta p+Q \ge0$, we finally obtain that
\[
\|\ft\|^p_{\Bppal(\bdry X)} \simle
\int_X g_f(x)^p r_{j(x)}^{p-\be/\eps-\theta p+Q} \,d\mu(x) 
\simle \int_X g_f^p \,d\mu.\qedhere
\]
\end{proof}

We next show that
Besov functions on the boundary 
Cantor set can be extended to  Newtonian functions on 
the regular $K$-ary tree.  

\begin{prop}\label{prop-ext}
Let $X$ be a regular $K$-ary  tree with 
the metric $d_X$ defined by the exponential weight 
as in~\eqref{unif-metric-trees}
with $\tw>0$ and the measure $\mu$ defined by the
exponential weight with $\beta>\log K$, and let $p\ge1$. Suppose that 
\begin{equation} \label{eq-alp-cond}
  \theta \geq 1- \frac{\beta - \log K}{p \tw }
  \quad \text{and} \quad \theta >0.
\end{equation}
Then there is a bounded linear extension operator 
\[
  \Ext:B^\theta_{p,p}(\partial X) \longrightarrow N^{1,p}(X),
\]  
such that for $u\in B^\theta_{p,p}(\bdy X)$, we have\/ 
$\Tr (\Ext u)=u$ $\nu$-a.e., 
where\/ $\Tr$ is the trace operator 
constructed in~\eqref{eq-def-trace}.  Furthermore, for
$\nu$-a.e.~$\z\in\bdy X$ and a geodesic $\gamma$  in $X$ terminating
at $\z$, we have\/ $\lim_{t\to\infty}\Ext u(\gamma(t))=u(\z)$.
Moreover, with $\ut=\Ext u$, we have
\begin{align*}
\|g_{\ut}\|_{L^p(X)} &\simle \|u\|_{B^\theta_{p,p}(\partial X)},\\
   \|\ut\|_{\Np(X)} 
        &\simle \|u\|_{L^p(\bdy X)}  
            + \|u\|_{B^\theta_{p,p}(\partial X)}
            = \|u\|_{\Bt^\theta_{p,p}(\partial X)}.
\end{align*}
\end{prop}

In view of Proposition~\ref{prop-trace}, the range in~\eqref{eq-alp-cond}
is sharp, cf.\ Theorem~\ref{thm-trace-sharp}.
Note also that if $u \in B^\theta_{p,p}(\bdy X)$ is  continuous 
then we have $\Tr({\Ext u})=u$ everywhere.

\begin{proof}  
Let $u\in\Bppal(\bdry X)$. For $x\in X$ with $|x|=n\in\N:=\{0,1,2,\ldots\}$ let
\begin{equation}   \label{eq-def-ext}
\ut(x) = \vint_{B(\z,r_n)} u\,d\nu,
\end{equation}
where $r_n=2e^{(1-n)\tw}/\tw$ is as in the 
proof of Proposition~\ref{prop-dense-cont}
and $\z\in\bdry X$ is any descendant of $x$.
Observe that $B(\z,r_n)$ consists of all the points
in $\partial X$ that have $x$ as an ancestor, that is, the geodesics connecting
the root $0$ to these points pass through $x$.
Note that $\ut(x)=u_n(\z)$, where $u_n$ is the piecewise constant
(and continuous) approximation of $u$ from Proposition~\ref{prop-dense-cont}.
Moreover,~\eqref{eq-def-trace} and~\eqref{eq-def-ext} imply that
$\Tr\ut(\z)=u(\z)$ whenever $\z\in\bdry X$ is a Lebesgue point of $u$.

If $y$ is a child of $x$, extend $\ut$ to the edge $[x,y]$ as follows:
First, we choose $\z\in\partial X$ so that
$\z$ is a descendant of $y$ as well. 
We can do this because the ultrametric property
of $\partial X$ tells us that every point in the ball $B(\z,r_n)$ is a center of this ball,
see Lemma~\ref{lem-ultra-metric-center}. 
For  each $t\in[x,y]$ set
\begin{equation}   \label{eq-def-gu-linear}
\gut(t)  = \frac{|\ut(x)-\ut(y)|}{d_X(x,y)}        
     = \frac{\tw|u_n(\z)-u_{n+1}(\z)|}{(1-e^{-\tw})e^{-\tw n}}
\end{equation} 
and $\ut(t)=\ut(x)+ \gut(t) \, d_X(x,t)$,
i.e.\ $\gut$ is constant and $\ut$ is linear 
(with respect to the metric $d_X$) on the edge $[x,y]$.
It follows that $\gut$ is a minimal upper gradient of $\ut$ on the edge $[x,y]$.
The contribution from this edge to $\int_X\gut^p\,d\mu$ is
\begin{align}  \label{eq-contrib-edge}
\int_{[x,y]}\gut^p\,d\mu &\simeq \int_n^{n+1} 
   \biggl( \frac{|u_n(\z)-u_{n+1}(\z)|}{e^{-\tw n}} \biggr)^p
      e^{-\be\tau}\,d\tau 
\simeq e^{(\tw p-\be)n} |u_n(\z)-u_{n+1}(\z)|^p.
\end{align}
Here, the choice of $\z\in\partial X$ is dictated 
by the child $y$ of $x$,  but $\z$ can be replaced by any choice of
$\xi\in B(\z,r_{n+1})$. Integrating~\eqref{eq-contrib-edge}
over this smaller ball, we obtain
\[
  \nu(B(\z,r_{n+1})) \int_{[x,y]}\gut^p\, d\mu
    \simeq e^{(\eps p-\beta)n}\int_{B(\z,r_{n+1})}|u_n(\xi)-u_{n+1}(\xi)|^p\, d\nu(\xi).
\]
Summing over all the edges in $X$ connecting vertices at level $n$ to vertices at level $n+1$
shows that the contribution to $\int_X\gut^p\, d\mu$ from all such edges is comparable to
\begin{equation*}     
e^{(\tw p-\beta)n} \int_{\bdry X} 
       \frac{|u_n(\z)-u_{n+1}(\z)|^p}{\nu(B(\z,r_{n+1}))}\,d\nu(\z).
\end{equation*}
Summing over all $n\in\N$ and writing
$|u_n(\z)-u_{n+1}(\z)| \le |v_n(\z)|+|v_{n+1}(\z)|$, where
$v_n=u_n-u$, we obtain from \eqref{eq-est-Lp-vn} that
\begin{align}   \label{eq-from-vn-to-Ep}
\int_X\gut^p\,d\mu 
   &\simle \sum_{n=0}^\infty e^{(\tw p-\beta)n} \int_{\bdry X} 
   \frac{(|v_n(\z)|+|v_{n+1}(\z)|)^p}{\nu(B(\z,r_{n+1}))} \,d\nu(\z)\nonumber\\
&\simle \sum_{n=0}^\infty \frac{e^{(\tw p-\beta)n}}{r_n^Q} 
   \int_{\bdry X} \vint_\Bzn  |u(\chi)-u(\z)|^p \,d\nu(\chi) \,d\nu(\z) \nonumber\\
&\simeq \sum_{n=0}^\infty \frac{e^{(\tw p-\beta)n} r_n^{\theta p}}{r_n^Q} 
    \biggl( \frac{E_p(u,r_n)}{r_n^{\theta}} \biggr)^p.
\end{align}
Since $r_n\simeq e^{-n\tw}$, Lemma~\ref{lem-norm-equiv-sum} shows that
the sum converges provided that
\[
e^{(\tw p-\beta)n} r_n^{\theta p-Q} \simeq e^{(\tw p-\beta-\tw(\theta p-Q))n} \le C
\]
for all $n=0,1,\ldots$\,, i.e.\ if $\tw p-\be-\tw(\theta p-Q)\le0$.
This is equivalent to $\theta\ge 1+Q/p-\be/\tw p$, 
which is~\eqref{eq-alp-cond}.
Thus $\|\gut\|_{L^p(X)}\simle\|u\|_{\Bppal(\bdry X)}$ for such $\theta$.

The finiteness of $\|\ut\|_{L^p(X)}$ now follows from the 
$(p,p)$-Poincar\'e inequality (Corollary~\ref{cor-pp-PI}),
since $\ut\in L^1(X)$ by the construction of $\ut$.
\end{proof}

Combining Propositions~\ref{prop-trace} and~\ref{prop-ext}
we obtain the following theorem identifying the trace space.

\begin{thm}   \label{thm-trace-sharp}
Let $X$ be a regular $K$-ary tree with the metric $d_X$ defined 
by the exponential weight as in~\eqref{unif-metric-trees} with $\eps>0$ and the
measure $\mu$ defined by the exponential weight with $\beta>\log K$,
and let $p\ge1$.  Then
the trace space of $N^{1,p}(X)$ is the Besov space $B^\theta_{p,p}(\bdy X)$,
where $\theta=1-(\beta-\log K)/p\eps$, with equivalent norms.
\end{thm}

The following embedding result follows from Proposition~\ref{prop-ext} 
by means of an embedding theorem for Newtonian spaces.

In the Euclidean setting, it is well known that the Besov spaces 
$\Bppal(\R^n)$ and $\Bppal(\Omega)$ for smooth Euclidean
subdomains $\Omega$ embed continuously into $C^\al$ provided that
$\theta=\al+n/p$, see e.g.\ Triebel~\cite[Section~2.7.1]{Tr}.
Our result extends this to regular Cantor type sets.

\begin{prop}   \label{prop-emb-Holder}
Let $X$ be a regular $K$-ary tree with the metric $d_X$ 
defined as in~\eqref{unif-metric-trees} by the exponential weight with $\eps>0$.
Let $Q=(\log K)/\eps$ be the Hausdorff dimension of $\bdry X$ 
and $p>1$. Then $\Bppal(\bdry X)\subset C^\al(\bdry X)$, in the sense that 
every $u\in\Bppal(\bdry X)$ has an $\al$-H\"older continuous
representative, provided that one of the following  conditions holds\/\textup{:}
\begin{enumerate}
\item \label{it-1-1/p} $Q<1$, $\displaystyle \theta \ge \frac{Q-1}{p}+1$  
and $\displaystyle \al=1-\frac{1}{p}>0${\rm;} 
\item \label{it-theta-1/p} $\displaystyle \frac{Q}{p}<\theta\le \frac{Q-1}{p}+1$,  
$\theta<1$ and $\displaystyle \al=\theta-\frac{Q}{p}${\rm;} 
\item \label{it-1-Q/p} $Q\ge1$, $\theta\ge1$ and\/ $\displaystyle 0<\al<1-\frac{Q}{p}$.
\end{enumerate}
\end{prop}
 
\begin{proof}
\ref{it-1-1/p} and~\ref{it-theta-1/p}
Equip $X$ with the measure $d\mu=e^{-\be|x|}d|x|$ as in~\eqref{tree-measure},
where $\be=\eps$ in~\ref{it-1-1/p} and $\be=Q\eps + p\eps(1-\theta)  \ge \eps$
in~\ref{it-theta-1/p}.  Note that in both cases, 
$\theta\ge1-(\be-\log K)/p\eps$
and $\be > \log K$.
Proposition~\ref{prop-ext} then shows that every $u\in\Bppal(\bdry X)$ can be
extended to $U\in\Np(X)$ so that for $\nu$-a.e.\ $\z\in\bdry X$,
$U(x)\to u(\z)$ as $x\to\z$ along a geodesic. By Corollary~\ref{cor-dim-s}, $\mu$ 
satisfies~\eqref{eq-dim-cond} with $s=\be/\eps\ge1$.
Note also that $p>s$, both in \ref{it-1-1/p} and \ref{it-theta-1/p}. 
It follows that $\Np(X)$ embeds continuously into $C^{1-s/p}(X)$
with respect to the metric $d_X$, by e.g.\ Theorem~5.1 in  Haj\l asz--Koskela~\cite{HaKo}.
Hence, the function $u^*(\z):=\lim_{x\to\z} U(x)$ is well defined and
$(1-s/p)$-H\"older continuous on $\bdry X$.
Since $1-s/p=\al$ and $u^*=u$ $\nu$-a.e., the result follows in this case.

\ref{it-1-Q/p}
Let $\tau=\al+Q/p$ and note that $Q/p<\tau\le Q/p+1-1/p$, $\tau<1$ and
$\al=\tau-Q/p$, i.e.\ \ref{it-theta-1/p} holds with $\theta$ replaced by $\tau$.
The already proved part~\ref{it-theta-1/p} together with \eqref{Besov-subset}
then gives the result.
\end{proof}

\section{Quasisymmetric mappings}
\label{QSvsBesov}

In this section we begin a discussion of quasisymmetric mappings between
boundaries of two trees.
We will assume that the boundaries are uniformly perfect and equipped
with Ahlfors regular measures.
We shall show that if $f:\partial X\to\partial Y$ is a quasisymmetric
mapping satisfying a certain dimension condition, then for every $p\ge 1$ and 
$\theta>0$ there is a bounded induced map 
$f_\#:B_{p,p}^\theta(\partial Y)\to B_{p,p}^\tau(\partial X)$ for some $\tau>0$. 
There are already indications
in the current literature that certain Besov spaces may be invariant under 
quasisymmetries; see Bourdon--Pajot~\cite{BP03}. Hence this
result is natural.

Primary examples for these investigations come from regular trees, 
but regularity is not strictly required.
All we need is uniform perfectness (or rather the fact that the mapping
under consideration is a power map) and Ahlfors regularity.

We therefore first formulate and prove our results for Besov spaces on 
general metric spaces and then specialize them to boundaries of trees. 
Recall that a metric space $Z$ is \emph{uniformly perfect} if there is a constant $C>1$
such that whenever $z\in Z$ and $0<r<\diam Z$, the annulus 
$B(z,r)\setminus B(z,r/C)$  is nonempty.
A measure $\nu$ on $Z$ is \emph{Ahlfors $Q$-regular} if for every $z\in Z$
and $0<r<2\diam Z$ we have $\nu(B(z,r))\simeq r^Q$.

A homeomorphism $f:Z\to W$ between two metric spaces $(Z,d_Z)$ and 
$(W,d_W)$ is \emph{quasisymmetric} if there is
a homeomorphism $\eta:[0,\infty)\to[0,\infty)$ such that whenever $x,y,z\in Z$
and $x\ne z$, then
\begin{equation}  \label{def-quasisym}
   \frac{d_W(f(x),f(y))}{d_W(f(x),f(z))} \le \eta\biggl(\frac{d_Z(x,y)}{d_Z(x,z)}\biggr).
\end{equation}

If $Z$ and $W$ are uniformly perfect then every
quasisymmetric map between them has to be 
a power quasisymmetric map, see Heinonen~\cite[Theorem~11.3]{Hei}. 
This means that $\eta$ can be chosen to be of the form
\begin{equation}
\label{qseta}
\eta (t) = \begin{cases}
A t^\alpha, & \text{if } t \leq 1,\\
A t ^{1/\alpha}, & \text{if } t \geq 1,
\end{cases}
\end{equation} 
for some $0<\alpha\le 1$ and $A>0$.
Interchanging the role of $y$ and $z$ in the definition~\eqref{def-quasisym} of 
quasisymmetry gives a lower bound with the function $\phi(t)=1/\eta(1/t)$.
Thus we have the following pair of inequalities for all $x,y,z\in Z$ with $x\ne z$:
\begin{equation*}
\phi \biggl( \frac{ d_Z (x,y)}{ d_Z (x,z)} \biggr) 
\leq \frac{ d_W (f(x), f(y))}{ d_W (f(x), f(z))}
 \leq  \eta \biggl( \frac{ d_Z (x,y)}{ d_Z (x,z)} \biggr).
\end{equation*}
It is easily verified that~\eqref{qseta} yields
\begin{equation} \label{qsphi}
\phi(t) =  \begin{cases}
  A^{-1}t^{1/\al}, & \text{if } t \le1, \\
  A^{-1}t^\al, & \text{if }  t \ge1,
\end{cases}
\end{equation}
where the constant $A$ is the one associated with $\eta$ in~\eqref{qseta}.
Note that the constant 1 in~\eqref{qseta} and \eqref{qsphi} can easily
be replaced by any other positive constant, by changing the constant $A$.

If $Z$ is bounded,  then for $x\in Z$ we can choose $z\in Z$
such that $d_Z(x,z)\ge \tfrac13\diam Z$, and from such a choice of $z$ we get
\begin{equation*}  
   C_1 d_Z(x,y)^{1/\alpha}\le  d_W(f(x),f(y))
        \le  C_2 d_Z(x,y)^{\alpha},
\end{equation*}
i.e.\ both $f$ and $f^{-1}$ are $\al$-H\"older.
 
We say that a bijective map $f:Z\to W$ is $(\al_1,\al_2)$-\emph{biH\"older}
if there are constants $C_1,C_2>0$ such that for all $x,y\in Z$,
\begin{equation}  \label{def-new-biHolder}
   C_1 d_Z(x,y)^{\al_1}\le  d_W(f(x),f(y))
        \le  C_2 d_Z(x,y)^{\alpha_2}.
\end{equation}
If $\al_1=1/\al$ and $\al_2=\al$, we say that $f$ is $\al$-\emph{biH\"older}.

Clearly, every $\al$-biH\"older map is $(1/\al,\al)$-biH\"older, but there 
may be better constants $\al_1$ and $\al_2$ for which~\eqref{def-new-biHolder}
holds. Conversely, if $Z$ is bounded then  every $(\al_1,\al_2)$-biH\"older map is 
$\al$-biH\"older with $\al=\min\{1/\al_1,\al_2\}$.

\begin{lem}\label{L1-isom}
Assume that $Z$ and $W$ are bounded metric spaces
equipped with an Ahlfors $Q_Z$-regular measure $\nu_Z$
and an Ahlfors $Q_W$-regular measure $\nu_W$ respectively.
Suppose that $f:Z\to W$ is an\/ 
$(\al_1,\al_2)$-biH\"older homeomorphism  such that 
\(
  Q_Z\ge \al_1 Q_W.
\)
Then the map $f$ induces a bounded embedding $f_\#:L^p(W)\to L^p(Z)$
for $p\ge1$, via composition.
\end{lem}

The proof below does not need $f$ to be biH\"older,
it is 
sufficient to require that $f^{-1}$ is a 
$1/\alp_1$-H\"older continuous homeomorphism.

\begin{proof} 
 Since  $Z$ and $W$ are separable, continuous functions are
dense in $L^p(Z)$ and $L^p(W)$. Let $u\in L^p(W)$ be a continuous function. 
Then $\int_{W}|u|^p\, d\nu_W$ and $\int_{Z} |u\circ f|^p\, d\nu_Z$ can be computed using Riemann sums.
 Cover $Z$ by  balls $B_i=B_Z(z_i,r)$ with common radius  
$r<1$ so that the balls $B_Z\bigl(z_i,\tfrac12r\bigr)$ are pairwise disjoint and
 \[
   \int_{Z}|u\circ f|^p\, d\nu_Z \simeq \sum_i |u\circ f(z_i)|^p \nu_Z(B_i)
 \]
 for some choice $z_i\in B_i$. 
The Ahlfors regularity condition then implies that the balls $\{B_i\}_i$
have bounded overlap with a bound independent of $r$.
Since $f$ is $(\alpha_1,\alpha_2)$-biH\"older continuous, 
we know that $B_W(f(z_i),Cr^{\alpha_1}) \subset f(B_i)$ and hence
\[
\nu_Z(B_i) \simeq r^{Q_Z} \le r^{\al_1 Q_W} 
\simle \nu_W(f(B_i)).
\]
Thus, we have
 \[
    \int_{Z}|u\circ f|^p\, d\nu_Z 
                 \simle \sum_i |u\circ f(z_i)|^p \nu_W(f(B_i)).
 \]
As $f$ is a homeomorphism and $\{B_i\}_i$ have bounded overlap,
so do $\{f(B_i)\}_i$, and letting $r\to 0$ we obtain
 \[
\int_{Z}|u\circ f|^p\, d\nu_Z
   \simle \limsup_{r\to 0}\sum_i |u\circ f(z_i)|^p \nu_W(f(B_i))
       \simle  \int_{W} |u|^p\, d\nu_W.
 \] 
By the density of continuous functions in the corresponding $L^p$-spaces, 
we see that $f_\#$ boundedly embeds $L^p(Z)$ into $L^p(W)$. 
\end{proof}

The following proposition is a simple consequence of 
Lemma~\ref{L1-isom} and provides us with embeddings between Besov spaces.

\begin{prop}  \label{prop-Besov-to-Besov-with-Q}
Assume that $Z$ and $W$ are bounded metric spaces
equipped with an Ahlfors $Q_Z$-regular measure $\nu_Z$ and
an Ahlfors $Q_W$-regular measure $\nu_W$ respectively. Let $f:Z\to W$ be 
an\/ $(\al_1,\al_2)$-biH\"older homeomorphism such that 
\begin{equation} \label{eq-QZ}
   Q_Z\ge \al_1 Q_W,
\end{equation} 
and $\theta, \tau>0$ and  $p\ge1$ be such that 
 \begin{equation} \label{eq-tau}
    \tau\le \al_2 \theta+\frac{\al_2 Q_W-Q_Z}{p}.
 \end{equation}
 Then $f$ induces a canonical bounded embedding
 \(
     f_\#:B^{\theta}_{p,p}(W)\to B^{\tau}_{p,p}(Z)
 \)
 via composition.
\end{prop}

Proposition~\ref{prop-Besov-to-Besov-with-Q} can immediately be applied to
the tree boundaries $Z=\bdry X$ and $W=\bdry Y$, considered in the previous
sections, provided that  each vertex in $X$ and $Y$ has at least two 
children (so that the boundaries are uniformly perfect) and the boundaries
$\bdry X$ and $\bdry Y$ are Ahlfors regular.
The Ahlfors regularity is guaranteed e.g.\ if $X$ and $Y$ are regular
trees, but can hold also in less regular situations.
The metrics on $\bdry X$ and $\bdry Y$ are the visual metrics given 
by~\eqref{bdy-metric}.
Such metrics can be defined also for nonregular trees.

An $(\al_1,\al_2)$-biH\"older homeomorphism
$f:\bdry X\to\bdry Y$ thus induces
a canonical bounded embedding 
$f_\#:B^{\theta}_{p,p}(\bdry Y)\to B^{\tau}_{p,p}(\bdry X)$,
provided that $Q_X\ge \al_1 Q_Y$ and
\[
    \tau\le \al_2\theta+\frac{\al_2 Q_Y-Q_X}{p}.
\]
In particular, this applies to quasisymmetric mappings between the boundaries
with $\al_1=\al^{-1}$ and $\al_2=\al$ for some $0<\al\le 1$.

\begin{proof}[Proof of Proposition~\ref{prop-Besov-to-Besov-with-Q}] 
Suppose that $u\in B^\theta_{p,p}(W)$, and let $v=u\circ f$.
Also let $\theta_0=\theta +Q_{W}/p$ and $\tau_0=\tau+Q_{Z}/p$. 
For $x,z\in Z$, we have by the H\"older continuity of $f$
and the fact that $\tau_0\le \alpha_2\theta_0$,
\begin{align}  \label{eq-tau-0-theta-0}
  \frac{|v(z)-v(x)|}{d_{Z}(z,x)^{\tau_0}}
&=\frac{|v(z)-v(x)|}{d_{W}(f(z),f(x))^{\theta_0}} 
         \frac{d_{W}(f(z),f(x))^{\theta_0}}{d_{Z}(z,x)^{\tau_0}}\\
    &\simle \frac{|v(z)-v(x)|}{d_{W}(f(z),f(x))^{\theta_0}} 
            d_{Z}(z,x)^{\alpha_2 \theta_0 - \tau_0}
 \simle \frac{|u\circ f(z)-u\circ f(x)|}{d_{W}(f(z),f(x))^{\theta_0}}.
\nonumber
\end{align}
By the second part of  Lemma~\ref{lem-norm-equiv-sum}, we see that
\begin{align*}
   \Vert v\Vert_{B_{p,p}^{\tau}(Z)}^p &\simeq 
    \int_{Z}\int_{Z} \biggl(\frac{|v(z)-v(x)|}
   {d_{Z}(z,x)^{\tau_0}}\biggr)^p \,d\nu_Z(z)\, d\nu_{Z}(x).
\end{align*}
Inequality \eqref{eq-tau-0-theta-0} and Lemma~\ref{L1-isom} then yield,
\begin{align*}
\Vert v\Vert_{B_{p,p}^{\tau}(Z)}^p
         &\simle \int_{Z}
              \int_{Z} \biggl(\frac{|u\circ f(z)-u\circ f(x)|}
   {d_{W}(f(z),f(x))^{\theta_0}}\biggr)^p \,d\nu_Z(z)\,d\nu_Z(x)\\ 
  &\simle \int_{W}\int_{W} \biggl(\frac{|u(y)-u(w)|}
   {d_{W}(y,w)^{\theta_0}}\biggr)^p \,d\nu_W(y)\,d\nu_W(w)
       \simeq   \Vert u\Vert_{B_{p,p}^{\theta}(W)}^p.
\end{align*}
Thus we have a bounded embedding $f_\#:B^{\theta}_{p,p}(W)\to B^{\tau}_{p,p}(Z)$, with 
control over the Besov  seminorm.  Control of the $L^p$ norm is given by Lemma~\ref{L1-isom}.
\end{proof}

\begin{remark} \label{rmk-qs-metric-change}
Assume that the metric space $Z$ is equipped with another 
``snowflaked'' metric 
$d_Z'$ satisfying
\begin{equation} \label{eq-snowflaking}
   d'_Z(z,x)=d_{Z}(z,x)^{\s}
\end{equation}
for some $\s>0$ and all $z,x\in Z$.
Then the ball $B(z,r)$ with respect to $d_Z$ equals the ball
$B'(z,r')$ with respect to $d'_Z$, where 
$r' = r^{\s}$, and $\nu_Z(B(\z,r)) =\nu_Z(B'(z,r'))$.
This means that the Ahlfors regularity exponent 
changes accordingly, i.e.\ $Q_Z'=Q_Z/\s$. 
Also, it is easily verified that $B^\tau_{p,p}(Z)$ with respect to $d_Z$
equals $B^{\tau'}_{p,p}(Z)$ with respect to $d_Z'$, where $\tau'=\tau/\s$.

A similar observation holds for $W$ with $d_W'(y,w)=d_W(y,w)^\ka$
and $\theta'=\theta/\ka$.
We also see that an $(\al_1,\al_2)$-biH\"older map $f:(Z,d_Z)\to(W,d_W)$
is $(\al_1',\al_2')$-biH\"older when regarded as a map from $(Z,d_Z')$
to $(W,d_W')$, with $\al_j'=\al_j\ka/\s$, $j=1,2$.

Note that the conditions \eqref{eq-QZ} and~\eqref{eq-tau}
remain invariant under such changes, and thus  Lemma~\ref{L1-isom} and 
Proposition~\ref{prop-Besov-to-Besov-with-Q} are invariant under 
``snowflaking''  of the metrics.

For $0 < \sigma \le  1$ the snowflaking in \eqref{eq-snowflaking}
always gives a new metric, but this is not true in general for $\sigma >1$.
However, for our primary examples, boundaries of trees, 
snowflaking always produces
a new metric also for $\sigma >1$.  Changing the weight exponents $\eps_{X}$ 
and $\eps_{Y}$ which determine the  metrics on $\partial X$ and $\partial Y$
to $\eps_{X}'$ and $\eps_{Y}'$ respectively,
gives new ``snowflaked'' visual metrics $d_X'\simeq d_X^\s$ and 
$d_Y'\simeq d_Y^\ka$ with
$\s=\eps_X'/\eps_X$ and $\ka=\eps_Y'/\eps_Y$.
Thus the identity is a quasisymmetric self-mapping of $\bdry X$ and
of $\partial Y$ and it follows 
that if $f:\partial X\to\partial Y$ is a quasisymmetric map for some 
$\eps_X,\eps_Y>0$  then it is quasisymmetric with respect to \emph{all} 
metrics on $\bdry X$ and $\partial Y$ given by~\eqref{bdy-metric}. 
\end{remark}

The conditions $Q_X\ge \al_1 Q_Y$ and 
$\tau\le \al_2\theta + p^{-1}(\al_2 Q_Y-Q_X)$ place restrictions 
on the type of biH\"older maps for which
Proposition~\ref{prop-Besov-to-Besov-with-Q} holds and
on the Hausdorff dimensions of $\bdry X$ and $\bdry Y$. 
The result applies to general metric spaces
but does not take the geometry of the spaces into consideration.
Our primary examples $\bdry X$ and $\bdry Y$ are Cantor type sets with more
structure than general metric spaces, and Besov spaces on such sets
can be regarded as traces of Newtonian spaces.
Thus the conclusions of Proposition~\ref{prop-Besov-to-Besov-with-Q} can be
improved in this case by exploiting embeddings between Newtonian spaces on trees.  
To do so, we need to extend quasisymmetric mappings between boundaries of trees to 
nice mappings between the trees. This will be the focus of the next section. The 
corresponding result for Besov spaces appears in Theorem~\ref{thm-Besov-inv} below.

\section{Extending quasisymmetries from the boundary to the 
tree, and embeddings of Besov spaces}
\label{sect-ext-qs}

\emph{In this section, we assume that both
$X$ and $Y$ are rooted trees such that each vertex has at least two children. 
In particular, $\bdy X$ and $\bdy Y$ are uniformly prefect,
 but no regularity is assumed\/
\textup{(}except for Theorem~\ref{thm-Besov-inv}\textup{)}}

\medskip

Proposition~\ref{prop-Besov-to-Besov-with-Q} 
and the comments after it show that quasisymmetric mappings between 
boundaries of two trees preserve Besov spaces. 
At the same time, we saw in Theorem~\ref{thm-trace-sharp}
that Besov functions on the boundary of a regular tree are traces of 
(and extend to) Newtonian functions on the tree.
It is therefore natural to ask whether every quasisymmetric map can be extended
to the tree as some mapping  preserving Newtonian spaces in a reasonable way.
The aim of this section is to study this question. The natural property of such extended 
functions is  the rough quasiisometry, see Definition~\ref{rough-qiso}.

We shall show that every quasisymmetry 
between the boundaries of two trees extends to a rough quasiisometry between the trees.
At the end of this section, in Theorem~\ref{thm-Besov-inv}, we will show 
that if the trees involved are regular, then the quasisymmetry  
between their boundaries induces a canonical bounded linear map 
between certain Besov spaces, under less restrictive assumptions
than those in Proposition~\ref{prop-Besov-to-Besov-with-Q}.

Let $X$ and $Y$ be two trees such that each vertex has at least two children.
Let $d_{X}$ and $d_{Y}$ be the metrics on $X$ and $Y$ given 
as in~\eqref{unif-metric-trees} by the exponential weights with
$\eps_X$ and $\eps_Y$, respectively.
Consider the Cantor type boundaries $\partial X$ and $\partial Y$ with respect
to these metrics and extend the metrics to the boundaries as in~\eqref{bdy-metric}. 
Also let $|\cdot-\cdot|$ denote the metric on the vertices of 
a tree 
given by the number of edges in the geodesic from one vertex to another. 
Recall that $|x|=|x-0_X|$ and $|y|=|y-0_Y|$ for $x\in X$ and $y\in Y$.

Assume that $f: \partial X \to \partial Y$ is an $\eta$-quasisymmetry,  
as in~\eqref{def-quasisym}.
Since each vertex in $X$ and $Y$ has more than one child, the boundaries 
$\bdry X$ and $\bdry Y$ are uniformly perfect and $\eta$ 
can be chosen to be of the form
\begin{equation}\label{qseta-old}
   \eta (t) = \begin{cases}
     A t^\alpha, & \text{if } t \leq 1,\\
     A t ^{1/\alpha}, & \text{if } t \geq 1,
\end{cases}
\end{equation}
for some $\alpha \le 1$ and $A>0$, by Theorem~11.3 in Heinonen~\cite{Hei}. 

We now use $f$ to construct a map $F: X \to Y$ as follows: For $x \in X$ we let
$F(x)$  be the largest vertex in $Y$ (with respect to $\le$) with the property that
\[
    f(\xi) > F(x) \mbox{ for all } \xi > x.
 \]
 In other words,  we consider all the descendants $\xi\in\partial X$ of $x$, and choose
 as $F(x)$ the largest common ancestor of all $f(\xi)$ for these $\xi$.
Thus $F: X \to Y$ is defined.  We shall use~\eqref{def-quasisym} to discover properties 
of this map,  namely for arbitrary $x_1, x_2 \in X$ we want bounds for
\( 
    |F(x_1)-F(x_2) |
\)
in terms of $|x_1-x_2|$.  

\begin{deff}\label{rough-qiso}
A (not necessarily continuous)
mapping $F:X\to Y$ is an $(L,\Lambda)$-\emph{rough quasiisometry} 
if whenever $x,y\in X$ we have
\[ 
  \frac{1}{L}|x-y|-\Lambda\le |F(x)-F(y)|\le L |x-y|+\Lambda,
\] 
and  for each $y\in Y$ there is a point $x\in X$ such that $|F(x)-y|\le L+\La$.
\end{deff}

Such maps can be regarded as mappings between the vertices
of the corresponding trees, or as mappings between the edge-connected trees  
(by mapping the edge between the two
vertices to the geodesic connecting the images of the two vertices). This dichotomy of rough 
quasiisometries will be
exploited in this and subsequent sections of this paper. Much of the current literature on rough 
quasiisometries call such maps quasiisometries.
However, we will follow the terminology of
Bonk--Heinonen--Koskela~\cite{BHK}  to avoid confusion with biLipschitz maps.

\begin{thm}\label{thm-qs2roughqiso}
Let $X$ and $Y$ be two rooted trees, such that each vertex has at least two 
children. Let $\eps_X$ and $\eps_Y$ give the weighted metrics on $X$ and $Y$ 
respectively as  in~\eqref{bdy-metric}. Suppose that $f:\partial X\to\partial Y$ is an 
$\eta$-quasisymmetric map, where
\begin{equation}   \label{eq-al1-al2-eta}
   \eta (t) = \begin{cases}
A t^{\alpha_1}, & \text{if } t \leq 1,\\
A t ^{\alpha_2}, & \text{if } t \geq 1,
\end{cases}
  \quad \text{and} \quad \alp_1,\alp_2,A>0.
\end{equation}
Then there is an\/
$(L,\Lambda)$-rough quasiisometry $F:X\to Y$ which extends continuously 
along geodesics in $X$ to $f$, and such that for all $x_1,x_2\in X$,
\begin{equation}  \label{eq-rough-qiso-L1-L2}
L_1|x_1-x_2| -\La \le |F(x_1)-F(x_2)| \le L_2|x_1-x_2| +\La,
\end{equation}
where 
\[
L_1=\frac{\al_1 \eps_X}{\eps_Y}, \quad L_2=\frac{\al_2\eps_X}{\eps_Y},
\quad L=\max\biggl\{\frac{1}{L_1},L_2\biggr\}
\quad \text{and} \quad  \Lambda=\frac{2\log A}{\eps_Y}.
\]
\end{thm}

The formulas for $L_1$ and $L_2$ (and thus also for $L$) are sharp
as we will see later in Remark~\ref{rmk-sharpness-L-alp}.

\begin{proof}
Let $x_1,x_2\in X$. We consider two cases.

\medskip
\emph{Case}~1. \emph{$x_1$ and $x_2$ are comparable with respect to the partial
order\/ $\le$.}
In this case we can without loss of generality assume that $x_2 > x_1$,
i.e.\ 
\[
|x_1-x_2|=|x_2|-|x_1|.
\]  
It is then clear by the definition of $F$ that $F(x_1 ) \leq F(x_2)$,
i.e.\ 
\[
|F(x_1)-F(x_2)|=|F(x_2)|-|F(x_1)|.
\]
Using~\eqref{bdy-metric} and the fact that $F(x_1)$ has at least two children, 
we can find $\xi^*,\xi_1\in\bdry X$ so that $\xi_1> x_1$, $\xi^*>x_2$ and
\[
d_Y(f(\xi^*),f(\xi_1)) = \frac{2}{\eps_Y}e^{- \eps_Y|F(x_1)|}.
\]
Since $x_2$ has at least two children, we can also find $\bdry X\ni \xi_2>x_2$
such that $d_X(\xi^*,\xi_2)=2e^{- \eps_X|x_2|}/\eps_X$.
Then 
\[
\frac{d_X(\xi^*,\xi_1)}{d_X(\xi^*,\xi_2)}
  \le \frac{2e^{- \eps_X|x_1|}
    /\eps_X}{2e^{-\eps_X|x_2|}/ \eps_X} 
= e^{ \eps_X|x_1-x_2|}  > 1
\]
and
\[ 
\frac{d_Y(f(\xi^*),f(\xi_1))}{d_Y(f(\xi^*),f(\xi_2))} 
\ge \frac{2e^{-\eps_Y|F(x_1)|}/
  \eps_Y}{2e^{-\eps_Y|F(x_2)|}/\eps_Y}
= e^{ \eps_Y|F(x_1)-F(x_2)|}.
\]
The $\eta$-quasisymmetry condition 
now yields
\[
e^{ \eps_Y|F(x_1)-F(x_2)|} 
\le \eta \biggl( \frac{d_X(\xi^*,\xi_1)}{d_X(\xi^*,\xi_2)} \biggr)
\le \eta(e^{ \eps_X|x_1-x_2|})
\le A e^{\al_2 \eps_X|x_1-x_2|}.
\]
From this we conclude that
\begin{equation}   \label{eq-case1Lip-1}
|F(x_1)-F(x_2)| \le \frac{\al_2\eps_X}{\eps_Y} |x_1-x_2| + \frac{\log A}{\eps_Y}.
\end{equation}
For the converse inequality, first find
$\xi,\xi_2\in\bdry X$ so that $\xi,\xi_2> x_2$ and
\[
d_Y(f(\xi),f(\xi_2)) = \frac{2}{\eps_Y}e^{-\eps_Y|F(x_2)|},
\]
and then $\bdry X\ni\xi_1>x_1$ so that 
$d_X(\xi,\xi_1)=2e^{-\eps_X|x_1|}/\eps_X$.
This is possible since both $F(x_2)$ and $x_1$ have at least two children.
Then 
\[
\frac{d_X(\xi,\xi_2)}{d_X(\xi,\xi_1)}
  \le  \frac{2e^{-\eps_X|x_2|}/\eps_X}{2e^{-\eps_X|x_1|}/\eps_X}
= e^{-\eps_X|x_1-x_2|} < 1
\]
and
\[ 
\frac{d_Y(f(\xi),f(\xi_2))}{d_Y(f(\xi),f(\xi_1))}
\ge \frac{2e^{-\eps_Y|F(x_2)|}/\eps_Y}{2e^{-\eps_Y|F(x_1)|}/\eps_Y}
= e^{- \eps_Y|F(x_1)-F(x_2)|}.
\]
The $\eta$-quasisymmetry condition now yields
\[
e^{- \eps_Y|F(x_1)-F(x_2)|} 
\le A e^{- {\al_1} \eps_X|x_1-x_2|},
\]
and we conclude that
\begin{equation}   \label{eq-case1Lip-2}
|F(x_1)-F(x_2)| \ge \frac{\al_1 \eps_X}{\eps_Y} |x_1-x_2| 
   - \frac{\log A}{\eps_Y}.
\end{equation}

\medskip
\emph{Case}~2. \emph{$x_1$ and $x_2$ are not comparable with respect to the partial
order\/ $\le$.}
Let $\xh$ be the largest common ancestor of $x_1$ 
and $x_2$ in $X$, and $\yh$ be the largest common ancestor of $F(x_1)$
and $F(x_2)$ in $Y$. (Note that it is possible to have $\hat{y}$ be one of 
$F(x_2)$, $F(x_1)$. Indeed, if $F(x_1)$ lies in the geodesic connecting $F(\hat{x})$
to $F(x_2)$, then necessarily $\hat{y}=F(x_1)$.)

We obtain from~\eqref{eq-case1Lip-1} that
\begin{align*}
   |F(x_1)-F(x_2)| &\le |F(x_1)-F(\hat{x})|+|F(\hat{x})-F(x_2)|\\
&\le \frac{\al_2\eps_X}{\eps_Y}
    (|x_1-\hat{x}|+|\hat{x}-x_2|)
     +\frac{2\log A}{\eps_Y}.
\end{align*}
Since $x_1$ and $x_2$ are not comparable, we have
\begin{equation} \label{eq-x1-x2}
|x_1-x_2|=|x_1-\xhat|+|\xhat-x_2|
\end{equation}
 and hence
\begin{equation*}\label{eq-case2Lip-1}
 |F(x_1)-F(x_2)|
\le\frac{\al_2\eps_X}{\eps_Y} |x_1-x_2|
   +\frac{2\log A}{ \eps_Y}.
\end{equation*}

For the converse inequality, 
we let $\xi_1,\xi_2\in\bdry X$ be such that $\xi_1> x_1$ and $\xi_2> x_2$.
Since $\yh$ is the largest common ancestor of $F(x_1)$ and $F(x_2)$,
which in turn are ancestors of $f(\xi_1)$ and $f(\xi_2)$,
we see that 
\[
\dbdY(f(\xi_1),f(\xi_2)) \le 2 e^{- \eps_Y|\yh|}/\eps_Y.
\]
By the definition of $F(x_1)$, we can choose a descendant
$\xi'_1\in\bdry X$ of $ x_1$ so that
\(
\dbdY(f(\xi_1),f(\xi'_1)) = 2e^{-\eps_Y|F(x_1)|}/\eps_Y
\)
and hence
\begin{equation} \label{eq-dbdY-eta-2}
\frac{\dbdY(f(\xi_1),f(\xi'_1))}{\dbdY(f(\xi_1),f(\xi_2))}
\ge \frac{2e^{-\eps_Y|F(x_1)|}/\eps_Y}
   {2e^{-\eps_Y|\yh|}/\eps_Y}
=e^{\eps_Y(|\yh|-|F(x_1)|)}.
\end{equation}

As  $x_1<\xi_1\in\bdry X$ and $x_2<\xi_2\in\bdry X$, we see that
the geodesic from $\xi_1$ to $\xi_2$ must pass through $\xh$.
Hence by~\eqref{bdy-metric},
\[
    \dbdX(\xi_1,\xi_2) = 2e^{-\eps_Y|\xh|}/\eps_X.
\] 
Since moreover 
\(
    \dbdX(\xi_1,\xi'_1) \le 2 e^{-\eps_X|x_1|}/ \eps_X,
\)
and $|\hat{x}|<|x_1|$
we obtain that
\[
\frac{\dbdX(\xi_1,\xi'_1)}{\dbdX(\xi_1,\xi_2)}
   \le \frac{2 e^{- \eps_X|x_1|}/\eps_X}
       {2e^{-\eps_X|\xh|}/\eps_X}
   =e^{\eps_X(|\xh|-|x_1|)} < 1.
\]
The quasisymmetry condition~\eqref{def-quasisym} then implies that
\begin{align*}  
\frac{\dbdY(f(\xi_1),f(\xi'_1))}{\dbdY(f(\xi_1),f(\xi_2))}
&\le \eta\biggl( \frac{\dbdX(\xi_1,\xi'_1)}{\dbdX(\xi_1,\xi_2)} \biggr)
\le\eta ( e^{\eps_X(|\xh|-|x_1|)} ) \le A e^{\al_1 \eps_X(|\xh|-|x_1|)}.
\end{align*}
Inserting this into~\eqref{eq-dbdY-eta-2} yields
\[
  \eps_Y(|\yh|-|F(x_1)|) \le \al_1 \eps_X(|\xh|-|x_1|) + \log A.
\]
Since $F(x_1)\ge\yh$ and $x_1\ge\xh$, this gives
\[
|F(x_1)-\yh| \ge \frac{\al_1 \eps_X}{\eps_Y} |x_1-\xh|
                    - \frac{\log A}{\eps_Y}.
\]
Similarly,
\[
|F(x_2)-\yh| \ge \frac{\al_1 \eps_X}{ \eps_Y} |x_2-\xh|
                    - \frac{\log A}{\eps_Y}.
\]
Summing up the last two estimates and using \eqref{eq-x1-x2} gives
\[ 
|F(x_1)-F(x_2)| = |F(x_1)-\yh| + |\yh-F(x_2)| 
    \ge \frac{\al_1 \eps_X}{\eps_Y} |x_1-x_2|
     - \frac{2\log A}{\eps_Y}
\] 
and proves the rough quasiisometry condition~\eqref{eq-rough-qiso-L1-L2}.

\medskip

To verify the ``density property'', let $y\in Y$, and $\chi\in\partial Y$ 
be a descendant of $y$, that is, $\chi>y$. Since $f$ is a quasisymmetry
between $\partial X$ and $\partial Y$, it is surjective. It follows that 
there exists $\xi\in\partial X$ such that $f(\xi)=\chi$. Because $f$ is 
continuous, it also follows that for each $r>0$ there is a 
positive real number $\delta$ such that 
$f(B(\xi,\delta)\cap\partial X)\subset B(\chi,r)\cap\partial Y.$  
We therefore conclude that there is a vertex $y_0\in Y$ 
such that $\chi>y_0\ge y$ and a vertex $x\in X$ with $F(x)=y_0$. Let $Y_0$ 
be the collection of all such $y_0$.

Suppose that $y$ is \emph{not} in the image of $X$ under $F$.
Then $y\notin Y_0$. Let $z\in Y_0$ be minimal (with respect to $\le$), 
i.e.\ such that whenever $y\le w<z$ then $w\notin Y_0$.  
Let $X_0$ be the collection of all $x\in X$ for which $F(x)=z$. 
It is nonempty, since $z\in Y_0$. 
Let $v\in X_0$ be minimal, i.e.\ such that whenever $x<v$ then $x\notin X_0$. 

Let $u\in X$ be the parent of $v$. Then $F(u)<F(v)=z$, by the
monotonicity of $F$, and $|u-v|=1$. 
Since $z$ is minimal in $Y_0$, it follows that $F(u)<y<z=F(v)$ and hence
\[
   |F(v)-y|<|F(v)-F(u)|\le L|v-u|+\Lambda=L+\Lambda.
\]
Thus every $y\in Y$ is within a distance $L+\Lambda$ of the image of $X$ under $F$. 
This completes the proof of the theorem.
\end{proof}

As promised at the end of Section~\ref{QSvsBesov}, we now give a 
Besov space invariance result.

\begin{thm}   \label{thm-Besov-inv}
Let $X$ and $Y$ be regular $K_X$- and $K_Y$-ary trees equipped with metrics
$d_X$ and $d_Y$ given by exponential weights
with exponents $\eps_X$ and $\eps_Y$, as in~\eqref{unif-metric-trees}.

Let $f:\bdry X\to\bdry Y$ be an $\eta$-quasisymmetric map 
with $\eta$ as in~\eqref{eq-al1-al2-eta}, and let  
$Q_X=(\log K_X)/\eps_X$ and $Q_Y=(\log K_Y)/\eps_Y$ be the Hausdorff
dimensions of $\bdry X$ and $\bdry Y$ given by Lemma~\ref{lem-dimension}.
Suppose that $ p \ge 1$ and $\theta_X,\theta_Y\in(0,1)$ satisfy
\begin{equation}    \label{eq-cond-theta}
 \theta_X \le \begin{cases} \displaystyle
     \frac{Q_X}{p} + \al_1\biggl( \theta_Y - \frac{Q_Y}{p} \biggr), 
         &\text{if } \displaystyle \theta_Y \ge \frac{Q_Y}{p}, \\[4mm]
 \displaystyle \frac{Q_X}{p} + \al_2\biggl( \theta_Y - \frac{Q_Y}{p} \biggr), 
         &\text{if } \displaystyle \theta_Y \le \frac{Q_Y}{p}.   \end{cases}
\end{equation}
Then $f$ induces a canonical embedding
$f_{\#}:B_{p,p}^{\theta_Y}(\bdry Y) \to B_{p,p}^{\theta_X}(\bdry X)$  
such that when $u\in B^{\theta_Y}_{p,p}(\bdy Y)$
is continuous, we have $f_{\#}(u)=u\circ f$.
\end{thm}

This improves upon Proposition~\ref{prop-Besov-to-Besov-with-Q}
for regular trees.

\begin{remark}  \label{rem-KYZ}
The inverse of an $\eta$-quasisymmetric mapping $f$,
with $\eta$ as in~\eqref{eq-al1-al2-eta}, is $\tilde{\eta}$-quasisymmetric
with 
\[
   \tilde{\eta}(t)\simeq \begin{cases}
      t^{1/\al_2} & \text{for } t\le1,  \\
      t^{1/\al_1}   &  \text{for } t\ge1.
     \end{cases}
\]
Thus $f^{-1}:\bdy Y\to\bdy X$ induces an embedding
$f^{-1}_{\#}:B_{p,p}^{\theta_X}(\bdry X) \to B_{p,p}^{\theta'_Y}(\bdry Y)$ whenever
\[
\theta'_Y \le \begin{cases} \displaystyle
    \frac{Q_Y}{p} + \frac{1}{\al_2}\biggl( \theta_X - \frac{Q_X}{p} \biggr), 
        &\text{if } \displaystyle \theta_X \ge \frac{Q_X}{p}, \\[4mm]
\displaystyle 
      \frac{Q_Y}{p} + \frac{1}{\al_1}\biggl( \theta_X - \frac{Q_X}{p} \biggr), 
        &\text{if } \displaystyle \theta_X \le \frac{Q_X}{p}.   \end{cases}
\]
When $\theta_X=Q_X/p$ and $\theta_Y=Q_Y/p$, this
shows that (for all $\al_1$ and $\alp_2$)  $f$ induces an equivalence of the
Besov spaces $B_{p,p}^{\theta_X}(\bdry X)$ and $B_{p,p}^{\theta_Y}(\bdry Y)$.
Thus, we recover  Koskela--Yang--Zhou~\cite[Theorem~5.1]{KYZ}
in our setting of boundaries of regular trees.

We allow for other exponents as well, which gives more general results
in our special setting.
For example, we obtain the following embeddings with $\tau\ge0$,
\begin{align*}
B_{p,p}^{Q_Y/p+\tau/\al_1}(\bdry Y) &\longhookrightarrow B_{p,p}^{Q_X/p+\tau}(\bdry X) 
\longhookrightarrow B_{p,p}^{Q_Y/p+\tau/\al_2}(\bdry Y), \\
B_{p,p}^{Q_Y/p-\tau/\al_2}(\bdry Y) &\longhookrightarrow B_{p,p}^{Q_X/p-\tau}(\bdry X) 
\longhookrightarrow B_{p,p}^{Q_Y/p-\tau/\al_1}(\bdry Y),
\end{align*}
which do not follow from Theorem~5.1 in~\cite{KYZ} when $\tau>0$.
In particular, if $\al_1=\al_2=\al$ (i.e.\ $f$ is a ``snowflaking''
mapping), then $f$ induces an equivalence of the Besov spaces
$B_{p,p}^{Q_X/p+\tau}(\bdry X)$ and $B_{p,p}^{Q_Y/p+\tau/\al}(\bdry Y)$ for all 
$\tau>-Q_X/p$.
\end{remark}

To prove Theorem~\ref{thm-Besov-inv} we shall use the following lemma
which, roughly speaking, gives a sufficient condition for an embedding between 
Newtonian spaces on trees.

\begin{lem}  \label{lem-ug-on-X-Lp} Let $X$ and $Y$ be rooted
trees equipped with metrics
$d_X$, $d_Y$ and measures $\mu_X$, $\mu_Y$ given by exponential weights
with exponents $\eps_X$, $\eps_Y$, $\be_X$ and $\be_Y$, respectively, as 
in~\eqref{unif-metric-trees} and~\eqref{tree-measure}.
Assume that each vertex in $X$ has at most $K_X$ number of children. 
Let $p \ge 1$ and assume that $F:X\to Y$ is an\/ $(L,\La)$-rough
quasiisometry. 
We extend $F$ to the edges of $X$ so that the edge\/ $[x,y]\subset X$ is mapped 
linearly\/ 
\textup{(}with respect to $d_X$ and $d_Y$\textup{)}
to the geodesic connecting $F(x)$ and $F(y)$ in $Y$. 
Also assume that there are a positive integer $n_0$ 
and a constant $C_0$ such that for all $x\in X$ with\/ $|x|\ge n_0$,
\begin{equation}  \label{eq-cond-F-for-Lp}
   (p \eps_X  -\be_X) |x| + (\be_Y- p\eps_Y )|F(x)| \le C_0.
\end{equation}

Let $u:Y\to\R$ be linear\/ {\rm(}with respect to $d_Y$\/{\rm)} on each edge,
and let $g_Y$ be the minimal upper
gradient of $u$ in\/ $(Y,d_Y)$ given on each edge\/ $[z,w]\subset Y$ by 
\[
g_Y(\tau) = \frac{|u(z)-u(w)|}{d_Y(z,w)} \simeq e^{\eps_Y|\tau|}|u(z)-u(w)|
\quad \text{for } \tau \in [z,w].
\]
Let $A=(L+\La)e^{\eps_Y(L+\La)+\eps_X}$. Then the function $g_X$ given 
on each edge\/ $[x,y] \subset X$
by
\[
g_X(t) =
  \begin{cases}
A e^{\eps_X|t|-\eps_Y|F(t)|} g_Y(F(t)), & \text{if } F(x) \ne F(y), \\
    0,  & \text{if } F(x) = F(y),
  \end{cases}
  \quad \text{for } t \in [x,y],
\]
is an upper gradient of $v=u\circ F$ in\/ $(X,d_X)$. Moreover,
\[
\|g_X\|_{L^p(X)} \simle \|g_Y\|_{L^p(Y)}.
\]
\end{lem}

\begin{proof}
That $g_Y$ is a minimal upper gradient of $u$ in $(Y,d_Y)$ is straightforward
from its definition.
Let $\ga$ be a geodesic in $X$ connecting two points (not necessarily vertices)
$a$ and $b$. 
By splitting $\ga$ into parts if necessary, we can assume that $a$ and $b$ 
belong to the same edge. 
Let $\ga'$ be the geodesic in $Y$ connecting $F(a)$ to $F(b)$.
Note that since $a$ and $b$  belong to the same edge, $\ga' = F\circ \ga$. 
Then by the definition of upper gradients,
\begin{equation}   \label{eq-ug-ga}
|v(a)-v(b)| = |u(F(a))-u(F(b))| \le \int_{\ga'} g_Y \,d_Ys,
\end{equation}
where $d_Ys$ denotes the arc length measure
on $\ga'$ with respect to the metric $d_Y$.
The metric on $X$ is with respect to the density $e^{-\eps_X|t|}$
and the metric on $Y$ is with respect to the density $e^{-\eps_Y|t|}$ 
(see~\eqref{unif-metric-trees}). 
By the linearity (with respect to $d_X$ and $d_Y$) of $F$ on the 
edge $[x,y]$ containing $a$ and $b$, we have for all $t\in[a,b]$ that
\[
\frac{d_Ys}{d_Xs} 
     = \frac{d_Y(F(x),F(y))}{d_X(x,y)}
\le A \frac{ e^{-\eps_Y|F(t)|}}{e^{-\eps_X|t|}},
\]
where $A$ is as in the statement of the lemma.
Inserting this into~\eqref{eq-ug-ga} gives
\[
|v(a)-v(b)| 
\le A
\int_{\ga} e^{\eps_X|t|-\eps_Y|F(t)|} g_Y(F(t)) \,d_Xs(t)
 = \int_{\ga} g_X \,d_Xs,
\]
and shows that $g_X$ is an upper gradient of $v$ in $(X,d_X)$.
(Note that if $F$ maps an edge $[x,y]\subset X$ to a single vertex $z\in Y$,
then the above construction gives $g_X\equiv0$ on $[x,y]$.)
To estimate the $L^p$-norm of $g_X$, note first that
\begin{align}   \label{eq-int-gX-to-sum-gY}
\int_X g_X^p\,d\mu_X &\simeq \sum_{\substack{x,y\in X \\ x\sim y}} \int_{[x,y]} 
         \bigl( e^{\eps_X|t|-\eps_Y|F(t)|} g_Y(F(t))\bigr)^p \,d\mu_X \\
&\simeq \sum_{\substack{x,y\in X \\ x\sim y}} 
    e^{(\eps_X|x|-\eps_Y|F(x)|)p - \be_X|x|} 
                \int_{[x,y]} g_Y(F(t))^p \,dt. \nonumber
\end{align}
Since $|F(x)-F(y)|\le L+\La$ and $g_X\equiv0$ on $[x,y]$ if $F(x)=F(y)$, we see that
\[
  \int_{[x,y]} g_Y(F(t))^p \,d|t| \simeq \int_{[F(x),F(y)]} g_Y(\tau)^p \,d|\tau|
     \simeq e^{\be_Y|F(x)|} \int_{[F(x),F(y)]} g_Y^p \,d\mu_Y.
\]
Inserting this into~\eqref{eq-int-gX-to-sum-gY} and using \eqref{eq-cond-F-for-Lp} 
give
\begin{align*}
   \int_X g_X^p\,d\mu_X 
    &\simeq \sum_{\substack{x,y\in X \\ x\sim y}} 
     e^{(p\eps_X -\be_X)|x| + (\be_Y -p \eps_Y )|F(x)|} 
           \int_{[F(x),F(y)]} g_Y^p \,d\mu_Y  \\
    &\simle \sum_{\substack{x,y\in X \\ x\sim y}} 
                \int_{[F(x),F(y)]} g_Y^p \,d\mu_Y.
\end{align*}

Now note that if $F(x)=F(y)$ for some $x,y\in X$, 
then $L^{-1}|x-y|-\Lambda\le 0$, and so $|x-y|\le L\Lambda$,
showing that for every $z\in Y$ there are at most $(K_X+1)^{L\La}$ elements 
in $F^{-1}(\{z\})$.
Furthermore, since for every edge $[x,y]\subset X$ the geodesic connecting 
$F(x)$ to $F(y)$ contains at most $L+\La$ edges, we can conclude that 
each edge in $Y$ belongs to at most $(L+\La)(K_X+1)^{L\La}$ geodesics connecting 
images of neighboring vertices of $X$.
It follows that 
\[
\sum_{\substack{x,y\in X \\ x\sim y}} \int_{[F(x),F(y)]} g_Y^p \,d\mu_Y 
\simeq \int_Y g_Y^p \,d\mu_Y, 
\]
which completes the proof.
\end{proof}

\begin{proof}[Proof of Theorem~\ref{thm-Besov-inv}.]
Note first that by Theorem~\ref{thm-qs2roughqiso}, $f$ extends to a rough
quasiisometry $F:X\to Y$ such that for all $x,y\in X$,
\[
L_1|x-y| -\La \le |F(x)-F(y)| \le L_2|x-y| +\La,
\]
where $L_1=\al_1\eps_X/\eps_Y$ and $L_2=\alp_2\eps_X/\eps_Y$.
Next, equip $X$ and $Y$ with measures $\mu_X$ and $\mu_Y$ 
given by exponential weights with exponents $\be_X$ and $\be_Y$, as 
in~\eqref{tree-measure}. Here 
\begin{align*}
\be_X = \log K_X + p\eps_X (1-\theta_X) & = [Q_X+p(1-\theta_X)]\eps_X,\\
\be_Y = \log K_Y + p\eps_Y (1-\theta_Y) & = [Q_Y+p(1-\theta_Y)]\eps_Y.
\end{align*}
By Proposition~\ref{prop-ext}, each $u\in B^{\theta_Y}_{p,p}(\bdry Y)$ extends to
a function $U\in N^{1,p}(Y)$ so that the minimal upper gradient of $U$ satisfies
\[
\|g_U\|_{L^p(Y)} \simle \|u\|_{B^{\theta_Y}_{p,p}(\bdry Y)}.
\] 
Lemma~\ref{lem-ug-on-X-Lp} then shows that the function $V=U\circ F$ has an upper
gradient $g\in L^p(X)$ with $\|g\|_{L^p(X)} \simle \|g_U\|_{L^p(Y)}$,
provided that~\eqref{eq-cond-F-for-Lp} holds. Moreover, the Poincar\'e 
inequality~\eqref{eq-est-PI-with-e-z} applied with $z=0_X$ shows that $V\in L^1(X)$ and 
\[
|V_X| \le |V(0_X)| + C \|g\|_{L^p(X)}.
\]
By the construction of $F$ we know that $F(0_X)=0_Y$. Therefore we have 
\[
|V(0_X)| = |U(0_Y)| = \biggl| \vint_{\bdry Y} u\,d\nu_Y \biggr| 
\simle \|u\|_{L^p(\bdry Y)}.
\]
Thus the $(p,p)$-Poincar\'e inequality in 
Corollary~\ref{cor-pp-PI}  gives that $V\in L^p(X)$
and $\|V\|_{L^p(X)} \simle \|u\|_{L^p(\bdy Y)}+ \|g\|_{L^p(X)}$.
Hence $\|V\|_{\Np(X)} \simle \|u\|_{\Bt^{\theta_Y}_{p,p}(\bdry Y)}$ and
Proposition~\ref{prop-trace} then implies that $V$ has a trace 
$\Tr V$ on $\bdry X$ such that 
$\|{\Tr V}\|_{\Bt^{\theta_X}_{p,p}(\bdry X)} 
 \simle \|u\|_{\Bt^{\theta_Y}_{p,p}(\bdry Y)}$.

To show that~\eqref{eq-cond-F-for-Lp} holds,  note first that
\[ 
\be_X-p\eps_X     = \eps_X (Q_X - p\theta_X)
    \quad and \quad
\be_Y-p\eps_Y = \eps_Y (Q_Y - p\theta_Y).
\] 
Thus it suffices to show that
\begin{equation}  \label{eq-cond-F-for-Lp-new}
   \eps_X(p\theta_X-Q_X) |x| + \eps_Y(Q_Y-p\theta_Y)|F(x)| \le C_0
\end{equation}
for some constant $C_0$. We need to distinguish two cases.

\medskip
\emph{Case} 1. If $\theta_Y\ge Q_Y/p$, then $Q_Y-p\theta_Y\le0$
and we use the fact that
\[
|F(x)| \ge |F(x)-F(0_X)| -|F(0_X)| \ge L_1|x| -\La,
\]
since $F(0_X)=0_Y$. Hence, as $\theta_X\le Q_X/p+\al_1(\theta_Y- Q_Y/p)$, 
the left-hand side of \eqref{eq-cond-F-for-Lp-new}  is majorized by
\[
\al_1\eps_X(p\theta_Y-Q_Y)|x| + \eps_Y(Q_Y-p\theta_Y)(L_1|x| -\La)
 = -\eps_Y(Q_Y-p\theta_Y)\La =: C_0,
\]
since $L_1=\al_1\eps_X/\eps_Y$.

\emph{Case} 2. If $\theta_Y\le Q_Y/p$, then $Q_Y-p\theta_Y\ge0$
and we use the fact that
\[
|F(x)| \le |F(0_X)| + |F(x)-F(0_X)| \le L_2|x| +\La.
\]
Since $\theta_X\le Q_X/p+\al_2 (\theta_Y - Q_Y/p)$ 
and $L_2=\al_2\eps_X/\eps_Y$, we see as in Case~1 
that~\eqref{eq-cond-F-for-Lp} holds with $C_0=\eps_Y(Q_Y-p\theta_Y)\La$.

\medskip

Finally, for continuous $u \in B^{\theta_Y}_{p,p}(\bdy Y)$ we have, by construction,
\begin{equation*} 
   \Tr V(\zeta)=\lim_{X\ni x\to\zeta}V(x)=\lim_{Y\ni y\to f(\zeta)}U(x)
=u \circ f(\zeta). \qedhere
\end{equation*}
\end{proof}

\section{Quasisymmetric extensions 
of rough quasiisometries between trees}
\label{sect-ext-rough}

\emph{In this section, we assume that both
$X$ and $Y$ are rooted trees such that each vertex has at least one child. 
In particular, each branch is infinite, but no regularity is assumed.}

\medskip

Having shown that quasisymmetries between the boundaries extend to rough 
quasiisometries between the trees, we are next concerned with the issue of whether a rough 
quasiisometry between the trees can be extended to
a quasisymmetry between the boundaries. For certain Gromov hyperbolic spaces 
(those whose boundaries are connected sets) this is known by the deep
work of Gromov~\cite{Gr} and Kapovich~\cite{Kap}, see also the exposition in
Bourdon--Pajot~\cite{BP00}. Given the simple nature of trees,
we are able to study this extendability question more directly here, 
but the fact that the boundaries of  trees are totally disconnected 
means that it is not sufficient to merely check for ``asymptotic'' quasisymmetry.  
Also, quasisymmetric mappings may reverse
the ``order'' of a triple of points; namely, $\xi$ can be closer to $\z$ than $\chi$, while
the image of $\chi$ is closer to the image of $\z$ than the image of $\xi$ is.  
This fact makes our process of checking various cases rather 
troublesome, but is overcome with
the help of Lemma~\ref{lem-order-preserv}, which says in essence that if such order reversal 
takes place, then $\xi$ and $\chi$ are relatively close.   
However, a careful accounting still must be taken.

We continue to use the same notation as in Section~\ref{sect-ext-qs}.
Recall that any rough quasiisometry between trees, because of the attendant 
lack of control at small scales, can without
loss of generality be thought of as a map solely between the vertices of the corresponding 
trees. This is the view we took in Theorem~\ref{thm-qs2roughqiso}
and which we continue with in this section.

Let $F:X\to Y$ be an $(L,\Lambda)$-rough quasiisometry.
We use the following construction to extend $F$ to a mapping
between the boundaries $\bdry X$ and $\bdry Y$: Given $\zeta\in\partial X$ let 
$\{x_i\}_{i=1}^\infty$ be the sequence of vertices in $X$ 
that form the geodesic connecting the root $x_0=0_X$ to $\zeta$
with $x_i$ being the parent of $x_{i+1}$ for $i\in\N$. 
Let $Y_\z=\{F(x_i): i\in\N\}$ be the collection of the images of the vertices in 
this sequence, and $\partial Y_\z:=\overline{Y}_\z\cap\partial Y$,
where the closure is taken with respect to the $d_Y$ metric.
Since $F$ is a rough quasiisometry we see that
\[
     |F(x_i) | \ge |F(x_i)-F(0_X)| - |F(0_X)| \ge \frac{1}{L} |x_i| - \Lambda -|F(0_X)|
     \to \infty,
     \quad \text{as } i \to \infty,
\]
and thus $Y_\z$ is unbounded in the metric $|\cdot-\cdot|$.
As all vertices in the tree $Y$ have finite degree, 
it is not hard to see that $\overline{Y}$
is compact, and thus $\bdy Y_\z$ is nonempty.

Lemma~\ref{lem-towards-chi} below shows that $\bdy Y_\z$ has exactly one point,
and thereby allows us to extend $F$ to the boundary $\bdry X$
by letting  $f(\z)=\chi$, where $\bdy Y_\z=\{\chi\}$.

\begin{lem}  \label{lem-towards-chi}
Let $\z\in\bdry X$, $\chi\in\bdry Y_\z$, and
$y\in Y$ be such that $y<\chi$.
Let\/ $\{x_i\}_{i=1}^\infty$ be the ordered sequence of vertices in $X$ forming the
geodesic connecting\/ $0_X$ to $\z$, whose image forms $Y_\z$.
Then  there is a positive integer $k$ such that whenever $x\in X$ satisfies $x\ge x_k$, 
we have $F(x)\ge y$ and hence $d_Y(F(x),\chi)\le 2e^{-\eps_Y|y|}/\eps_Y$.
In particular, $\bdry Y_z=\{\chi\}$.
\end{lem}

\begin{proof}
The statement is trivial for $y=0_Y$. Assume therefore that $y\ne0_Y$.
Since $\chi$ is a limit point of $Y_\z$ and $y<\chi$, 
there exist positive integers $k$ and $j$ with
$k> j$ such that $F(x_{j}), F(x_k)\ge y$  and
\begin{equation}  \label{eq-choose-x0}
   |x_k-x_j|\ge L(L+2\Lambda+1+|F(x_j)-y|)
\end{equation}
Now let $x\in X$ 
such that $x\ge x_k$.  We need to show that $F(x)\ge y$. We do so by 
contradiction. Suppose $F(x)\not\ge y$. Then because $F(x_k)\ge y$, 
we can trace back from $x$ towards $x_k$ to find
$z\in X$ with $x_k< z\le x$ such that $F(z)\not\ge y$ but the parent 
$\hat{z}$ of $z$ satisfies $F(\hat{z})\ge y$.  
Denoting the parent of $y$ by $\hat{y}$, we now see using~\eqref{eq-choose-x0}
that
\begin{align*}
   |F(z)-\hat{y}|+|\hat{y}-F(x_j)|&=|F(z)-F(x_j)|\ge \frac1L|z-x_j|-\Lambda\\
     &\ge \frac1L|x_k-x_j|-\Lambda \ge L+\Lambda+1+|F(x_j)-y|.
\end{align*}
Given that $F(x_j)\ge y$ and so $|\hat{y}-F(x_j)|=1+|y-F(x_j)|$, we obtain
\begin{equation}\label{eq-imm-above}
   |F(z)-\hat{y}|\ge L+\Lambda.
\end{equation}
On the other hand, since the geodesic from $F(z)$ to $F(\hat{z})$ must 
pass through $\hat{y}$, we have
$|F(z)-\hat{y}|+|\hat{y}-F(\hat{z})|=|F(z)-F(\hat{z})|\le L+\Lambda$,
and given that $F(\hat{z})\ge y$, we know that $|\hat{y}-F(\hat{z})|\ge 1$. Therefore
\[
   |F(z)-\hat{y}|\le L+\Lambda-1,
\]
which contradicts~\eqref{eq-imm-above}. Thus,
for $x\ge x_k$ we have $F(x)\ge y$ and hence
$d_Y(F(x),\chi)\le 2e^{-\eps_Y|y|}/\eps_Y$.

Letting $y\to\chi$ along the geodesic from $0_Y$ to $\chi$ now shows that
$F(x_j)\to\chi$ when $x_j\to\z$ along the geodesic from 
$0_X$ to $\z$.
\end{proof}

The next goal is to show that $f:\bdry X \to \bdry Y$ is a bijection and
a homeomorphism. This will be done through several lemmas, some of which 
(Lemmas~\ref{lem-quasigeod} and~\ref{lem-dist-F(x)-y})  may be of independent interest.

\begin{lem} \label{lem-surjective}
The mapping $f$ is surjective. 
\end{lem}

\begin{proof}
Let $\om\in\bdry Y$ be arbitrary and let $\{y_j\}_{j=0}^\infty$ be a geodesic 
terminating at $\om$. By the density condition, for every $j=0,1,\ldots$\,, there
exists $x_j\in X$ such that $|F(x_j)-y_j|\le L+\La$.
Then $F(x_j)\to\om$ as $j\to\infty$.
By compactness, there is a subsequence (also denoted $\{x_j\}_{j=0}^\infty$)
converging in the metric $d_X$ to some $\z\in \itoverline{X}$. Since
\[
    |x_j|=|x_j-0_X| \ge \frac{|F(x_j)-F(0_X)|-\Lambda}{L} \to \infty
    \quad \text{as } j \to \infty,
\]
we see that $\zeta \in \bdy X$.

Letting $\xhat_j$ be the largest common ancestor of $\{x_i\}_{i=j}^\infty$,
we have that $\xhat_j\to\z$ along a geodesic from $0_X$ to $\z$. Hence 
$F(\xhat_j)\in Y_\z$ and $F(\xhat_j)\to f(\z)$. Lemma~\ref{lem-towards-chi} 
applied to $y=\yhat_k$, where $\yhat_k$
is the largest common ancestor of $F(\xhat_k)$ and $f(\z)$, then shows 
that $F(x_j)\ge\yhat_k$ for all sufficiently large $j$.
Thus $d_Y(F(x_j),f(\z))\le 2e^{\eps_y|\yhat_k|}/\eps_Y \to 0$ as $k\to\infty$, showing that $f(\z)=\om$.
\end{proof}

The proof of the main result of this section needs the following
Morse-type lemma.

\begin{lem}\label{lem-quasigeod}
Let $\gamma$ be a geodesic in $X$ connecting $x_0\in X$ to $\zeta\in\partial X$,
and $\ga'$ be the geodesic in $Y$ connecting $F(x_0)$ to $f(\z)\in\bdry Y$.
Let $  \tau > (L+\Lambda)(2L^2+3\Lambda L +1)$ and $\tau' > L +\Lambda$.
Then 
 \[
    F(\gamma)\subset\bigcup_{y\in\ga'} B_Y(y,\tau)
\quad \text{and} \quad
    \ga'\subset\bigcup_{x\in\ga}B_Y(F(x),\tau'),
 \]
where $B_Y$ denotes open balls in $Y$ with respect to the metric $|\cdot-\cdot|_Y$.
In particular,
 the Hausdorff distance $\dist_H(F(\gamma),\ga') < \tau$.
\end{lem}

The proof of this lemma is a straightforward modification of the one found in
Kapovich~\cite[Lemma~3.43]{Kap} and employs the following simple
projection lemma.

\begin{lem} \label{lem-proj}
Let $F$ and $F'$ be two disjoint nonempty pathconnected closed subsets of a tree 
$Y$, and\/ $\proj_{F}:Y\to F$ be the nearest point projection,  
that is, for each $y\in Y$, 
the point\/ $\proj_{F}(y)\in F$ is such that\/ $\dist(y,F)=|y-\proj_{F}(y)|$.  Then
\begin{enumerate}
\item \label{it-proj-a}
$\proj_{F}$ is well defined\/\textup{;}
\item \label{it-proj-b}
the set\/ $\proj_{F}(F')$ has exactly one element. 
\end{enumerate}
\end{lem}

\begin{proof} 
We start with part~\ref{it-proj-b}. 
Let $x_1,x_2 \in F'$. 
Then for $j=1,2$ there is $z_j \in F$  with $\dist(x_j,F)=|z_j-x_j|$.
The geodesic from $x_j$ to $z_j$ has a last vertex $y_j \in F'$
and we let $\ga_j$ be the geodesic from $y_j$ to $z_j$. It only
hits $F'$ at $y_j$ and $F$ at $z_j$.
Moreover, let $\ga$ be the geodesic in $F'$ from $y_1$ to $y_2$,
and $\phi$ be the geodesic in $F$ from $z_2$ to $z_1$.
Then the concatenation of $\ga$, $\ga_2$ and $\phi$
is a geodesic from $y_1$ to $z_1$, which
must coincide with $\ga_1$ as $Y$ is a tree, and in particular $z_1=z_2$.
This completes the proof of \ref{it-proj-b}.  
Part~\ref{it-proj-a} follows by letting $x=x_1=x_2$.
\end{proof}

\begin{proof}[Proof of Lemma~\ref{lem-quasigeod}]
Let $\{x_i\}_{i=0}^\infty$ be the sequence of vertices in $X$ representing $\ga$
so that $x_i\to\z$ as $i\to\infty$.
For each $i$ let $\ga_i$ be the geodesic in $Y$ connecting $F(x_i)$ to 
$F(x_{i+1})$, and let $\pip=\sum_{i=0}^\infty\ga_i$ be the concatenation of these 
geodesics. Because $F$ is a rough quasiisometry, we know that 
\[
    F(\ga)\subset 
     \phi \subset \bigcup_{i=0}^\infty B_Y(F(x_i), \tau').
\]
Since $\ga' \subset \phi$ this proves the second inclusion.
Thus it suffices to show that $\pip\subset\bigcup_{y\in\ga'}B_Y(y,\tau)$.
Since $F(x_i)\to f(\z)$, 
we can find a subsequence $\{x_{i_k}\}_{k=1}^\infty$ such that  
\[
  d_Y(F(x_{i_{k+1}}),f(\z))<\tfrac12d_Y(F(x_{i_k}),f(\z))
   <\tfrac{1}{2}d_Y(F(x_0),f(\z))
\]
for each positive integer $k$. Let $y_i$ be the largest common 
ancestor of $F(x_i)$ and $f(\z)$, $i=0,1,\ldots$\,. Then
\[
       \tfrac{1}{2} d_Y(F(x_i),f(\z)) < d_Y(y_i,f(\z)) 
       \le d_Y(F(x_i),f(\z)),
       \quad i=0,1,\ldots,
\]
and thus
\[
    d_Y(y_{i_{k+1}},f(\z)) <     d_Y(y_{i_{k}},f(\z)) <     d_Y(y_{0},f(\z)).
\]
It follows that both $y_{i_k}$ and $y_{i_{k+1}}$ belong to $\ga'$, and also that
the geodesic connecting $F(x_{i_k})$ to $F(x_{i_{k+1}})$ 
contains $y_{i_k}$.
This geodesic is contained in the path
$\sum_{i=i_k}^{i_{k+1}-1} \ga_i$ between
$F(x_{i_k})$ and $F(x_{i_{k+1}})$, and thus $y_{i_k} \in \ga_{j_k}$ 
for some $j_k$ with $i_k \le j_k < i_{k+1}$.

Suppose next that a subpath $\pip':=\sum_{k=i}^{j-1}\ga_k$ of $\phi$, 
with $j>i$, does not intersect $\ga'$.
Because each $\ga_{j_k}$, $k=1,2,\ldots$\,, intersect $\ga'$,
we can choose $\pip'$ so that both
$\ga_{i-1}$ and $\ga_j$ intersect $\ga'$.
Since $F$ is a rough quasiisometry, both $\ga_{i-1}$ and $\ga_j$ have 
length at most $L+\La$. We can thus conclude that both 
$F(x_i)$ and $F(x_j)$ lie within the closed $(L+\Lambda)$-neighborhood 
of $\ga'$. By the rough quasiisometry again, we know that
\begin{equation}  \label{eq-est-end-pts}
  |F(x_i)-F(x_j)| \ge \frac{|x_i-x_j|}{L}-\Lambda = \frac{j-i}{L}-\Lambda.
\end{equation}
By Lemma~\ref{lem-proj}, we know that $\proj_{\ga'}(\pip')=\{a\}$ for some $a\in\ga'$, and hence
\begin{align}
  |F(x_i)-a|&=\dist(F(x_i),\ga')\le L+\Lambda, \label{eq-Fxi-a} \\
  |F(x_j)-a|&=\dist(F(x_j),\ga')\le L+\Lambda. \nonumber
\end{align}
It follows that $|F(x_i)-F(x_j)|\le 2(L+\Lambda)$, and inserting this
into~\eqref{eq-est-end-pts} yields 
\(
   j-i\le L(2L+3\Lambda).
\)
Using that $|F(x_k)-F(x_{k+1})|\le L+\Lambda$ for all $k$, we
find that 
\[
     \ell(\pip')\le (j-i)(L+\Lambda) \le L(L+\Lambda)(2L+3\Lambda).
\]
Together with \eqref{eq-Fxi-a} this shows the first inclusion.
\end{proof}

\begin{lem}              \label{lem-injective}
The mapping $f$ is injective.
\end{lem}

\begin{proof}
Suppose that  there are $\z_1,\z_2\in\partial X$ with $\z_1\ne\z_2$ such that
$f(\zeta_1)=f(\zeta_2)=\chi$. 
Let $\hat{x}$ be the largest common ancestor of $\zeta_1$ and $\zeta_2$.
By Lemma~\ref{lem-quasigeod}, the geodesic from $F(\xhat)$ to $\chi$
lies within $\tau$-neighborhoods of both $Y_{\z_1}$ and $Y_{\z_2}$.
It follows that $Y_{\z_1}$ and $Y_{\z_2}$ lie within $2\tau$-neighborhoods of
each other.
Thus, for each $x_k$ in the geodesic from $\xhat$ to $\z_1$, there is
$x'_{k}$ in the geodesic from $\xhat$ to $\z_2$ such that 
$|F(x_k)-F(x'_k)|\le2\tau$.
This gives
\[ 
|x_k-x'_k| \le L(|F(x_k)-F(x'_k)|+\La) \le L(2\tau+\La).
\] 
On the other hand, $|x_k-x'_k| 
\ge |x_k-\xhat| \to \infty$, as $x_k \to \z_1$, giving a contradiction.
\end{proof}

Hence now we know that $f:\partial X\to \partial Y$ is a bijective mapping.   
To show that it is a homeomorphism, one can use Lemma~\ref{lem-towards-chi}
to prove continuity of $f$, and then the compactness of $\bdry X$ and
$\bdry Y$ gives the continuity of $f^{-1}$.
We shall instead in Lemma~\ref{lem-Holder-ct} below
improve upon this and show that $f$ is biH\"older continuous.
The following useful lemma is a rather simple consequence of Lemma~\ref{lem-quasigeod}.

\begin{lem}  \label{lem-dist-F(x)-y}
Let $x\in X$ be the largest common ancestor of $\z,\xi\in\bdry X$, $\z \ne \xi$,
and $y\in Y$ be the largest common ancestor of $f(\z)$ and $f(\xi)$.
Then\/ $|F(x)-y|\le C(L,\La)$.
\end{lem}

\begin{proof}
By Lemma~\ref{lem-quasigeod}, the geodesics from $y$ to $f(\z)$ and to $f(\xi)$
are within $\tau'$-neighborhoods of the images of the geodesics from $x$
to $\z$ and to $\xi$, respectively.
It follows that there are vertices $a,b\in X$ such that $x\le a<\z$, $x\le b<\xi$,
\begin{equation}   \label{eq-Fa-Fb}
|F(a)-y|\le\tau' \quad \text{and} \quad |F(b)-y|\le\tau'.
\end{equation}
Hence, as $F$ is a rough quasiisometry,  we have
\[
|x-a| \le |b-a| \le L(|F(b)-F(a)|+\La) \le L(2\tau'+\La)
\]
and consequently, using~\eqref{eq-Fa-Fb} again, 
\begin{align*}
|F(x)-y| &\le |F(x)-F(a)| + |F(a)-y|  \\
     & \le L|x-a|+\La+\tau'
\le L^2(2\tau'+\La)+\La+\tau'. 
    \qedhere
\end{align*}
\end{proof}

\begin{lem}\label{lem-Holder-ct}
The mapping $f:\partial X\to\partial Y$ is\/
$(\al_2,\al_1)$-biH\"older continuous with 
\begin{equation} \label{eq-Holder-ct}
\al_1=\frac{L_1\eps_Y}{\eps_X}
\quad \text{and} \quad
\al_2=\frac{L_2\eps_Y}{\eps_X}, 
\end{equation}
that is,
\[
        C_1 d_Z(x,y)^{\al_2}\le  d_W(f(x),f(y))
        \le  C_2 d_Z(x,y)^{\alpha_1},
\]
where the constants $C_1$ and $C_2$ 
depend only on $L_1$, $L_2$, $\La$, $|F(0_X)|$, $\eps_X$ and $\eps_Y$.
\end{lem}

\begin{proof} 
Given $\z,\xi\in\bdry X$ with $\z\ne\xi$, let $x$ be the largest 
common ancestor of $\zeta$ and $\xi$.
Similarly, let $y$ denote the largest common ancestor of $f(\zeta)$ 
and $f(\xi)$. We shall estimate 
$d_{Y}(f(\zeta),f(\xi))\simeq e^{-\eps_Y|{y}|}$ in terms of 
$d_{X}(\zeta,\xi)\simeq e^{-\eps_X|{x}|}$.

Lemma~\ref{lem-dist-F(x)-y} implies that $|F(x)-y|\le C$ and hence
$e^{-\eps_Y|{y}|} \simeq e^{-\eps_Y|F(x)|}$.
At the same time, the rough quasiisometry 
property~\eqref{eq-rough-qiso-L1-L2}
gives
\begin{align*}
|F(x)| &\le |F(x)-F(0_X)| +|F(0_X)| \le L_2|x| +\La +|F(0_X)|
 \intertext{and} 
|F(x)| &\ge |F(x)-F(0_X)| -|F(0_X)| \ge  L_1|x| -\La -|F(0_X)|.
\end{align*}
From this we conclude that
\begin{align*}
d_{Y}(f(\zeta),f(\xi)) &\simeq e^{- \eps_Y|F(x)|}
\simge e^{-L_2 \eps_Y|x|} \simeq d_{X}(\zeta,\xi)^{L_2 \eps_Y/\eps_X}
 \intertext{and} d_{Y}(f(\zeta),f(\xi)) &\simeq e^{-\eps_Y|F(x)|}
 \simle e^{-L_1\eps_Y|x|} \simeq d_{X}(\zeta,\xi)^{L_1 \eps_Y/ \eps_X}.
\end{align*}
This gives the biH\"older condition \eqref{def-new-biHolder}
with $\alp_1$ and $\alp_2$ as in \eqref{eq-Holder-ct}.
\end{proof}

\begin{lem}\label{lem-order-preserv}
Let $\z,\xi,\chi\in\partial X$ be such that\/ $0<d_X(\z,\xi)<d_X(\z,\chi)$ and
assume that
\begin{equation}   \label{eq-dX-z-chi-r0}
d_X(\z,\chi) \le \biggl( \frac{C_1}{C_2} \bigl( \tfrac13 \diam\bdry X
              \bigr)^{\al_2}\biggr)^{1/\al_1}  =:r_0,
\end{equation}
where $\al_1$, $\al_2$, $C_1$ and $C_2$ are as in  Lemma~\ref{lem-Holder-ct}. Let $x\in X$ 
be the largest common ancestor of $\z$ and $\xi$, and let $y\in X$ 
be the largest common ancestor of $\z$ and $\chi$.
If\/ $|x-y|\ge L(3C(L,\Lambda)+2\Lambda+L)=:s_0$, where $C(L,\Lambda)$ is 
as in Lemma~\ref{lem-dist-F(x)-y}, then $d_Y(f(\z),f(\xi))\le d_Y(f(\z),f(\chi))$.
\end{lem}

To prove this we will use the following obvious fact:  If $u,v,w\in Y$ satisfy
$|u-v|\ge|w-v|$ and $u\le v$, then $u\le w$.

\begin{proof}
Suppose that $d_Y(f(\z),f(\xi))> d_Y(f(\z),f(\chi))$. 
Let $x_1\in Y$ be the largest common ancestor of
$f(\z)$ and $f(\xi)$, and $y_1$ be the largest common ancestor of 
$f(\z)$ and $f(\chi)$.  Note that $x_1 < y_1$. By
Lemma~\ref{lem-dist-F(x)-y} we know that $|F(x)-x_1|\le C(L,\Lambda)$ and 
$|F(y)-y_1|\le C(L,\Lambda)$. Therefore by the rough quasiisometry of $F$,
\begin{align}\label{eq-order-1}
   |x_1-y_1| &\ge |F(x)-F(y)|-|F(x)-x_1|-|F(y)-y_1| \nonumber\\
&\ge \frac{1}{L} |x-y|-\La -2C(L,\La) \ge C(L,\Lambda)+\Lambda+L.
\end{align}
Since $x_1<y_1$ and $|F(y)-y_1|\le C(L,\Lambda)\le|x_1-y_1|$, 
our remark before the proof shows that $x_1\le F(y)$.

Now let $\om\in\bdry X$ be such that 
$d_X(\z,\om)\ge \max\bigl\{\tfrac13\diam\bdry X,d_X(\z,\chi)\bigr\}$.
Let $z$ be the largest common ancestor of $\z$ and $\omega$, and
$z_1$ be the largest common ancestor of $f(\z)$ and $f(\omega)$. 
Note that $z\le y\le x$.
Then by Lemma~\ref{lem-Holder-ct} and~\eqref{eq-dX-z-chi-r0} we have that
\[
d_Y(f(\z),f(\om)) \ge C_1 d_X(\z,\om)^{\al_2} \ge C_2 d_X(\z,\chi)^{\al_1}
> C_2 d_X(\z,\xi)^{\al_1} \ge d_Y(f(\z),f(\xi))
\]
from which it follows that
$z_1< x_1$. As in~\eqref{eq-order-1} we have
\[
   |x_1-z_1| \ge \frac{1}{L} |x-z|-\La -2C(L,\La) \ge C(L,\Lambda)+\La+L.
\]

Suppose that $F(z)\ge x_1$. Then
by Lemma~\ref{lem-dist-F(x)-y},
\[
  C(L,\La) \ge |F(z)-z_1|  \ge |x_1-z_1| \ge C(L,\Lambda)+\La+L,
\]
which is not possible. Hence
$F(z)\not\ge x_1$. Since $F(y)\ge x_1$, we can find  
$p\in X$ and its parent $\hat{p}$ with $z\le\hat{p}<p\le y\le x$
such that $F(p)\ge x_1$ and $F(\hat{p})\not\ge x_1$.  It follows that 
\[
    |F(p)-x_1|< |F(p)-F(\hat{p})|\le L+\Lambda.
\]
This leads to a contradiction because by Lemma~\ref{lem-dist-F(x)-y} again,
\begin{align*}
   L+\Lambda+C(L,\Lambda)&> |F(p)-x_1| + |F(x)-x_1| \ge |F(p)-F(x)|\\ 
&\ge \frac1L |p-x|-\Lambda \ge \frac1L|y-x|-\Lambda \ge 3C(L,\Lambda)+\Lambda+L,
\end{align*}
which is not possible. 
Thus the assumption that $d_Y(f(\z),f(\xi))> d_Y(f(\z),f(\chi))$ 
is false, and the lemma is proved.
\end{proof}

Finally, we are ready to prove the main result of this section.

\begin{thm}\label{rough-qs}
If the rough quasiisometry $F:X\to Y$ satisfies 
\[
L_1|x_1-x_2| -\La \le |F(x_1)-F(x_2)| \le L_2|x_1-x_2| +\La
\quad \text{for all } x_1,x_2\in X,
\]
then the mapping $f:\partial X\to\partial Y$
is 
an $\eta$-quasisymmetric map,
where
\begin{equation} \label{eq-eta}
   \eta (t) = \begin{cases}
A t^{\alpha_1}, & \text{if } t \leq 1,\\
A t ^{\alpha_2}, & \text{if } t \geq 1,
\end{cases}
 \quad
  \alp_1 = \frac{L_1\eps_Y}{\eps_X},
  \quad
  \alp_2 = \frac{L_2\eps_Y}{\eps_X}
  \quad \text{and} \quad
  A > 0.
\end{equation}
\end{thm}

\begin{remark} \label{rmk-sharpness-L-alp}
Suppose that each vertex 
in $X$ and $Y$ has 
at least two children.
Let $f: \bdy X \to \bdy Y$ be an $\eta$-quasisymmetric map
with $\eta$ as in \eqref{eq-eta}.
If we first extend it to a rough quasiisometry $F: X \to Y$ using
Theorem~\ref{thm-qs2roughqiso}
and then apply Theorem~\ref{rough-qs} we get $f$ back
with the same exponents $\alp_1$ and $\alp_2$.
This shows that the formulas for $\alp_1$ and $\alp_2$ in 
Theorem~\ref{rough-qs} as well as the formulas
for $L_1$ and $L_2$ in Theorem~\ref{thm-qs2roughqiso} 
are all sharp.

Similarly, 
if $F:X\to Y$ is an $(L,\Lambda)$-rough quasiisometry, 
then
Theorem~\ref{rough-qs} gives us a quasisymmetry $f:\bdry X\to\bdry Y$,
which in turn, by Theorem~\ref{thm-qs2roughqiso},
 induces an $(L,\Lambda')$-rough quasiisometry $G:X\to Y$ 
(called $F$ in Section~\ref{sect-ext-qs})
with a better behavior than the 
original map $F$.
By construction, $G(0_X)=0_Y$ and $G$ is order-preserving, 
i.e.\ $G(x)\le G(y)$ whenever $x\le y$.
Moreover, for every $x\in X$ we have $|F(x)-G(x)|\le \tau$
for some $\tau$.
To see this, let $x\in X$ be arbitrary and let $\z,\chi\in\bdry X$ 
be such that $\z,\chi > x$ and
$f(\z)$ and $f(\chi)$ are descendants of two distinct children of 
$G(x)$, which is possible by the construction of $G$.
Let $z \ge x$ be the largest common ancestor of $\z$ and $\chi$.
Since $x$ has at least two children, there exists $\xi\in\bdry X$ such that $x$ 
is the largest common ancestor of $\z$ and $\xi$.
Because $\xi > x$ and $G$ is order-preserving (and induces $f$),
we get
that $f(\xi)>G(x)$. 
Therefore $d_Y(f(\z),f(\chi))$ equals either $d_Y(f(\z),f(\xi))$ or $d_Y(f(\chi),f(\xi))$.
In the first case we have, as $f$ is an $\eta$-quasisymmetry, that
\[
1 = \frac{d_Y(f(\z),f(\chi))}{d_Y(f(\z),f(\xi))} \le \eta\biggl( \frac{d_X(\z,\chi)}{d_X(\z,\xi)} \biggr)
= \eta (e^{\eps_X(|x|-|z|)}),
\]
which yields $|x-z|=|z|-|x|\le \tau_0:=(\log\eta^{-1}(1))/\eps_X$.
In the second case, a similar argument with the roles of $\z$ and $\chi$ interchanged gives 
$|x-z|\le\tau_0$.
Thus, in either case we have $|F(x)-F(z)|\le L\tau_0+\La$. Together with the estimate
$|F(z)-G(x)|\le C(L,\La)$ of Lemma~\ref{lem-dist-F(x)-y}, this gives
$|F(x)-G(x)|\le \tau$.
\end{remark}

Theorem~\ref{rough-qs}, for more general Gromov hyperbolic spaces, seems to have been stated 
in Bourdon--Pajot~\cite{BP00}, where the credit for it 
is given to Gromov~\cite{Gr}. However, we were not able to find this result in~\cite{Gr},
and so we give a self-contained proof here.
Our proof of Theorem~\ref{rough-qs} 
uses tools inspired by the proof of Kapovich~\cite[Theorem~3.47]{Kap}.
Similar ideas can be found in Jeffers~\cite{Jeff} for a result
characterizing isometries of $\mathbb{H}^n$.

\begin{proof}[Proof of Theorem~\ref{rough-qs}]
Let $\zeta,\xi,\chi\in\partial X$ be such that $\chi\ne\zeta\ne \xi$, and let
 \[
   t=\frac{d_X(\zeta,\xi)}{d_X(\zeta,\chi)}.
 \]
Assume first that $d_X(\z,\xi)\le r_0$ and 
$d_X(\z,\chi)\le r_0$, where $r_0$ is as in Lemma~\ref{lem-order-preserv}.
Adopting the notation from  Lemma~\ref{lem-order-preserv} we let
$x,y\in X$ and $x_1,y_1\in Y$ be the largest common ancestors of the pairs
$\z$ and $\xi$, $\z$ and $\chi$,
$f(\z)$ and $f(\xi)$, and $f(\z)$ and $f(\chi)$, respectively.
By Lemma~\ref{lem-dist-F(x)-y} and the rough quasiisometry of $F$,
\begin{align}   \label{eq-comp-x1-y1-ge}
   |x_1-y_1| &\ge |F(x)-F(y)|-|F(x)-x_1|-|F(y)-y_1| \nonumber\\
&\ge {L_1} |x-y|-\La -2C(L,\La),
\end{align}
where $L=\max\{1/L_1,L_2\}$,
and similarly,
\begin{equation}   \label{eq-comp-x1-y1-le}
   |x_1-y_1|
\le |F(x)-F(y)|+|F(x)-x_1|+|F(y)-y_1|
\le {L_2}|x-y|+\Lambda+2C(L,\Lambda).
\end{equation}
We have $t=e^{-\eps_X(|x|-|y|)}$.
With $t_0=e^{-\eps_X s_0}<1$ and $t_1=1/t_0=e^{\eps_X s_0}>1$, where
$s_0=L(3C(L,\La)+2\La+L)$ is as in Lemma~\ref{lem-order-preserv},
we consider three cases.

\medskip

\noindent \emph{Case}~A. \emph{$0<t\le t_0$.}
In this case $x > y$ and
 $|x-y|\ge s_0$.
So by Lemma~\ref{lem-order-preserv}, $x_1 \ge y_1$.
Thus by \eqref{eq-comp-x1-y1-ge},
\begin{align*}
\frac{d_Y(f(\z),f(\xi))}{d_Y(f(\z),f(\chi))}&= e^{-\eps_Y(|x_1|-|y_1|)}
     =  e^{-\eps_Y|x_1-y_1|}
     \simle e^{-\eps_Y {L_1}|x-y|} = t^{ {L_1} \eps_Y/\eps_X}.
\end{align*}

\medskip

\noindent \emph{Case}~B. \emph{$t\ge t_1$.}
In this case $x < y$ and $|x-y|\ge s_0$.
So by Lemma~\ref{lem-order-preserv},
applied with the roles of $\xi$ and $\chi$ swapped,
we see that $x_1 \le y_1$. Thus by \eqref{eq-comp-x1-y1-le},
\begin{align*}
\frac{d_Y(f(\z),f(\xi))}{d_Y(f(\z),f(\chi))}&= e^{-\eps_Y(|x_1|-|y_1|)}
     = e^{\eps_Y|x_1-y_1|}\simle e^{\eps_Y  {L_2}|x-y|} = t^{  {L_2}\eps_Y/\eps_X}.
\end{align*}

\noindent \emph{Case}~C. \emph{$t_0\le t\le t_1$.}
In this case  $|x-y|\le s_0$ and
thus  by~\eqref{eq-comp-x1-y1-le},
$
  |x_1-y_1| \le  L_2 s_0 + \La+2C(L,\La),
$
from which it follows that
\[
\frac{d_Y(f(\z),f(\xi))}{d_Y(f(\z),f(\chi))}=e^{-\eps_Y(|x_1|-|y_1|)}
     \le e^{\eps_Y|x_1-y_1|} \simle 1.
\]
We have thus shown that if  $\z,\xi,\chi\in X$,  $\z\ne\xi\ne\chi$,
$d_X(\z,\xi)\le r_0$ and $d_X(\z,\chi)\le r_0$,
then 
\[ 
\frac{d_Y(f(\z),f(\xi))}{d_Y(f(\z),f(\chi))}
  \le \tilde{\eta} \biggl( \frac{d_X(\zeta,\xi)}{d_X(\zeta,\chi)} \biggr),
\quad \text{where} \quad
\tilde{\eta}(t) \simeq \begin{cases}
 t^{ L_1 \eps_Y/\eps_X}, & \text{if } 0 \le t \le t_0,\\
1, & \text{if } t_0 < t < t_1, \\
 t^{ L_2\eps_Y/\eps_X}, & \text{if } t \ge t_1,
\end{cases}
\] 
and the comparison constant depends only on $L$ and $\La$.

Assume now that $d_X(\z,\xi)\ge r_0$ or $d_X(\z,\chi)\ge r_0$.
We shall again distinguish three cases. Note that $r_0\simeq\diam\bdry X$.
\medskip

\noindent \emph{Case}~1. \emph{$d_X(\z,\xi)\le r_0 \le d_X(\z,\chi) \le\diam\bdry X$.}
Then $t\le1$ and by Lemma~\ref{lem-Holder-ct},
\[ 
\frac{d_Y(f(\z),f(\xi))}{d_Y(f(\z),f(\chi))} 
\le \frac{C_2 d_X(\z,\xi)^{\al_1}}{C_1r_0^{\al_2}}
\simle \biggl( \frac{d_X(\z,\xi)}{d_X(\z,\chi)} \biggr)^{\al_1}
\simle \tilde{\eta} \biggl( \frac{d_X(\z,\xi)}{d_X(\z,\chi)} \biggr).
\] 

\noindent \emph{Case}~2. \emph{$\diam\bdry X \ge d_X(\z,\xi)\ge r_0 \ge d_X(\z,\chi)$.}
Then $t\ge1$ and by Lemma~\ref{lem-Holder-ct},
\[ 
\frac{d_Y(f(\z),f(\xi))}{d_Y(f(\z),f(\chi))} 
\le \frac{\diam\bdry Y}{C_1 d_X(\z,\chi)^{\al_2}}
\simle \biggl( \frac{d_X(\z,\xi)}{d_X(\z,\chi)} \biggr)^{\al_2}
\simle \tilde{\eta} \biggl( \frac{d_X(\z,\xi)}{d_X(\z,\chi)} \biggr).
\] 

\noindent \emph{Case}~3. \emph{$r_0\le d_X(\z,\xi)\le\diam\bdry X$
and $r_0\le d_X(\z,\chi)\le\diam\bdry X$.} Then $t\simeq1$ and
by Lemma~\ref{lem-Holder-ct} again,
\[
\frac{d_Y(f(\z),f(\xi))}{d_Y(f(\z),f(\chi))}
\simeq 1 \simeq \frac{d_X(\z,\xi)}{d_X(\z,\chi)}
\simle \tilde{\eta} \biggl( \frac{d_X(\z,\xi)}{d_X(\z,\chi)} \biggr).
\]
Thus there is $A'$ such that
\[
\frac{d_Y(f(\z),f(\xi))}{d_Y(f(\z),f(\chi))}
  \le A'\tilde{\eta} \biggl( \frac{d_X(\zeta,\xi)}{d_X(\zeta,\chi)} \biggr)
\]
for all $\zeta,\xi,\chi\in\partial X$ with $\chi\ne\zeta\ne \xi$.
A homeomorphism $\eta:[0,\infty)\to[0,\infty)$
of the form \eqref{eq-eta} such that $\eta \ge A'\tilde{\eta}$, shows that $f$ is 
$\eta$-quasisymmetric. 
\end{proof}

We conclude this paper by considering a Mostow-type rigidity 
result. The setting of trees where the edges are hyperbolic regions  
pasted together in a combinatorial way was studied in Bourdon--Pajot~\cite{BP00a}. 
It was shown in~\cite{BP00a} that if there is a rough quasiisometry (called a 
quasiisometry in~\cite{BP00a}) between two such hyperbolic trees (called hyperbolic buildings there), then  
that rough quasiisometry is a bounded distance from an isometry between the trees; in particular, the two 
hyperbolic trees, if rough quasiisometrically equivalent, are necessarily isometric. Their proof needs the boundaries of hyperbolic trees to be connected 
(and in fact to support a Poincar\'e inequality). 
In contrast, in our setting the boundaries 
of the two trees are totally disconnected, and support no Poincar\'e inequality.

There are many other rigidity theorems of various types in geometry.
It is shown in Beardon--Minda~\cite{BM} and Jeffers~\cite{Jeff} that
any bijective self-map of the hyperbolic space or the Euclidean space must be an isometry if it 
preserves complete geodesics. Given that trees are  naturally Gromov hyperbolic, it is natural to 
ask similar questions in our setting. We show below that if an injective and almost surjective
map between trees maps geodesics into geodesics, then it is an isometry.
As Example~\ref{Example} below shows, in general we have no rigidity theorem of the
Bourdon--Pajot type for rough quasiisometries.

\begin{prop}\label{prop-dist-preserve}
Let $X$ and $Y$ be two rooted trees such that each vertex has  at least two children.
Assume that $G:X\to Y$ 
satisfies the following assumptions\/\textup{:}
\begin{enumerate}
\item \label{G-inj}
$G$ is injective\/\textup{;}
\item \label{G-geod}
$G$ maps geodesics into geodesics\/\textup{;}
\item \label{G-density}
for each $y \in Y$ there is $x \in X$ such that $G(x) > y$\textup{;}
\item \label{G-0}
$G(0_X)=0_Y$ or\/ $0_Y$ has at least three children.
\end{enumerate}
Then
$G$ is an isometry\/ {\rm(}with respect to\/ $|\cdot-\cdot|$\/{\rm)}. 
\end{prop}

Note that Condition~\ref{G-density} is satisfied by a rough quasiisometry 
$F:X\to Y$
because of the density property $\dist_H(Y,F(X))<\infty$.

Here we say that a mapping \emph{maps geodesics into geodesics} if  every geodesic line 
(that is, a geodesic connecting 
two boundary points) in $X$ gets mapped into a geodesic line in 
$Y$ (but not necessarily onto).
Note that each geodesic path [x,y] in
a tree lies inside a geodesic 
line, but we do not require its image 
to lie in the geodesic $[G(x),G(y)]$.

As a consequence of the above proposition, if the map $G$ mentioned in Remark~\ref{rmk-sharpness-L-alp}
is injective and maps geodesics into geodesics, then it must be an isometry and the boundary map $f$ must be a snowflake 
map. 

In the proof below we strongly use the 
density assumption \ref{G-density}.
That \ref{G-density} cannot be dropped from the assumptions
of Proposition~\ref{prop-dist-preserve} is obvious,
and it is easy to construct an example showing that \ref{G-0} cannot
be dropped either.
The following example shows that 
neither \ref{G-inj} nor \ref{G-geod} can be dropped.

\begin{example}\label{Example}
Let $X$ be a binary rooted tree and $Y$ be a ternary rooted tree. 
We inductively
map $X$ to $Y$ 
as follows: The root $0_X$ is mapped to $0_Y$, and with $x_1$ and $x_2$ 
being the children of $0_X$, we set
$G(x_1)=y_1$ and $G(x_2)=0_Y$, where $y_1$, $y_2$ and $y_3$ 
are the three children of $0_Y$. With $x_{2,1}$ and $x_{2,2}$ being the 
children of $x_2$, we set $G(x_{2,1})=y_2$ and $G(x_{2,2})=y_3$.  
Repeating this procedure  for the regular subtrees rooted at  
$x_1$,  $x_{2,1}$, and $x_{2,2}$, we extend $G$ to the next generation. 
Iterating this process, we obtain a rough quasiisometry $G:X\to Y$ 
satisfying 
\[
\tfrac12 |x-y| - 2 \le |G(x)-G(y)| \le |x-y| \quad
\text{for } x,y\in X.
\]
(The worst case being when both $x$ and $y$ are mapped to the same vertices
as their parents, and the same holds for every other of their ancestors.)
Observe that $G$ is surjective and maps geodesics into geodesics.
However, it is not injective, and 
there is no isometry between $\partial X$ and $\partial Y$, nor
between $X$ and $Y$. Thus, the conclusion of 
Proposition~\ref{prop-dist-preserve} fails here,
showing that the injectivity assumption cannot
be dropped. 

However, by Theorem~\ref{rough-qs}, $G$ still induces a quasisymmetry
between the boundaries $\bdry X$ and $\bdry Y$ with $\eta$ as in~\eqref{eq-eta},
$\al_1=\eps_Y/2\eps_X$ and $\al_2=\eps_Y/\eps_X$.
Note that if we equip $\bdry X$ and $\bdry Y$ with the visual metrics
given by~\eqref{bdy-metric} with $\eps_X=\log3$ and $\eps_Y=\log2$, then
$\bdry X$ can be identified with the usual ternary Cantor dust, while 
$\bdry Y$ corresponds to a totally disconnected variant of the 
Sierpi\'nski gasket, see Example~\ref{ex-fractal}, and that in this case
$0<\al_1,\al_2<1$.

Furthermore, for $y\in Y$ let $H(y)$ be the smallest (with respect to $\le$)
$x\in X$ such that $G(x)=y$.
Then $H:Y\to X$ is an order-preserving rough quasiisometry with $H(0_Y)=0_X$
which fulfills all the requirements in 
Proposition~\ref{prop-dist-preserve} but for
\ref{G-geod}
since the geodesic
from $y_2$ to $y_3$ is mapped to $0_Y$, $x_{2,1}$ and $x_{2,2}$ which are not
contained in any common geodesic.
As $H$ is not an isometry, this shows that 
\ref{G-geod} 
cannot be dropped from Proposition~\ref{prop-dist-preserve}.
\end{example}

\begin{proof}[Proof of Proposition~\ref{prop-dist-preserve}.]
We use $\xhat$ to denote the parent of a vertex $x$.  
We seek to show that the children of $0_X$ map bijectively 
to the children of $0_Y$.
Assume first that $G(0_X)=0_Y$ and let $x$ be a child of $0_X$.
If $\widehat{G(x)}\ne 0_Y$, then
by the density 
assumption \ref{G-density}
and the fact that each vertex of $Y$ has at least two children,  
there exists $t\in X$ such that $G(t)>\widehat{G(x)}$ 
and $G(t)$ is not in the geodesic ray containing $0_Y$ and $G(x)$.
In other words, $G(x)$, $G(t)$ and $0_Y$ cannot belong to any geodesic in $Y$.
On the other hand, it is always possible to find a geodesic in $X$ containing
$x$, $t$ and $0_X$. 
This violates 
\ref{G-geod},
and hence $G(x)$ must be 
a child of $0_Y$.
Conversely, if $y$ is a child of $0_Y$, then there exists $u\in X$ such that 
$G(u)>y$, by the density assumption \ref{G-density} again.
Let $a$ be the child of $0_X$ which is an ancestor of $u$, 
and $a'$ be another child of $0_X$.
By the above, $G(a)$ and $G(a')$ are children of $0_Y$.
If $G(a)\ne y\ne G(a')$, then $G(u)$, $G(a)$ and $G(a')$ do not belong to any 
geodesic in $Y$, but it is always possible to find a geodesic in $X$ containing 
$u$, $a$ and $a'$, which is a contradiction.
Thus $y=G(a)$ or $y=G(a')$.
Hence the children of $0_X$ map bijectively to the children of $0_Y$.

Next, we proceed by induction. 
Assume that for all $z\in X$ with 
$1 \le |z|\le n$, we have that 
$G(z)$ is a child of $G(\zhat)$,
and let $x\in X$ be arbitrary with $|x|=n+1\ge2$.
We distinguish two cases:

\medskip
\emph{Case} 1. If $G(x)>G(\xhat)$ is not a child of $G(\xhat)$ 
then as above, by the density 
assumption \ref{G-density},
we can find 
$v\in X$ such that $G(x)$ and $G(v)$ are not comparable 
with respect to $\le$ and
their largest common ancestor is the parent of $G(x)$.
Then $G(x)$, $G(\xhat)$ and $G(v)$ do not belong to any
geodesic in $Y$, but $x$, $\xhat$ and $v$ belong to a geodesic in $X$.
This is a contradiction.

\medskip

\emph{Case} 2. If $G(x)\not>G(\xhat)$ then, as the ancestors
of $G(\xhat)$ are exactly the images of the ancestors of $\xhat$
(by the induction hypothesis),
we see that $G(x)$ and $G(\xhat)$ are not comparable 
with respect to the partial ordering.
Let $z$ be their largest common ancestor.
If $z\ne0_Y$ then the image of the geodesic from $0_X$ to $x$ 
must contain $G(x)$, $G(\xhat)$ and $0_Y$, which is impossible.
On the other hand, if $z=0_Y$, then by the above, $G(x)$ is not a child
of $0_Y$, i.e\ $z$ is not the parent of $G(x)$.
Again by the density assumption \ref{G-density}, 
we can find $w\in X$ such that
$G(w)$ and $G(x)$ are not comparable 
with respect to the partial ordering, but their largest common
ancestor is the parent of $G(x)$.
Thus, the vertices $G(x)$, $G(\xhat)$ and $G(w)$ do not belong 
to any geodesic in $Y$ but there exists a geodesic in $X$ containing
$x$, $\xhat$ and $w$.
This final contradiction shows that $G(x)$ must be a child of $G(\xhat)$.
By induction, this holds for all $x\in X$ and hence  $|G(x)-G(y)| = |x-y|$
for all $x,y\in X$.

\medskip

To show that $G$ is surjective, let $y'\in Y$ be arbitrary.
By the density assumption \ref{G-density}, 
there
exists $x'\in X$ such that $G(x')>y'$.
Since $|G(x')|=|x'|$, we can find $t'\le x'$ such that $|y'|=|t'|$.
By the above, we have $G(t')\le G(x')$ and $|G(t')|=|t'|=|y'|$, showing that
$G(t')=y'$.
Thus $G$ is surjective, and hence an isometry.

Finally, if $G(0_X) \ne 0_Y$ then let
$Y'$ be the tree $Y$ rerooted at $G(0_X)$.
Then $Y$ and $Y'$ are isometric with respect to $|\cdot-\cdot|$.
Let $G': X \to Y'$ be the map induced by $G$.
Note that $G'(0_X)=0_{Y'}$ and that
each vertex in $Y'$ has at least two children.
Thus we
can apply the above result to $G'$ to show that $G'$ is an isometry,
which is equivalent to $G$ being an isometry.
\end{proof}


\begin{thebibliography}{99}

\bibitem{BM} \art{Beardon, A. \AND Minda,  D.}
           {Sphere-preserving maps in inversive geometry}
          {Proc. Amer. Math. Soc.}{130}{2002}{987--998}
         
\bibitem{BePe} \art{Bellissard, J. \AND Pearson, J.}
            {Noncommutative Riemannian geometry and diffusion 
          on ultrametric Cantor sets}{J. Noncommut. Geom.}{3}{2009}{447--480}
         
\bibitem{BS1} \book{Bennett, C. \AND Sharpley, R.}
        {Interpolation of Operators}
        {Pure and Applied Mathematics {\bf 129}, Academic Press, Boston, MA, 1988}

\bibitem{BBbook} \book{Bj\"orn, A. \AND Bj\"orn, J.}
        {\it Nonlinear Potential Theory on Metric Spaces}
        {EMS Tracts in Mathematics {\bf 17},
        European Math. Soc., Zurich, 2011}

\bibitem{BBS} \art{\auth{Bj\"orn}{A}, \auth{Bj\"orn}{J}
	\AND \auth{Shanmugalingam}{N}}
        {The Dirichlet problem for \p-harmonic functions on metric spaces}
        {J. Reine Angew. Math.} {556} {2003} {173--203}

\bibitem{BBS4} \art{Bj\"orn, A., Bj\"orn, J. \AND Shanmugalingam, N.}
	{Sobolev extensions of H\"older continuous and characteristic 
	functions on metric spaces}
	{Canadian J. Math.} {59} {2007} {1135--1153}

\bibitem{BS} \art{Bj\"orn, J. \AND Shanmugalingam, N.}
        {Poincar\'e inequalities, uniform domains, and extension properties for 
         Newton--Sobolev spaces 
        in metric spaces} {J. Math. Anal. Appl.} {332} {2007} {190--208} 

\bibitem{BHK} \book{Bonk, M., Heinonen, J. \AND Koskela, P.}
        {Uniformizing Gromov Hyperbolic Spaces} 
          {Ast\'erisque {\bf 270} (2001)}

\bibitem{BP00a} \art{Bourdon, M. \AND Pajot, H.}{Rigidity of quasi-isometries for some 
         hyperbolic buildings}{Comment. Math. Helv.}{\bf 75}{2000}{701--736}
       
\bibitem{BP00} \artin{Bourdon, M. \AND Pajot, H.}
         {Quasi-conformal geometry and hyperbolic geometry}
        {\emph{Rigidity in Dynamics and Geometry} (Cambridge, 2000), 
          pp. 1--17, Springer, Berlin, 2002}
  
\bibitem{BP03} \art{Bourdon, M. \AND Pajot, H.}{Cohomologie $l_p$ et espaces de Besov}
         {J. Reine Angew. Math.}{558}{2003}{85--108}
       
\bibitem{BH} \book{Bridson, M. \AND Haefliger, A.}
         {Metric Spaces of Non-positive Curvature}
         {Grundlehren der Mathematischen Wissenschaften {\bf 319},
          Springer, Berlin, 1999}
          
\bibitem{DSV}\art{Dalrymple, K., Strichartz, R. \AND Vinson, J.}
          {Fractal differential equations on the Sierpi\'nski gasket} 
           {J. Fourier Anal. Appl.}{5}{1999}{203--284} 

\bibitem{DaGaNh01} \artin{\auth{Danielli}{D}, \auth{Garofalo}{N}
        \AND \auth{Nhieu}{D.-M}}
        {Sub-elliptic Besov spaces and the characterization of traces
          on lower dimensional manifolds}
        {\emph{Harmonic Analysis and Boundary Value Problems}
          \textup{(}Fayetteville, AR, 2000\/\textup{)},
        Contemp. Math. {\bf 277}, pp. 19--37, 
        Amer. Math. Soc., Providence, RI, 2001}

\bibitem{DaGaNh07} \book{\auth{Danielli}{D}, \auth{Garofalo}{N}
        \AND \auth{Nhieu}{D.-M}}
        {Non-doubling Ahlfors Measures, Perimeter Measures, and 
        the Characterization of the Trace Spaces of Sobolev Functions 
        in Carnot--Carath\'eodory spaces}
        {{Mem. Amer. Math. Soc.} {{\bf 182}:857} {(2006)}}
          
\bibitem{Fal} \book{Falconer, K.}{Techniques in Fractal Geometry}
          {Wiley, Chichester, 1997}
          
\bibitem{FM} \art{Farb, B. \AND Mosher, L.}
     {A rigidity theorem for the solvable Baumslag--Solitar groups}
            {Invent. Math.}{131}{1998}{419--451}

\bibitem{garn} \art{Garnett, J.}
        {Positive length but zero analytic capacity}
        {Proc. Amer. Math. Soc.}{24}{1970}{696--699; errata, ibid. 
        {\bf 26} (1970), 701} 

\bibitem{GKS} \art{Gogatishvili, A., Koskela, P. \AND Shanmugalingam, N.}
        {Interpolation properties of Besov spaces defined on metric spaces}
        {Math. Nachr.} {283} {2010} {215--231}
        
\bibitem{Gr} \artin{Gromov, M.}{Hyperbolic groups}
        {\emph{Essays in Group Theory},  Math. Sci. Res. Inst. Publ. {\bf 8}, 
           pp. 75--263, Springer, New York, 1987}
        
\bibitem{Haj1} \artin{\auth{Haj\l asz}{P}}
	{Sobolev spaces on metric-measure spaces}
	{\emph{Heat Kernels and Analysis on Manifolds, Graphs and
	Metric Spaces} (Paris\textup, 2002),
	Contemp. Math. {\bf 338}, pp. 173--218,
	Amer. Math. Soc., Providence, RI, 2003}
        
\bibitem{HaKo} \book{Haj\l asz, P. \AND Koskela, P.}
	{Sobolev met Poincar\'e}
	{{Mem. Amer. Math. Soc.} {\bf 145}:688 (2000)}


\bibitem{HM} \art{Haj\l asz, P. \AND Martio, O.}
         {Traces of Sobolev functions on fractal type sets and 
        characterization of extension domains}
        {J. Funct. Anal.}{143}{1997}{ 221--246}
        
\bibitem{HaKu} \artin{Hambly, B. \AND Kumagai, T.}
          {Heat kernel estimates for symmetric random walks on a 
        class of fractal graphs and stability under rough isometries}
        {\emph{Fractal Geometry and Applications\/\textup{:} 
            a Jubilee of Benoit Mandelbrot}, Part {\bf 2},  
       Proc. Sympos. Pure Math. {\bf 72}, pp. 233--259, 
       Amer. Math. Soc., Providence, RI, 2004}
          
\bibitem{Hei} \book{Heinonen, J.}{Lectures on Analysis on Metric Spaces}
         {Universitext, Springer, New York,  2001}
         
\bibitem{HKM} \art{Heinonen, J., Kilpel\"ainen, T. \AND Martio, O.}
        {Harmonic morphisms in nonlinear potential theory}
        {Nagoya Math. J.}{125}{1992}{115--140}
        
\bibitem{HeKiMa} \book{\auth{Heinonen}{J},
	\auth{Kilpel\"ainen}{T}
	\AND \auth{Martio}{O}}
        {Nonlinear Potential Theory of Degenerate Elliptic Equations}
        {2nd ed., Dover, Mineola, NY, 2006}

\bibitem{HeiKo} \art{Heinonen, J. \AND Koskela, P.}
          {Quasiconformal mappings in metric spaces with controlled geometry}
          {Acta Math.}{181}{1998}{1--61}
        
\bibitem{HPS} \art{Herman, P. E., Peirone, R. \AND Strichartz, R. S.}
       {\p-energy and \p-harmonic functions 
         on Sierpi\'nski gasket type fractals}
         {Potential Anal.}{20}{2004}{125--148}
         
\bibitem{HSX} \art{Herron, D., Shanmugalingam, N. \AND Xie, X.}
        {Uniformity from Gromov hyperbolicity}
        {Illinois J. Math.}{52}{2008}{1065--1109}

\bibitem{ivan} \book{Ivanov, L. D.}
        {Variations of Sets and Functions}
        {Nauka, Moscow, 1975 (Russian)} 
       
\bibitem{Jeff} \art{Jeffers, J.}{Lost theorems of geometry}
         {Amer. Math. Monthly}{107}{2000}{800--812}
               
\bibitem{JW80} \art{Jonsson, A. \AND Wallin, H.}
    {The trace to subsets of $\mathbb{R}^n$ of Besov spaces in the general case}
         {Anal. Math.}{6}{1980}{223--254}
      
\bibitem{JW84} \book{Jonsson, A. \AND Wallin, H.}
         {Function Spaces on Subsets of\/ $\mathbf{R}^n$}
         {Math. Rep. {\bf 2}:1, Harwood, London, 1984}


\bibitem{Kap} \book{Kapovich, M.}{Hyperbolic Manifolds and Discrete Groups}
         {Progress in Mathematics {\bf 183}, Birkh\"auser, Boston, 2001}
         
\bibitem{Ki} \art{Kigami, J.}
        {Dirichlet forms and associated heat kernels on the Cantor set induced 
      by random walks on trees}
          {Adv. Math.}{225}{2010}{2674--2730}
         
\bibitem{Kort} \art{Korte, R.}
        {Geometric implications of the Poincar\'e inequality}
        {Results Math.}{50}{2007}{93--107}
        
\bibitem{KYZ}  \art{Koskela, P., Yang, D. \AND Zhou, Y.}
        {Pointwise characterizations of Besov and Triebel--Lizorkin spaces 
        and quasiconformal mappings}
        {Adv. Math.} {226}{2011}{3579--3621} 

\bibitem{MazBook}
       \book{Maz{\cprime}ya, V. G.}
        {Sobolev Spaces with Applications to Elliptic Partial Differential Equations}
        {Springer, Heidelberg, 2011}
   
\bibitem{S11b} \artprep{Semmes, S.}
  {An introduction to the geometry of ultrametric spaces}
         {\emph{Preprint}, 2007, {\tt arXiv:0711.0709}}
  
\bibitem{S11a} \artprep{Semmes, S.}{Cellular structures, quasisymmetric mappings, 
     and spaces of homogeneous type}
         {\emph{Preprint}, 2007, {\tt arXiv:0711.1333}}
  
\bibitem{Str}\art{Strichartz, R. S.}
         {Some properties of Laplacians on fractals}
         {J. Funct. Anal.}{164}{1999}{181--208}

\bibitem{SW} \art{Strichartz, R. S. \AND Wong, C.}
         {The \p-Laplacian on the Sierpi\'nski gasket}
          {Nonlinearity}{17}{2004}{595--616}
 
\bibitem{Sh-rev} \art{Shanmugalingam, N.}
        {Newtonian spaces\textup{:} An extension of Sobolev spaces
        to metric measure spaces}
        {Rev. Mat. Iberoam.}{16}{2000}{243--279}

\bibitem{Sh-harm} \art{Shanmugalingam, N.} 
        {Harmonic functions on metric spaces}
        {Illinois J. Math.}{45}{2001}{1021--1050}

        
\bibitem{SX} \art{Shanmugalingam, N. \AND Xie, X.}
        {A rigidity property of some negatively curved solvable Lie groups}
        {Comment. Math. Helv.}{87}{2012}{805--823}

         
\bibitem{Tr} \book{Triebel, H.}{Theory of Function Spaces}
          {Monographs in Mathematics {\bf 78}, Birkh\"auser, Basel, 1983}

\end{thebibliography}
\end{document}